\documentclass[12pt,a4paper]{article}
\usepackage[margin=1in]{geometry}
\usepackage[latin1]{inputenc}
\usepackage{amsmath}
\usepackage{amsthm}
\usepackage{amsfonts}
\usepackage{amssymb}
\usepackage[makeroom]{cancel}
\usepackage{graphicx}
\usepackage{epstopdf}
\usepackage{bm}
\usepackage{fullpage}
\usepackage{color}
\usepackage{hyperref}
\usepackage[title,titletoc]{appendix}
\usepackage[noblocks]{authblk}
\usepackage{caption,subcaption}
\usepackage{float}
\usepackage{wrapfig}
\usepackage[mathlines]{lineno}
\usepackage{caption}
\usepackage{cite}
\usepackage{cleveref}
\usepackage{cases}
\usepackage{enumitem}
\usepackage{tikz}
\usepackage{ifthen}
\DeclareMathSizes{10}{10}{10}{10}
\graphicspath{ {Figures/} }

\usepackage{mathrsfs}

\hypersetup{%
    colorlinks=true,
        linkcolor=blue,
    filecolor=magenta,      
    urlcolor=violet,
    citecolor=blue,
}

\makeatletter
\def\input@path{{TikzPicz/}}
\makeatother

\title{Travelling waves in a PDE--ODE coupled system with nonlinear diffusion}
\author[1]{K. Mitra\footnote{email: \href{mailto:koondanibha.mitra@ru.nl}{koondanibha.mitra@ru.nl} }}
\author[2]{J.M. Hughes}
\author[3]{S. Sonner}
\author[4]{H.J. Eberl}
\author[5]{J.D. Dockery}
\affil[1,3]{Radboud University, The Netherlands}
\affil[2]{University of British Columbia, Canada}
\affil[4]{University of Guelph, Canada}
\affil[5]{Montana State University, USA}
\numberwithin{equation}{section}
\makeatletter
\DeclareMathSizes{\@xpt}{\@xpt}{6}{5}
\makeatother

\newcounter{def}
 
  \newcounter{EUXproperties}
 \stepcounter{EUXproperties}
 
\newtheorem{theorem}{Theorem}[section]
\newtheorem{lemma}[theorem]{Lemma}
\newtheorem{proposition}[theorem]{Proposition}
\newtheorem{corollary}{Corollary}[theorem]
\newtheorem{remark}{Remark}[section]
\newtheorem{definition}[def]{Definition}

\theoremstyle{definition}

\newtheorem{assumption}{Assumption}[section]
\newtheorem{algorithm}{Algorithm}[section]

\def \g  {\gamma}
\def \d  {\delta}
\def \D  {\Delta}
\def \e  {\varepsilon}

\def \vr  {\varrho}
\def \k  {\kappa}
\def \l  {\lambda}
\def \om {\omega}
\def \Om {\Omega}

\def \t  {\tau}

\def \z  {\zeta}

\def \Gf {\mathcal{G}}
\def \Mc {\mathfrak{M}}

\def \p  {\partial}
\def \N  {{\mathbb{N}}}
\def \R  {{\mathbb{R}}}
\def \sgn {{\rm sign}}
\def \dz {\tfrac{\dd}{\dd \xi}}
\def \dt {\tfrac{\dd}{\dd \t}}

\def \calR {\mathfrak{R}}
\def \calU {\mathcal{U}}
\def \dd {\mathrm{d}}
\newcommand{\co}[1]{\langle {#1} \rangle}
\newcommand{\ul}[1]{\underline{#1}}
\let\oldequation\equation
\let\oldendequation\endequation

\renewenvironment{equation}
  {\linenomathNonumbers\oldequation}
  {\oldendequation\endlinenomath}
  
  \let\oldalign\align
\let\oldendalign\endalign

\renewenvironment{align}
  {\linenomathNonumbers\oldalign}
  {\oldendalign\endlinenomath}

\begin{document}
\maketitle

\begin{abstract}
We analyze travelling wave (TW) solutions for nonlinear systems consisting of an ODE coupled to a degenerate PDE with a diffusion coefficient  
that vanishes as the solution tends to zero and blows up as it approaches its maximum value. 
Stable TW solutions for such systems have previously been observed numerically as well as in biological experiments on the growth of cellulolytic biofilms. 
In this work, we provide an analytical justification for these observations and prove existence and stability results for TW solutions of such models. 

Using the TW ansatz and a first integral, the system is reduced to an autonomous dynamical system with two unknowns. Analysing the system in the corresponding phase--plane, the existence of a unique TW is shown, which possesses a sharp front and a diffusive tail, and is moving with a constant speed. The linear stability of the TW in two space dimensions is proven under suitable assumptions on the initial data. Finally, numerical simulations are presented that affirm the theoretical predictions on the existence, stability and parametric dependence of the TWs. 
\end{abstract}

\textbf{Keywords:} travelling waves, stability, degenerate diffusion, PDE--ODE system, biofilm

\tableofcontents

\section{Introduction}
Motivated by mathematical models for biofilm growth, we investigate
travelling wave (TW) solutions for coupled PDE--ODE systems of the form
\begin{subequations}\label{eq:PDE}
\begin{align}
&\p_t M=\p_x [D(M)\,\p_x M] + \left (f(S)-\lambda\right )M,\\
& \p_t S= -\g\,f(S)\,M,
\end{align}
\end{subequations}
where $x\in\R$ and $t>0$ represent the space and time coordinates respectively.
The functions $M$ and $S$ are normalized, $S$ takes values in the interval $[0,\infty)$ and $M$ in $[0,1)$.
The growth and decay characteristics of the system are represented by the constants $\l,\,\g\in (0,1)$. 
The diffusion coefficient $D:[0,1)\to [0,\infty)$ has a singularity as well as a degeneracy. More specifically, it satisfies $D(m)\nearrow \infty$ as $m\nearrow 1$, and $D(0)=0$. The source function $f:[0,\infty)\to [0,1]$ is Lipschitz continuous and increasing. The exact assumptions on $D$ and $f$ are stated in \Cref{sec:Preliminiary}.

For $M$ we consider the following boundary conditions 
\begin{subequations}\label{eq:PDEC}
\begin{align}
M(\pm \infty,t)=0, \quad (D(M)\,\p_x M)(\pm \infty,t)=0,
\end{align}
and initial conditions that are consistent with the above. In particular, we assume that
\begin{align}
\begin{cases}
M(x,0)\in [0,1), \quad S(x,0)\in [0,\infty) &\text{ for all } x< 0,\\
M(x,0)=0, \quad  S(x,0)=1 &\text{ for all } x\geq 0.
\end{cases}
\end{align}
\end{subequations}

Standard approaches to prove the existence and stability of TW solutions  
for equations such as the Fischer-KPP equation do not generalize to the system \eqref{eq:PDE}.
Difficulties arise through the degeneracy and singularity of the diffusion coefficient, and the nonlinear coupling between the PDE and the ODE, which leads to a  non-monotone profile for $M$. In this paper, we use ordering of orbits in the $(M,S)$ phase--plane to prove the existence and uniqueness of a TW solution for the system \eqref{eq:PDE} satisfying suitable boundary conditions. We further derive stability results using asymptotic expansions. Moreover, we present numerical simulations that affirm the theoretical predictions on the existence, stability and parametric dependence of the TWs. 

Several examples of semilinear evolution equations coupled to an ODE 
through the source--term can be found in 
 \cite{logan2001transport} as models of reactive  transport through the subsurface.  In mathematical biology, examples of similar systems are found in \cite[Chapter 13]{murray2002mathematicalbiology}, as well as a discussion on their TW solutions.
PDE--ODE coupled systems of various type with nonlinear diffusion coefficients are used to model variety of other physical  or biological processes, ranging from hysteretic flow through porous media \cite{mitra2021existence}, to
 tumor growth \cite{harley2014existence}, and wound healing \cite{harley2014novel}.
The particular motivation for our analysis is the model for cellulolytic biofilm growth in \cite{eberl2017spatially} for which TW solutions have been observed numerically. 
Cellulolytic biofilms play an important role in the production of cellulosic ethanol, a renewable biofuel that can be implemented in the existing transportation infrastructure.
In contrast to more traditional biofilms that form on mostly abiotic surfaces and develop colonies that grow into the surrounding aqueous phase, 
many cellulolytic biofilms degrade and consume the biological material that supports them and form crater like depressions, a phenomenon known as inverted colony formation. Since the nutrients are immobile, whereas, the biofilm expands spatially, a PDE--ODE model was proposed for cellulolytic biofilms in \cite{eberl2017spatially} which is a special case of \eqref{eq:PDE}. In the model, $M$ represented  the biomass density and $S$ the nutrient concentration. 
The bacteria consume nutrients and degrade the biological material which results in the production of biomass. This was described in \cite{eberl2017spatially} using the Monod reaction function $f$, whereas, the spatial spreading of biomass was modelled by a density-dependent diffusion coefficient $D$, with the corresponding expressions being
\begin{subequations}\label{eq:DefD}
\begin{align}
D(m)&=\frac{\d\,m^a}{(1-m)^b}, && \d>0, \ a,\,b>1,\label{eq:DefDa}\\[.5em]
f(s)&= \frac{s}{\k+s}, && 0<\k\ll 1.\label{eq:DefDb}
\end{align}
\end{subequations}

Cellulolytic biofilm formation has been studied both numerically and experimentally. In \cite{wang2011spatial}, an agent based stochastic discrete cellular automaton model was used to study the system. The model in \cite{eberl2017spatially}, on the other hand, is the deterministic continuum model \eqref{eq:PDE} with 
the specific functions $f$ and $D$ in \eqref{eq:DefD}. Extensions of this model that account for attachment of cells from the aqueous phase to the biofilm were presented in \cite{rohanizadegan2020SPDE,hughes2022RODE} using either It\^o stochastic differential equations or random differential equations.
Numerical simulations of the original deterministic model in \cite{eberl2017spatially} suggest the existence of TW solutions, which describe a constant rate of degradation of the cellulosic material that is utilised by the bacteria. Additional numerical evidence for this is given in \cite{hughes2022RODE} where a different time integration method is used, along with an independent implementation. Celluloytic biofilm systems are very difficult to observe experimentally with time-lapse microscopy techniques. Nevertheless, in \cite{wang2011spatial} experiments were reported that suggest degradation of paper chads by cellulosic biofilms at a constant speed, which gives indirect evidence for TW like degradation in the biological system.
Furthermore, experimental observations in \cite{wang2011spatial}
and in \cite{dumitrache2015mathematical} indicate that the width of the microbially 
active band remains constant as the wave of microbial crater formation 
spreads.
A rigorous proof of the existence and features of TW solutions of the model in \cite{eberl2017spatially}, or an answer to the question `under which conditions on parameters such TWs can be found', have so far been open problems. In this study we provide the answers. For this purpose, the initial and boundary conditions \eqref{eq:PDEC} are chosen to be consistent with the physical setting of the numerical experiments in \cite{eberl2017spatially}.  We emphasize that, although our study is motivated and prompted by the cellulosic biofilm system, our results are valid for a significantly wider class of problems.

The existence of TWs for the (scalar) porous medium equation with a nonlinear source term was investigated in \cite{de1991travelling}. It was shown that under certain conditions on the coefficients, there exists a minimum speed for which TW solutions exist. The results were extended to include nonlinear advection terms in \cite{Pablo2000}. Furthermore, in \cite{biro2002} the stability of such TW solutions was shown in one space dimension. For the porous medium equation with Fischer type reaction term, 
the existence of TWs was shown in \cite{sanchez1996shooting}, again for wave-speeds larger than a minimum value. In \cite{de1998travelling}, further qualitative properties of the TW solutions were analysed. These results were generalized in \cite{efendiev2009classification} for equations with the biofilm diffusion coefficient $D$ in \eqref{eq:DefDa}. However, the aforementioned results are limited to scalar equations with Fisher-type nonlinear reaction term, and thus, exclude the complex interplay between the ODE and the PDE solutions. While
the TW profile is a monotone function with respect to $x$ in all the mentioned results, this is not the case for the TW profile of $M$ in our system. 

TWs for PDE--ODE coupled systems have been studied for  multiphase flow through porous media in \cite{VANDUIJN2018232,mitra2018wetting,mitra2019fronts}, where non-monotone profiles of $M$ have been observed. The ordering of orbits in the phase--plane is also used in these papers to predict the behaviour of the TWs. The existence of TWs in two-dimensions for a PDE--ODE model was investigated in \cite{mitra2020travelling} in the context of hysteretic flow through porous media, and non-planar TWs were shown to exist. However, TWs in these cases originate from the advection term, rather than the source term.  Non-monotone profiles have also been observed for the TWs of PDE--ODE systems arising in biology, see \cite{harley2014existence,harley2014novel} for examples. TWs for a PDE--PDE coupled model of bacteria spreading in an aqueous phase were analysed in \cite{satnoianu2001travelling}. Nevertheless, these systems differ fundamentally in their structure from \eqref{eq:PDE}.  In our setting, as will be observed later, the TWs have 
distinctive features which distinguish them from the examples above. They inherit a sharp front and a minimum speed of propagation like  TWs of the porous medium equation. However, due to the coupling with the ODE, they exhibit a non-monotone profile with a diffusive tail.

The outline of our paper is as follows:
In Section \ref{sec:Preliminiary}, we state the assumptions on the associated functions, and using the TW ansatz, the system \eqref{eq:PDE} is reduced to a dynamical system with two unknowns. The existence result for the TWs is also stated, see \Cref{theo:main}. In \Cref{sec:TW existence}, using phase--plane analysis we develop the auxiliary results which are then used to prove \Cref{theo:main}. \Cref{sec:stability} is dedicated to proving a linear stability result for the TWs in two space dimensions using asymptotic expansion. In \Cref{sec:num}, numerical results are presented for a discretization of the full PDE--ODE system, an ODE approach inspired by the TW analysis, and a numerical continuation approach. All three different approaches concur about the existence/non-existence of TWs in a parametric regime indicated by our theory. Furthermore, numerical results showing the influence of the parameters, and the stability of the TWs are presented. In \Cref{sec:interpretation}, we interpret the analytical and numerical results in the context of cellulolytic biofilms, and discuss possible generalizations and future applications.

\section{Preliminaries and main result}\label{sec:Preliminiary}
We investigate the existence of solutions for the system \eqref{eq:PDE} with the boundary and initial conditions \eqref{eq:PDEC}. To this end, we assume that $D$ and $f$ have the following properties:

  \begin{enumerate}[label=(P\theEUXproperties)]
 \setlength\itemsep{-0.2em}
\item The diffusion coefficient $D:[0,1)\to [0,\infty)$ is an increasing function in $C^1([0,1))$ which satisfies for constants $a,\,b>1$
$$
\lim\limits_{m\searrow 0} \frac{D(m)}{m^a}\in \R^+,\;\;\text{ and }\;\; \lim\limits_{m\nearrow 1} (1-m)^b\,D(m)\in \R^+.
$$
\label{prop:D}

\stepcounter{EUXproperties}
\item The source function $f: [0,1]\to [0,1]$ is an increasing function in $C^1([0,1])$ which satisfies for a constant $\k\in (0,1]$,
$$
f(0)=0,\quad f'(0)=\frac{1}{\k},\;\;\text{ and }\;\; f(1)\in (\l,1].
$$
\label{prop:f}
\end{enumerate}

The TW ansatz for the subsequent theory is:
\begin{assumption}[TW ansatz]\label{ass:TW}
For a wave-speed $v>0$, and the travelling wave coordinate $\xi=x-v t$, there exist $M,S:\R\times\R^+\to [0,1]$ which satisfy \eqref{eq:PDE}--\eqref{eq:PDEC} and
$$
M(x,t)=M(\xi), \quad  S(x,t)=S(\xi).
$$
\end{assumption}
With this ansatz, the system \eqref{eq:PDE} is written  as
\begin{subequations}\label{eq:ODE1}
\begin{align}
-&v \dz M=\dz \left [D(M)\,\dz M \right ] + \left (f(S)-\lambda\right )M,\label{eq:ODE1m}\\
& v\dz S= \g\, f(S)\, M.\label{eq:ODE1s}
\end{align}
\end{subequations}
The initial and boundary conditions \eqref{eq:PDEC} are transformed into
\begin{subequations}\label{eq:ODEBC}
\begin{align}
M(-\infty)&= (D(M)\,\dz M)(- \infty)=0,\\
M(\xi)&=(D(M)\,\dz M)(\xi)=0, \quad S(\xi)=1 \text{ for all } \xi\geq 0.
\end{align}
\end{subequations}

\subsection{Auxiliary quantities and reductions}
For a given solution $(M,S)$ of \eqref{eq:ODE1}--\eqref{eq:ODEBC}, the accumulated biomass $\om:\R\to \R^+$ until $\xi\in \R$, is defined as
\begin{align}\label{eq:Defom}
\om(\xi):=\int_{-\infty}^\xi M,\quad \text{ implying }\quad  \dz \om=M \text{ and } \om(-\infty)=0.
\end{align}
Assuming $M$, $S$ and $\om$ are smooth enough, \eqref{eq:ODE1m} is then rewritten using \eqref{eq:ODE1s} and \eqref{eq:Defom} as
$$
-v \dz M=\dz \left [D(M)\,\dz M \right ] + \dz\left (\tfrac{v}{\g}S -\l \om \right ).
$$
Integrating the above equation from $-\infty$ to $\xi$, we have using \eqref{eq:ODEBC} that
$$
-v M = D(M)\,\dz M  + \tfrac{v}{\g}(S-S(-\infty)) -\l \om.
$$
Observe that $S$ is a non-decreasing function by \eqref{eq:ODE1s} with $S(0)=1$, and $S(\xi)\geq 0$ for all $\xi\in \R$. Therefore,  $S(-\infty)\geq 0$ is well-defined. Upon rearranging the above equation  one has
\begin{align}
D(M)\,\dz M =\l \om -v \,(M + \tfrac{1}{\g}(S-S(-\infty))).
\end{align}
Passing $\xi\to \infty$ and using \eqref{eq:ODEBC} one further has
\begin{align}\label{eq:ominf}
\om(+\infty)=\tfrac{v}{\l\,\g}(1-S(-\infty)),
\end{align}
which serves as a kind of Rankine-Hugoniot condition for the wave-speed $v$.
Finally, using the relations above, \eqref{eq:ODE1} is rewritten as an autonomous dynamical system for $M$, $S$ and $\om$, 
\begin{subequations}\label{eq:ODE2}
\begin{align}
&\dz M= \tfrac{1}{D(M)}[\l \om -v \,(M + \tfrac{1}{\g}(S-S(-\infty)))],
\label{eq:ODE2m}\\
& \dz S= \tfrac{\g}{v} \, f(S)\,M,\label{eq:ODE2s}\\
& \dz \om=M\label{eq:ODE2om}.
\end{align}
\end{subequations}
Observing that $S(\xi)\geq S(-\infty)\geq 0$ for all $\xi\in \R$, equation \eqref{eq:ODE2s} is rewritten using \eqref{eq:ODE2om} as
\begin{align}\label{eq:dSdOmega}
\dz S=\tfrac{\g}{v}f(S)\dz \om \quad \text{ or }\quad  \tfrac{v}{\g f(S)} \dz S=\dz \om.
\end{align}
We introduce the function $F:(0,1]\to \R^+$ as
\begin{align}\label{eq:DefF}
F(s):=\int^1_s \frac{\dd \vr}{f(\vr)}.\quad  \text{ It follows from \ref{prop:f} that }\;  F'(s)<0,\;\lim\limits_{s\searrow 0}F(s)=\infty,\; F(1)=0.
\end{align}
The limit $F(s)\to \infty$ for $s\searrow 0$ follows from \ref{prop:f} since $f(s)\sim  s\slash \k$ in a right neighbourhood of $s=0$, and consequently $F(s)= \int_s \frac{1}{f}\sim -\k\log(s)\to \infty$ as $s\searrow 0$.  Integrating \eqref{eq:dSdOmega} from $\xi\leq 0$ to $+\infty$ and using \eqref{eq:ODEBC} one has $\tfrac{v}{\g} F(S)= \om(+\infty)-\om$, which upon rearranging and using  \eqref{eq:ominf} gives
\begin{align}\label{eq:omeq}
\om=\tfrac{v}{\g}\left [\tfrac{1}{\l}(1-S(-\infty))- F(S)\right ]. 
\end{align}
Passing $\xi\to -\infty$ in the above equation, using $\om(-\infty)=0$ from \eqref{eq:Defom} and cancelling equal terms, we get
$$
S(-\infty)+\l F(S(-\infty))=1.
$$
Since $F$ is strictly decreasing and convex as evident from \ref{prop:f}, there can at most be two solutions of the equation $g(s)=s+ \l F(s)=1$ in $(0,1]$. One trivial solution is $s=1$. Since $g(0)=\infty$ from \eqref{eq:DefF}, the existence of the second solution  is guaranteed if $g'(1)>0$, or  $F'(1)=-1\slash f(1)>-1\slash \l$ which holds due to \ref{prop:f}. Hence, we define $s_{-\infty} \in (0,1)$ as the solution of 
\begin{align}\label{eq:Sinf}
g(s_{-\infty})=s_{-\infty}+\l F(s_{-\infty})=1.
\end{align} 

\begin{remark}[The value of $s_{-\infty}\in (0,1)$] Let $f$ be given by the expression in \eqref{eq:DefD}. Then, $F(s)=1-s-\k \log(s)$. Thus, for $\k\ll 1$ and $\l\sim 1$, one has
$$
s_{-\infty}\approx \exp\left (-\tfrac{1-\l}{\k\l}\right ).
$$
For the parameters $\l=0.42$, $\g=0.4$ and $\k=0.01$ used in \cite{eberl2017spatially}, we estimate that
$$
s_{-\infty}\approx 10^{-60},
$$ 
which is negligible for all practical purposes. This suggests that in this parameter regime, the substrate is fully depleted after the TW has passed. In \Cref{sec:NumCont}, we provide an example of parameters for which $s_{-\infty}\approx 0.11$, and hence, there remains a significant level of residual  substrates.
\end{remark}
Finally, substituting \eqref{eq:omeq} into  \eqref{eq:ODE2} we get the reduced autonomous dynamical system with only two unknowns $M$ and $S$, i.e., 
\begin{subequations}\label{eq:DS}
\begin{align}
&\dz M= \tfrac{v}{\g\, D(M)}[(1-S-\l F(S)) -\g M],
\label{eq:DSm}\\[.5em]
& \dz S= \tfrac{\g}{v}\, f(S)\,M\label{eq:DSs}.
\end{align}
\end{subequations}
The revised boundary conditions for this system are
\begin{align}\label{eq:BC}
M(-\infty)=0,\; S(-\infty)=s_{-\infty} \text{ and }  M(\xi)=0,\; S(\xi)=1 \text{ for all }\xi\geq 0.
\end{align}
This will be the main system analysed in this paper.

\begin{remark}[The flux conditions at $\xi=0$ and $\xi=-\infty$]
Observe that any solution $(M,S)$ of the system \eqref{eq:DS}--\eqref{eq:BC}, automatically satisfies the boundary condition for the flux, i.e., $D(M)\dz M=0$ at $\xi=0$ and $\xi=-\infty$. 
\end{remark}

\subsection{Main theorem}\label{sec:mainTheo}
For the rest of this study, we focus on the following parametric regime: for $F$ defined in \eqref{eq:DefF}, and $g=\mathbb{I}+\l F$, let $\l,\g\in (0,1)$ be such that
\begin{align}\label{eq:parameter}
g(f^{-1}(\l))= f^{-1}(\l) + \l F(f^{-1}(\l))< 1-\g.
\end{align}
Observe that, the function $g(y)$ takes its minimum value $g_{\min}$ at $y=f^{-1}(\lambda)$.  The condition \ref{prop:f} then guarantees that $f^{-1}(\l)\in (0,1)$.  The existence of $s_{-\infty}\in (0,1)$ in \eqref{eq:Sinf} proves that $g_{\min}<1$. Assumption \eqref{eq:parameter} enforces a stronger condition, i.e., that $g_{\min}<1-\g$. 

For the function $f(s)=s\slash (\k + s)$ the condition can be stated in a more compact form as
\begin{align}\label{eq:parameter_2}
0<\g + \l + \k\l\,(1-\log(\k\l))\leq 1.
\end{align}
We introduce the following important integral:
\begin{subequations}\label{eq:DefG}
\begin{align}
\Gf(s)&:=\int_s^1 (\vr+\l F(\vr)-(1-\g))\, \frac{\dd\vr}{f(\vr)}\label{eq:DefGa}\\
&\overset{\eqref{eq:DefF}}=\int_s^1\frac{\vr}{f(\vr)}\,\dd \vr +\frac{\l}{2}\, F^2(s)-(1-\g)\,F(s).\label{eq:DefGb}
\end{align}
\end{subequations}
This representation follows using $F'=1\slash f$ and $F(1)=0$ from \eqref{eq:DefF}.
The shape of the $\Gf$-integral is shown in \Cref{fig:G}. Observe that, $\Gf(1)=0$ and $\Gf(s)$ is a decreasing function for $s>s^*$, where $s=s^*$ solves $s+ \l F(s)=1-\g$. Hence, $\Gf> 0$ in an interval $(s_g,1)$ where $s_g\in (0, s^*)$. Also from \eqref{eq:DefGb}, $\Gf(s)\to +\infty$ as $s\searrow 0$ since $F(s)\to \infty$ in this case, see \eqref{eq:DefF}. Hence, depending on the parameter values, $\Gf$ might or might not have a negative part. This has a profound effect on the existence of TWs as stated in our main theorem below.

\begin{figure}
\centering
\begin{tikzpicture}
[xscale=6,yscale=30,domain=0:1,samples=100]
\draw[ultra thick,->] (0,0)--(1.1,0) node[scale=1.2,above] {$s$};
\draw[ultra thick,->] (0,-.02)--(0,.15) node[scale=1.2,left] {$\Gf$};
\node[scale=.8, below left] at (0, 0) {$0$};

\node[scale=1, above right] at (0.5,0.13) {$\l_1=\g_1=.4$};

\node[scale=1, above right] at (0.605,0.066) {$\l_2=\g_2=.3$};

\draw[dashed] (0.535,0.07059)--(.535,0) node[below, scale=1] {$s^*_{2}$};

\draw[dashed] (0.295,0.14635)--(0.295,0) node[below, scale=1] {$s^*_{1}$};

\draw[ultra thick,dashdotted,red] (1,0)--(0.985,0.0041264)--(0.975,0.0081939)--(0.965,0.012202)--(0.955,0.016152)--(0.945,0.020043)--(0.935,0.023875)--(0.925,0.027648)--(0.915,0.031362)--(0.905,0.035017)--(0.895,0.038614)--(0.885,0.042151)--(0.875,0.04563)--(0.865,0.04905)--(0.855,0.052411)--(0.845,0.055713)--(0.835,0.058956)--(0.825,0.06214)--(0.815,0.065266)--(0.805,0.068332)--(0.795,0.07134)--(0.785,0.074289)--(0.775,0.077179)--(0.765,0.08001)--(0.755,0.082782)--(0.745,0.085495)--(0.735,0.08815)--(0.725,0.090745)--(0.715,0.093282)--(0.705,0.09576)--(0.695,0.098179)--(0.685,0.10054)--(0.675,0.10284)--(0.665,0.10508)--(0.655,0.10726)--(0.645,0.10939)--(0.635,0.11145)--(0.625,0.11346)--(0.615,0.11541)--(0.605,0.1173)--(0.595,0.11913)--(0.585,0.1209)--(0.575,0.12261)--(0.565,0.12426)--(0.555,0.12585)--(0.545,0.12739)--(0.535,0.12886)--(0.525,0.13028)--(0.515,0.13163)--(0.505,0.13293)--(0.495,0.13417)--(0.485,0.13534)--(0.475,0.13646)--(0.465,0.13752)--(0.455,0.13852)--(0.445,0.13946)--(0.435,0.14034)--(0.425,0.14116)--(0.415,0.14193)--(0.405,0.14263)--(0.395,0.14327)--(0.385,0.14385)--(0.375,0.14437)--(0.365,0.14483)--(0.355,0.14523)--(0.345,0.14557)--(0.335,0.14585)--(0.325,0.14606)--(0.315,0.14622)--(0.305,0.14631)--(0.295,0.14635)--(0.285,0.14632)--(0.275,0.14622)--(0.265,0.14607)--(0.255,0.14585)--(0.245,0.14556)--(0.235,0.14522)--(0.225,0.1448)--(0.215,0.14432)--(0.205,0.14378)--(0.195,0.14316)--(0.185,0.14248)--(0.175,0.14172)--(0.165,0.14089)--(0.155,0.13999)--(0.145,0.13901)--(0.135,0.13795)--(0.125,0.1368)--(0.115,0.13557)--(0.105,0.13424)--(0.095,0.1329)--(0.085,0.13215)--(0.075,0.13206)--(0.065,0.13262)--(0.055,0.1338)--(0.045,0.13552)--(0.035,0.13771)--(0.025,0.14015)--(0.015,0.14231);

\draw[ultra thick,blue]  (1,0)--(0.985,0.0031324)--(0.975,0.0061943)--(0.965,0.0091858)--(0.955,0.012107)--(0.945,0.014957)--(0.935,0.017737)--(0.925,0.020447)--(0.915,0.023085)--(0.905,0.025653)--(0.895,0.028151)--(0.885,0.030578)--(0.875,0.032934)--(0.865,0.035219)--(0.855,0.037434)--(0.845,0.039577)--(0.835,0.04165)--(0.825,0.043652)--(0.815,0.045583)--(0.805,0.047443)--(0.795,0.049232)--(0.785,0.050949)--(0.775,0.052596)--(0.765,0.054172)--(0.755,0.055676)--(0.745,0.057109)--(0.735,0.05847)--(0.725,0.05976)--(0.715,0.060979)--(0.705,0.062126)--(0.695,0.063201)--(0.685,0.064205)--(0.675,0.065137)--(0.665,0.065997)--(0.655,0.066785)--(0.645,0.067501)--(0.635,0.068145)--(0.625,0.068716)--(0.615,0.069215)--(0.605,0.069642)--(0.595,0.069996)--(0.585,0.070278)--(0.575,0.070487)--(0.565,0.070622)--(0.555,0.070685)--(0.545,0.070674)--(0.535,0.07059)--(0.525,0.070432)--(0.515,0.0702)--(0.505,0.069895)--(0.495,0.069515)--(0.485,0.069061)--(0.475,0.068532)--(0.465,0.067928)--(0.455,0.067249)--(0.445,0.066494)--(0.435,0.065664)--(0.425,0.064757)--(0.415,0.063774)--(0.405,0.062715)--(0.395,0.061577)--(0.385,0.060363)--(0.375,0.05907)--(0.365,0.057698)--(0.355,0.056247)--(0.345,0.054715)--(0.335,0.053104)--(0.325,0.051411)--(0.315,0.049635)--(0.305,0.047777)--(0.295,0.045835)--(0.285,0.043808)--(0.275,0.041694)--(0.265,0.039493)--(0.255,0.037202)--(0.245,0.034821)--(0.235,0.032347)--(0.225,0.029778)--(0.215,0.027113)--(0.205,0.024347)--(0.195,0.021478)--(0.185,0.018503)--(0.175,0.015416)--(0.165,0.012214)--(0.155,0.0088906)--(0.145,0.0054383)--(0.135,0.0018489)--(0.125,-0.0018877)--(0.115,-0.0057842)--(0.105,-0.0098566)--(0.095,-0.013875)--(0.085,-0.016367)--(0.075,-0.017117)--(0.065,-0.016177)--(0.055,-0.013617)--(0.045,-0.0095467)--(0.035,-0.0041461)--(0.025,0.01)--(0.015,0.03)--(0.005,0.05);
\end{tikzpicture}
\caption{The plot of $\Gf$ as defined in \eqref{eq:DefG} for $f(s)=s\slash (\k+s)$, $\l_1=\g_1=0.4$ and $\l_2=\g_2=0.3$. The points $s=s^{*}\in (0,1)$ marked are the solutions of $s+ \l F(s)=1-\g$.}\label{fig:G}
\end{figure}
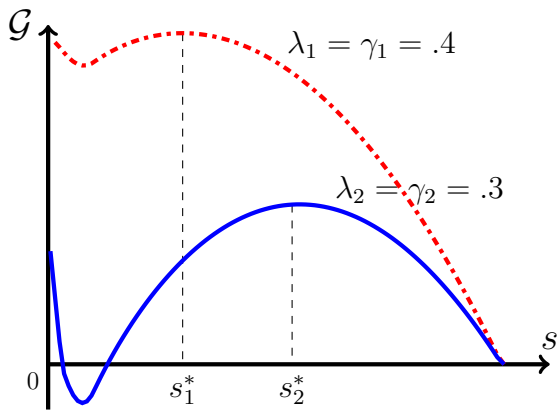

\begin{theorem}[Existence of the TW solution]
Assume \ref{prop:D}--\ref{prop:f}. Let \eqref{eq:parameter} be satisfied and let $\Gf(s)>0$ for all $s\in (s_{-\infty},1)$. Then there exists a unique $v>0$ such that a travelling wave solution $(M,S):\R\to [0,1]^2$ with $D(M)\dz M\in C(\R)$ and $S\in C^1(\R)$ exists satisfying \eqref{eq:DS}--\eqref{eq:BC}.\label{theo:main}
\end{theorem}

\begin{remark}[Conditions on existence]
The condition \eqref{eq:parameter}, used in \Cref{theo:main}, provides upper bounds for  $\l$ and $\g$, whereas, the condition $\Gf(s)>0$ for $s\in (s_{-\infty},1)$ provides lower bounds for $\l$ and $\g$ for the existence of the TW solutions. We show in \Cref{pros:NonExistence} that the latter condition is also a necessary condition.
\end{remark}

\section{The existence of travelling waves} \label{sec:TW existence}
In this section we prove \Cref{theo:main} by analysing the dynamical system \eqref{eq:DS}.
\subsection{The phase--plane}
For a given orbit $\xi\mapsto (M,S)$ satisfying \eqref{eq:DS}, the scaled TW coordinate $\t$ is defined by the coordinate transform
\begin{align}\label{eq:DefTau}
\t(\xi):=\int_{0}^\xi \tfrac{\dd\vr}{D(M(\vr))},\quad  \text{ implying } \dz \t= \tfrac{1}{D(M(\xi))}.
\end{align}
Moreover, to shorten notation, we introduce
\begin{align}\label{eq:DefEll}
\ell(s;\l,\g):=\frac{1}{\g}[(1-s)-\l F(s)],
\end{align}
which from \ref{prop:f} and \eqref{eq:DefF} has the properties
\begin{subequations}\label{eq:propEll}
\begin{align}\label{eq:propElla}
&\ell\in C^1((0,1]),\; \ell'(s)>0 \text{ for } s<s_{\mathrm{M}}:=f^{-1}(\l), \text{ and } \ell'(s)<0 \text{ for } s>s_{\mathrm{M}};\\
&\ell(1)=0,\text{ and } \lim\limits_{s\searrow 0}\ell(s)=-\infty.\label{eq:propEllb}
\end{align}
\end{subequations}
Observe that, in terms of the function $g$ introduced in \Cref{sec:Preliminiary}, $\ell(s)=\g^{-1}[1-g(s)]$ (recall that $g=\mathbb{I}+\l\,F$). In the following sections, we will only use the properties \eqref{eq:propEll}  of $\ell$, and will not further use $F$ or $g$, to keep the notation as clear as possible. From \eqref{eq:DefGa}, we additionally have that
\begin{align}\label{eq:relationGell}
\Gf(s)=\g\int_s^1 \left (\frac{1-\ell(\vr)}{f(\vr)}\right )\dd\vr.
\end{align}

With the coordinate transform \eqref{eq:DefTau} and the definition \eqref{eq:DefEll}, system \eqref{eq:DS} is re-written as
\begin{subequations}\label{eq:GE}
\begin{align}
&\dt M= v\,[\ell(S)- M],
\label{eq:GEm}\\[.5em]
& \dt S= \tfrac{\g}{v}\,f(S)\, M D(M)\label{eq:GEs}.
\end{align}
\end{subequations}

\begin{figure}[h!]
\begin{subfigure}{.5\textwidth}
\begin{tikzpicture}
[xscale=6,yscale=4.3,domain=0:1,samples=100]
\draw[ultra thick,->] (0,0)--(1.1,0) node[scale=1.2,above] {$m$};
\draw[ultra thick,->] (0,0)--(0,1) node[scale=1.2,above right] {$s$};
\node[scale=.8, below left] at (0, 0) {$0$};

\draw[dashed] (1,1)--(0,1) node[scale=.8,left] {$1$};
\draw[dashed] (1,1)--(1,0) node[scale=.8,below] {$1$};

\draw[dotted] (1,0.13)--(0,0.13) node[scale=.9,above left] {$s^*$};
\draw[dotted] (1,0.05)--(0,0.05) node[scale=.9,left] {$s_*$};

\draw[thick,dashdotted,red] (0.85525,0.01)--(0.96467,0.03)--(1,0.05)--(1.0134,0.07)--(1.0161,0.09)--(1.0124,0.11)--(1.0044,0.13)--(0.99346,0.15)--(0.98023,0.17)--(0.96525,0.19)--(0.94886,0.21)--(0.93132,0.23)--(0.91283,0.25)--(0.89352,0.27)--(0.87353,0.29)--(0.85293,0.31)--(0.83181,0.33)--(0.81022,0.35)--(0.78822,0.37)--(0.76586,0.39)--(0.74316,0.41)--(0.72016,0.43)--(0.69689,0.45)--(0.67337,0.47)--(0.64962,0.49)--(0.62566,0.51)--(0.60151,0.53)--(0.57717,0.55)--(0.55267,0.57)--(0.52802,0.59)--(0.50322,0.61)--(0.47828,0.63)--(0.45322,0.65)--(0.42804,0.67)--(0.40275,0.69)--(0.37735,0.71)--(0.35185,0.73)--(0.32625,0.75)--(0.30057,0.77)--(0.2748,0.79)--(0.24895,0.81)--(0.22302,0.83)--(0.19702,0.85)--(0.17095,0.87)--(0.14482,0.89)--(0.11862,0.91)--(0.092356,0.93)--(0.066037,0.95)--(0.039662,0.97)--(0.013234,0.99);

\draw[thick,dashdotted,red] (0.28935,0.0001)--(0.42749,0.0003)--(0.49156,0.0005)--(0.53367,0.0007)--(0.56504,0.0009)--(0.59004,0.0011)--(0.6108,0.0013)--(0.62854,0.0015)--(0.64402,0.0017)--(0.65774,0.0019)--(0.67006,0.0021)--(0.68123,0.0023)--(0.69145,0.0025)--(0.70086,0.0027)--(0.70957,0.0029)--(0.71768,0.0031)--(0.72527,0.0033)--(0.7324,0.0035)--(0.73911,0.0037)--(0.74545,0.0039)--(0.75146,0.0041)--(0.75717,0.0043)--(0.76261,0.0045)--(0.7678,0.0047)--(0.77276,0.0049)--(0.77751,0.0051)--(0.78207,0.0053)--(0.78645,0.0055)--(0.79066,0.0057)--(0.79471,0.0059)--(0.79862,0.0061)--(0.8024,0.0063)--(0.80604,0.0065)--(0.80957,0.0067)--(0.81299,0.0069)--(0.8163,0.0071)--(0.81951,0.0073)--(0.82263,0.0075)--(0.82565,0.0077)--(0.82859,0.0079)--(0.83145,0.0081)--(0.83424,0.0083)--(0.83695,0.0085)--(0.83959,0.0087)--(0.84216,0.0089)--(0.84467,0.0091)--(0.84712,0.0093)--(0.84951,0.0095)--(0.85185,0.0097)--(0.85413,0.0099);

\node[scale=.8, above,red] at (0.3770, 0.5) {$ m
=\ell(s)$};

\draw[thick,->,teal] (0.3770,0.2228)--(0.4949,0.2947);
\draw[thick,->,teal] (0.89,0.0049)--(0.89,0.1);
 \draw[thick,->,teal] (0,0.3)--(0.15,0.3);
 \draw[thick,->,teal] (0,0.5418)--(0.0968,0.5418);
  \draw[thick,->,teal]  (0,0.0982)--(0.1023,0.0982);
  \draw[thick,->,teal]   (0.8396,0.2461)--(0.8396,0.3706);
   \draw[thick,->,teal]  (0.7161,0.3745)--(0.7180,0.4757);
   \draw[thick,->,teal]  (0,0.8006)--(0.0525, 0.8);
    \draw[thick,->,teal]      (0.7585,0.7247)--(0.6719,0.8298);
    \draw[thick,->,teal]         (0.0986,0.87)--(0.0968,0.95);
    \draw[thick,->,teal]   (0.2848,0.964)--(0.2037,0.9912);
    \draw[thick,->,teal] (.5797,0.9232)--(0.4710,0.9971);
    \draw[thick,->,teal]    (.8544,0.8804)--(.7806,.9854);
        \draw[thick,->,teal]     (0.6498,0.1411)--(0.6903,0.2228);
                \draw[thick,->,teal]         (0.8710,0.5924)--(0.8488,0.6955);
                \draw[thick,->,teal]        
(0.4372,0.7597)--(0.3561,0.7772);

\draw[ultra thick, blue,->] (0,0)--(0.01,6.1039e-05)--(0.02,0.0002442)--(0.03,0.00054962)--(0.04,0.00097752)--(0.05,0.0015282)--(0.06,0.0022021)--(0.07,0.0029997)--(0.08,0.0039216)--(0.09,0.0049685)--(0.1,0.0061412)--(0.11,0.0074406)--(0.12,0.0088677)--(0.13,0.010424)--(0.14,0.01211)--(0.15,0.013927)--(0.16,0.015877)--(0.17,0.017962)--(0.18,0.020183)--(0.19,0.022542)--(0.2,0.025041)--(0.21,0.027683)--(0.22,0.030469)--(0.23,0.033403)--(0.24,0.036488)--(0.25,0.039725)--(0.26,0.043119)--(0.27,0.046673)--(0.28,0.050391)--(0.29,0.054277)--(0.3,0.058335)--(0.31,0.06257)--(0.32,0.066987)--(0.33,0.071593)--(0.34,0.076392)--(0.35,0.081393)--(0.36,0.086601)--(0.37,0.092026)--(0.38,0.097675)--(0.39,0.10356)--(0.4,0.10969);

\draw[ultra thick, blue,->] 
(0.4,0.10969)--(0.41,0.11607)--(0.42,0.12273)--(0.43,0.12967)--(0.44,0.13691)--(0.45,0.14447)--(0.46,0.15237)--(0.47,0.16063)--(0.48,0.16928)--(0.49,0.17836)--(0.5,0.18789)--(0.51,0.19793)--(0.52,0.20852)--(0.53,0.21973)--(0.54,0.23163)--(0.55,0.24433)--(0.56,0.25794)--(0.57,0.27263)--(0.58,0.28863)--(0.59,0.30626)--(0.6,0.32601)--(0.61,0.34871)--(0.62,0.37598)--(0.63,0.41196)--(0.64,0.5)--(0.64,0.5)--(0.63,0.58804)--(0.62,0.62402)--(0.61,0.65129)--(0.6,0.67399)--(0.59,0.69374)--(0.58,0.71137)--(0.57,0.72737)--(0.56,0.74206)--(0.55,0.75567) node[above,scale=.8] {$(M,S)$};
\end{tikzpicture}
\end{subfigure}
\begin{subfigure}{.5\textwidth}
\begin{tikzpicture}
[xscale=7,yscale=400,domain=0:1,samples=100]
\draw[ultra thick,->] (0.2,-.001)--(1.1,-.001) node[scale=1.2,above] {$m$};
\draw[ultra thick,->] (0.2,-.001)--(0.2,.01) node[scale=1.2,above right] {$s$};
\node[scale=.8, below left] at (0.2, -.001) {$0$};

\draw[dashed] (1,.01)--(0.2,.01);
\draw[dashed] (1,.01)--(1,-.001) node[scale=.8,below] {$1$};
\draw[dashed] (0.2,.00007)--(1.1,.00007);

\draw[thick,dashdotted,red] (0.2,0)--(0.28935,0.0001)--(0.42749,0.0003)--(0.49156,0.0005)--(0.53367,0.0007)--(0.56504,0.0009)--(0.59004,0.0011)--(0.6108,0.0013)--(0.62854,0.0015)--(0.64402,0.0017)--(0.65774,0.0019)--(0.67006,0.0021)--(0.68123,0.0023)--(0.69145,0.0025)--(0.70086,0.0027)--(0.70957,0.0029)--(0.71768,0.0031)--(0.72527,0.0033)--(0.7324,0.0035)--(0.73911,0.0037)--(0.74545,0.0039)--(0.75146,0.0041)--(0.75717,0.0043)--(0.76261,0.0045)--(0.7678,0.0047)--(0.77276,0.0049)--(0.77751,0.0051)--(0.78207,0.0053)--(0.78645,0.0055)--(0.79066,0.0057)--(0.79471,0.0059)--(0.79862,0.0061)--(0.8024,0.0063)--(0.80604,0.0065)--(0.80957,0.0067)--(0.81299,0.0069)--(0.8163,0.0071)--(0.81951,0.0073)--(0.82263,0.0075)--(0.82565,0.0077)--(0.82859,0.0079)--(0.83145,0.0081)--(0.83424,0.0083)--(0.83695,0.0085)--(0.83959,0.0087)--(0.84216,0.0089)--(0.84467,0.0091)--(0.84712,0.0093)--(0.84951,0.0095)--(0.85185,0.0097)--(0.85413,0.0099);

\draw[ultra thick,  blue, domain=0.2:.35,->] plot(\x,{.1*(.2-\x)*(.2-\x)});
\draw[ultra thick,  blue, domain=0.35:.5,->] plot(\x,{.1*(.2-\x)*(.2-\x)}) node[right,scale=.8] {$(M,S)$};

\node[scale=.8, left] at (0.2,.00007) {$s_{-\infty}$};

\node[scale=.8, red, right] at (0.5,0.0038) {$ m=\ell(s)$};
\draw[thick,->,teal]    (0.5694,0.0059)--(0.60,0.0075);
\draw[thick,->,teal]    (0.1973,0.0090)--(0.3231,0.0090);
\draw[thick,->,teal]    (0.1973,0.009)--(0.3263,0.009);
\draw[thick,->,teal]    (0.1995,0.0052)--(0.2747,0.0052);
\draw[thick,->,teal] (0.1984,0.0020)--(0.2392,0.0020);
\draw[thick,->,teal]  (0.3941,-0.0003)--(0.3941,0.0007);
\draw[thick,->,teal]       (0.5941,0.0005)--(0.5941,0.002);
\draw[thick,->,teal]     (0.77,0.004)--(0.77,0.006);
\draw[thick,->,teal]      (0.8392,0.0013)--(0.7694,0.0025);
\end{tikzpicture}
\end{subfigure}
\caption{(left) The direction of orbits of the dynamical system \eqref{eq:DS} in the phase--plane $[0,1)\times (0,1]$. (right) A zoomed view into the phase--plane near the equilibrium point $(0,s_{-\infty})$. An orbit $(M,S)$ originating from $(0,s_{-\infty})$ is also shown.}\label{fig:OrbitArrow}
\end{figure}
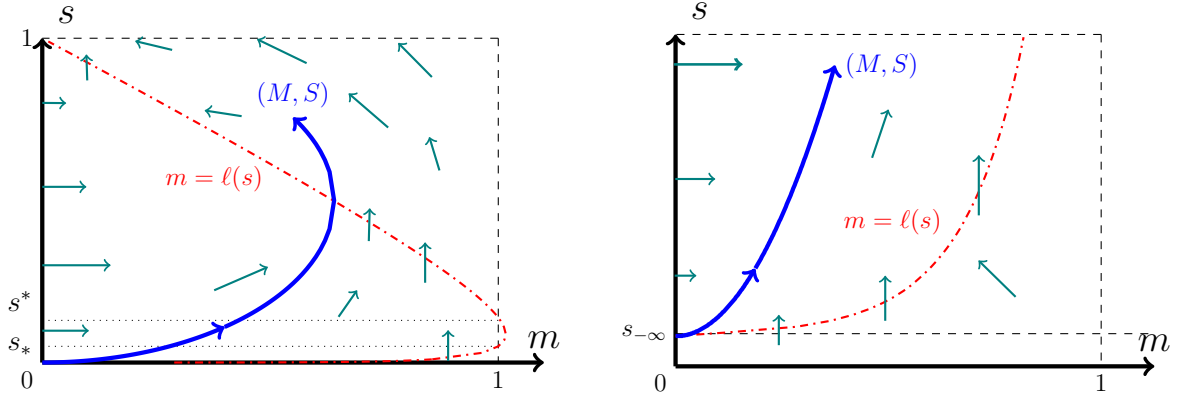

Since this is an autonomous system, we look into the phase--plane $[0,1)\times (0,1]$. Henceforth, $(m,s)$ will represent a point in this phase--plane, and $(M,S)$ will denote the orbits in this plane. The directions of the orbits are shown in \Cref{fig:OrbitArrow}. The line $m=\ell(s)$ is highlighted and represents the nullcline of $M$, i.e., the points where $\dt M=0$. Due to the restriction \eqref{eq:parameter} imposed, this nullcline intersects the line $m=1$ at precisely two points $s_*$ and $s^*$. This follows from the properties of $\ell$ in \eqref{eq:propEll} along with the observation that 
$$\ell(s_{\mathrm{M}})=\g^{-1}[1-s_{\mathrm{M}}-\l F(s_{\mathrm{M}})]\overset{\eqref{eq:parameter}}>1,\quad \text{ where }  s_{\mathrm{M}}=f^{-1}(\l).$$
Hence, the $s$-coordinates $s_*,\,s^*\in (0,1)$ satisfying
\begin{align}\label{eq:SupastSdwnast}
s_*:=\min\{s\in (0,1): \ell(s)=1\}<s_{\mathrm{M}},\quad s^*:=\max\{s\in (0,1): \ell(s)=1\}>s_{\mathrm{M}},
\end{align}
exist (consistent with the definition of $s^*\in (0,1)$ below \eqref{eq:DefG}). Consequently, we have the following:

\begin{lemma}[The existence of orbits] Let $v>0$ be given, \eqref{eq:parameter} be satisfied and let $(M_0,S_0)\in \calR :=[0,1)\times [s_{-\infty},1].$ Then, there exists a unique orbit $\t\mapsto (M,S)\in (C^1(\R))^2$, satisfying \eqref{eq:GE} with $(M,S)(0)=(M_0,S_0)$. The equilibrium  points of the system \eqref{eq:GE} are $(0,s_{-\infty})$ and $(0,1)$ with $s_{-\infty}\in (0,1)$ satisfying \eqref{eq:Sinf}. If $(M_0,S_0)$ is not an equilibrium point, then
\begin{enumerate}[label=(\roman*)]
\item for $\t>0$, the orbit either exits $\calR$ through the line $\{s=1\}$ or ends at $(0,1)$.
\item for $\t<0$, the orbit enters $\calR$ either though the line segment $\{m=0,s\geq s_{-\infty}\}$ or through $\{s=s_{-\infty}\}$.
\end{enumerate}
\label{lemma:orbit_exists}
\end{lemma}
\begin{proof}
Observe that, orbits satisfying \eqref{eq:GE} are locally well-posed at any point in $[0,1)\times [s_{-\infty},1]$. This follows from the  Picard--Lindel\"of theorem since the right hand sides of \eqref{eq:GE} are locally--Lipschitz with respect to $M$ and $S$ for $M<1$ and $S\in[s_{-\infty},1]$.

\emph{(i)} The direction of the orbits implies that $(M,S)$ can exit $\calR$ through the boundaries $\{m=1\}$ or $\{s=1\}$. To rule out the line $\{m=1\}$, we assume the contrary, i.e., we suppose there exists $\t_1>0$ such that $M(\t_1)=1$ and $M(\t)<1$ for all $0<\t<\t_1$. Then, $(M,S)$ satisfies \eqref{eq:GE} for all $\t<\t_1$. It is straightforward to see that $S(\t_1)\in [s_*,s^*]$ since $\lim_{\t\nearrow \t_1}\dt M=v\,[\ell(S(\t_1))-1]$ and the definition of $\t_1$ demands that $\lim_{\t\nearrow \t_1}\dt M\geq 0$. Using the intermediate value theorem, for a given $\varepsilon<1-M_0$, there exists $\t_\varepsilon\in (0,\t_1)$ such that $M(\t_\varepsilon)=1-\varepsilon$. Moreover, observe that there exists a constant $C>0$ such that 
\begin{align}\label{eq:Mleq1contracdiction}
\dt M\leq C \text{ in } (\t_\varepsilon,\t_1),\;\text{ or, integrating in } (\t_\varepsilon,\t_1),\quad  \t_1-\t_\varepsilon\geq \frac{\varepsilon}{C}.
\end{align}
From \eqref{eq:GEs} and using \eqref{eq:Mleq1contracdiction}, for some constants $C_{1\slash 2}>0$ independent of $\varepsilon$, one has 
$$
S(\t_1)-S(\t_\varepsilon)= \int_{\t_\varepsilon}^{\t_1} \tfrac{\g}{v} f(S) M D(M) \dd \t \geq  C_1 \int_{\t_\varepsilon}^{\t_1} D(1-\varepsilon)\,\dd \t \overset{\ref{prop:D}}\geq C_2\frac{\t_1-\t_\varepsilon}{\varepsilon^b}\overset{\eqref{eq:Mleq1contracdiction}}\geq \frac{C_2}{ C\varepsilon^{b-1}}\to \infty,
$$
as $\varepsilon\to 0$ (as $b>1$). Since $S(\t_\varepsilon)>0$, this contradicts $S(\t_1)\in [s_*,s^*]$, thus proving that $(M,S)$ cannot exit through $\{m=1\}$. 

\emph{(ii)} For $\t<0$, the fact that the orbits can enter through the mentioned segments is clear. The orbit cannot enter through the boundary $\{s=1\}$ since $S$ is strictly increasing in $\calR$ for $m>0$, and $(0,1)$ is an equilibrium point. The fact that the orbit cannot enter through $\{m=1\}$ follows similarly as the proof of point (i). 
\end{proof}

\begin{remark}[The travelling wave solutions avoid the degeneracy at $m=1$]
The proof above shows that, for any TW solution $(M,S)$ satisfying \eqref{eq:DS}--\eqref{eq:BC}, there exists a constant $\varepsilon>0$ such that 
$$
0\leq M(\xi)\leq 1-\varepsilon, \text{ for all } \xi\in \R. 
$$
Hence, the degeneracy of the diffusion coefficient $D$, due to the possibility of $D(M)\to \infty$ as $M\nearrow 1$, is avoided.
\end{remark}

\subsection{The orbit connecting with $(0,1)$}\label{sec:orbit01}

For any orbit $(M,S)$ described in \Cref{lemma:orbit_exists}, $S$ is strictly increasing for all $\t\in \R$ provided $(M,S)\in \calR$ and $M>0$. Hence, for a given $s\in (s_{-\infty},1)$ there can  exist at most one $\t\in \R$ such that  $S(\t)=s$. This allows us to introduce the unique mapping $S\mapsto M$ through the following function.

\begin{definition}[The $\Mc$--map]\label{def:Mc}
For a given $(M_0,S_0)\in \calR$, let $(M,S)$ be the unique orbit $\t\mapsto (M,S)\in (C^1(\R))^2$, satisfying \eqref{eq:GE} and $(M,S)(0)=(M_0,S_0)$. Then the continuous function $\Mc:[s_{-\infty},1]\to [0,1)$ is defined as
\begin{align}\label{eq:DefMc}
\Mc(s):=\begin{cases} M(\t) &\text{ if there exists $\t\in \R$ such that } S(\t)=s,\\
0 &\text{ otherwise}.
\end{cases}
\end{align}
\end{definition}
Let us introduce the function
\begin{align}\label{eq:DefPhi}
\Phi(m):= \int_0^m \vr\, D(\vr)\,\dd\vr, \text{ such that } \Phi'(m)=m D(m)\geq 0.
\end{align}
Observe that $\Mc$ satisfies  $\Mc(S_0)=M_0$, and for all $\Mc>0$,
\begin{subequations}\label{eq:McEq}
\begin{align}\label{eq:McEqA}
\dfrac{\dd \Mc}{\dd s}= \dfrac{v^2}{\g}\dfrac{\ell(s)-\Mc(s)}{f(s)\, \Mc\, D(\Mc)}.
\end{align}
Using \eqref{eq:DefPhi}, we alternatively rewrite the equation above as
\begin{align}\label{eq:McEqB}
\dfrac{\dd \Phi(\Mc)}{\dd s}=\dfrac{v^2}{\g f(s)} [\ell(s)- \Mc(s)].
\end{align}
\end{subequations}
Our focus will be on a specific group of maps $\Mc$ which originate from $(\e,1)$.

\begin{figure}

\end{figure}

\begin{lemma}\label{lemma:LowerBound}
For fixed $v,\,\e>0$ and $(M_0,S_0)=(\e,1)$, let $\Mc^{\e}$  denote the $\Mc$-mapping in the sense of \Cref{def:Mc}. For $a>1$ introduced in \ref{prop:D}, let $\underline{\Mc}:(0,1]\to [0,1]$ solve
$$
\int_0^{\underline{\Mc}(s)} \frac{D(\vr)}{\vr^{a-1}}\, \dd\vr= \frac{v^2}{\g} F(s). 
$$ 
 Then there exists $\underline{s}\in (0,1)$ independent of $\e$, such that 
$$
\ell(s)< \underline{\Mc}(s)< \Mc^\e(s) \text{ for all } \underline{s}\leq s< 1.
$$
\end{lemma}

\begin{proof}
Observe from \ref{prop:D} that $\underline{\Mc}$ is well-defined and satisfies the ODE,
\begin{align}\label{eq:McLB}
\dfrac{\dd \Phi(\underline{\Mc})}{\dd s}= -\frac{v^2}{\g} \dfrac{\underline{\Mc}^a}{f(s)} 
\end{align}
and $\underline{\Mc}(1)=0$.
Subtracting \eqref{eq:McEqB} and integrating in $(s,1)$ one gets
\begin{align}
&\Phi(\Mc^\e(s))-\Phi(\underline{\Mc}(s))=\Phi(\e)+ \frac{v^2}{\g}\int_s^1  \left [(\Mc^\e(\vr)-\underline{\Mc}(\vr)) + (\underline{\Mc}(\vr)-\underline{\Mc}^a(\vr) -\ell(\vr))  \right ]\frac{\dd \vr}{f(\vr)} \nonumber\\
&=\Phi(\e)+\frac{v^2}{\g}\int_s^1 \frac{1}{f(\vr)} (\Mc^\e-\underline{\Mc})(\vr) + \frac{v^2}{\g}\int_s^1 \frac{\underline{\Mc}(\vr)}{f(\vr)}  \left (1-\underline{\Mc}^{a-1}(\vr) -\tfrac{\ell(\vr)}{\underline{\Mc}(\vr)}\right ). \label{eq:MlowerBound1}
\end{align}
Note that, $\ell'(1)= -\g^{-1}(1-\l\slash f(1))\overset{\ref{prop:f}}<0$, whereas
$\underline{\Mc}(s)\to 0$ as $s\to 1$, and therefore,
 $$\frac{d\underline{\Mc}}{ds}=-\frac{v^2}{\g}\frac{\underline{\Mc}^{a-1}}{f(s)\,D(\underline{\Mc})}\to-\infty 
 \text{ as } s\to 1.$$
  Hence, using L'H\^{o}pital's rule, there exists $\underline{s}\in (0,1)$ independent of $\e$, such that 
\begin{align}\label{eq:MlowerBound2}
\left (1-\underline{\Mc}^{a-1}(s) -\frac{\ell(s)}{\underline{\Mc}(s)}\right )\geq 0 \text{ for all } \underline{s}\leq s<1.
\end{align}
Observe that $\Mc^\e(s)> \underline{\Mc}(s)$ in a left neighbourhood of $s=1$, simply because $\Mc^\e(1)=\e>0=\underline{\Mc}(1)$. Then, \eqref{eq:MlowerBound1}--\eqref{eq:MlowerBound2}  imply that $\Mc^\e(s)> \underline{\Mc}(s)$ for all $\underline{s}\leq s<1$. To see this, assume the contrary, i.e., $\Mc^\e(s_1)=\underline{\Mc}(s_1)$ for some $s_1\in (\underline{s},1)$ and $\Mc^\e(s)> \underline{\Mc}(s)$ for $s_1<s<1$.
Then, from \eqref{eq:MlowerBound1}--\eqref{eq:MlowerBound2} we have $\Phi(\Mc^\e(s_1))>\Phi(\underline{\Mc}(s_1))$, thus contradicting our assumption. This concludes the proof.
\end{proof}

\begin{theorem}[Existence of an orbit connecting with $(0,1)$] Let $v>0$ be fixed and \eqref{eq:parameter} be satisfied. Let $\Mc^\e$ denote the $\Mc$-mapping introduced in \Cref{lemma:LowerBound} with $\Mc^\e(1)=\e$. Then there exists a  function $\Mc:[s_{-\infty},1]\to [0,1)$ which satisfies \eqref{eq:McEq} with $\Mc(1)=0$ and for all $s\in [s_{-\infty},1]$, $\Mc^\e(s)\to \Mc(s)$ as $\e\to 0$ (see \Cref{fig:Meps}). Moreover, define the function $\z:(s_{-\infty},1]\to (-\infty,0]$ $($the $\z$-$\mathrm{map})$ as
\begin{align}\label{eq:Defzeta}
\z(s):=-\frac{v}{\g}\int_s^1 \frac{\dd \vr}{f(\vr)\,\Mc(\vr)}.
\end{align}
Then $\z$ is differentiable and increasing with $s$ whenever $\Mc(s)>0$, and $\lim_{s\nearrow 1}\z(s)=0$. For any $s\in (s_{-\infty},1]$ and $\z(s)\in \R^-$, defining $\xi=\z(s)$, $M=\Mc(s)$ and $S=s$, the mapping $\xi\mapsto (M,S)$ solves \eqref{eq:DS}.\label{theo:01}
\end{theorem}

\begin{figure}[h!]
\centering
\begin{tikzpicture}
[xscale=10,yscale=4.3,domain=0:1,samples=100]
\draw[ultra thick,->] (0,0)--(1.1,0) node[scale=1.2,above] {$m$};
\draw[ultra thick,->] (0,0)--(0,1.1) node[scale=1.2,left] {$s$};
\node[scale=.8, below left] at (0, 0) {$0$};

\draw[thick, dashed] (1,1)--(0,1) node[scale=.8,left] {$1$};
\draw[thick, dashed] (1,1)--(1,0) node[scale=.8,below] {$1$};

\draw[thick,dashdotted,red] (0.85525,0.01)--(0.96467,0.03)--(1,0.05)--(1.0134,0.07)--(1.0161,0.09)--(1.0124,0.11)--(1.0044,0.13)--(0.99346,0.15)--(0.98023,0.17)--(0.96525,0.19)--(0.94886,0.21)--(0.93132,0.23)--(0.91283,0.25)--(0.89352,0.27)--(0.87353,0.29)--(0.85293,0.31)--(0.83181,0.33)--(0.81022,0.35)--(0.78822,0.37)--(0.76586,0.39)--(0.74316,0.41)--(0.72016,0.43)--(0.69689,0.45)--(0.67337,0.47)--(0.64962,0.49)--(0.62566,0.51)--(0.60151,0.53)--(0.57717,0.55)--(0.55267,0.57)--(0.52802,0.59)--(0.50322,0.61)--(0.47828,0.63)--(0.45322,0.65)--(0.42804,0.67)--(0.40275,0.69)--(0.37735,0.71)--(0.35185,0.73)--(0.32625,0.75)--(0.30057,0.77)--(0.2748,0.79)--(0.24895,0.81)--(0.22302,0.83)--(0.19702,0.85)--(0.17095,0.87)--(0.14482,0.89)--(0.11862,0.91)--(0.092356,0.93)--(0.066037,0.95)--(0.039662,0.97)--(0.013234,0.99);

\draw[thick,dashdotted,red] (0.28935,0.0001)--(0.42749,0.0003)--(0.49156,0.0005)--(0.53367,0.0007)--(0.56504,0.0009)--(0.59004,0.0011)--(0.6108,0.0013)--(0.62854,0.0015)--(0.64402,0.0017)--(0.65774,0.0019)--(0.67006,0.0021)--(0.68123,0.0023)--(0.69145,0.0025)--(0.70086,0.0027)--(0.70957,0.0029)--(0.71768,0.0031)--(0.72527,0.0033)--(0.7324,0.0035)--(0.73911,0.0037)--(0.74545,0.0039)--(0.75146,0.0041)--(0.75717,0.0043)--(0.76261,0.0045)--(0.7678,0.0047)--(0.77276,0.0049)--(0.77751,0.0051)--(0.78207,0.0053)--(0.78645,0.0055)--(0.79066,0.0057)--(0.79471,0.0059)--(0.79862,0.0061)--(0.8024,0.0063)--(0.80604,0.0065)--(0.80957,0.0067)--(0.81299,0.0069)--(0.8163,0.0071)--(0.81951,0.0073)--(0.82263,0.0075)--(0.82565,0.0077)--(0.82859,0.0079)--(0.83145,0.0081)--(0.83424,0.0083)--(0.83695,0.0085)--(0.83959,0.0087)--(0.84216,0.0089)--(0.84467,0.0091)--(0.84712,0.0093)--(0.84951,0.0095)--(0.85185,0.0097)--(0.85413,0.0099);

\node[scale=.9, above,red] at (0.37, 0.52) {$m
=\ell(s)$};
\node[scale=.9, right] at (0.0378,0.3981) {$0<\e_1<\e_2$};

\node[scale=.8,above] at (0.15,1) {$(\e_1,1)$};
\node[scale=.8,above] at (0.25,1) {$(\e_2,1)$};

\draw[ultra thick,magenta,dashdotted]
(0,1)--(0.1,0.99)--(0.14142,0.98)--(0.17321,0.97)--(0.2,0.96)--(0.22361,0.95)--(0.24495,0.94)--(0.26458,0.93)--(0.28284,0.92)--(0.3,0.91)--(0.31623,0.9)--(0.33166,0.89)--(0.34641,0.88)--(0.36056,0.87)--(0.37417,0.86)--(0.3873,0.85)--(0.4,0.84)--(0.41231,0.83)--(0.42426,0.82)--(0.43589,0.81)--(0.44721,0.8)--(0.45826,0.79)--(0.46904,0.78)--(0.47958,0.77)--(0.4899,0.76)--(0.5,0.75)--(0.5099,0.74)--(0.51962,0.73)--(0.52915,0.72)--(0.53852,0.71)--(0.54772,0.7);
\node[below,scale=0.8] at (0.36056,0.87) {$\underline{\Mc}$};

\draw[dashed] (1,0.7)--(0,0.7) node[scale=0.8,left] {$\underline{s}$};

\draw[ultra thick,blue,->]
(0,1)--(0.01,0.99994)--(0.02,0.99976)--(0.03,0.99945)--(0.04,0.99902)--(0.05,0.99847)--(0.06,0.9978)--(0.07,0.997)--(0.08,0.99608)--(0.09,0.99503)--(0.1,0.99386)--(0.11,0.99256)--(0.12,0.99113)--(0.13,0.98958)--(0.14,0.98789)--(0.15,0.98607)--(0.16,0.98412)--(0.17,0.98204)--(0.18,0.97982)--(0.19,0.97746)--(0.2,0.97496)--(0.21,0.97232)--(0.22,0.96953)--(0.23,0.9666)--(0.24,0.96351)--(0.25,0.96027)--(0.26,0.95688)--(0.27,0.95333)--(0.28,0.94961)--(0.29,0.94572)--(0.3,0.94167)--(0.31,0.93743)--(0.32,0.93301)--(0.33,0.92841)--(0.34,0.92361)--(0.35,0.91861)--(0.36,0.9134)--(0.37,0.90797)--(0.38,0.90232)--(0.39,0.89644)--(0.4,0.89031)--(0.41,0.88393);

\draw[ultra thick,blue,->] (0.41,0.88393)--(0.42,0.87727)--(0.43,0.87033)--(0.44,0.86309)--(0.45,0.85553)--(0.46,0.84763)--(0.47,0.83937)--(0.48,0.83072)--(0.49,0.82164)--(0.5,0.81211)--(0.51,0.80207)--(0.52,0.79148)--(0.53,0.78027)--(0.54,0.76837)--(0.55,0.75567)--(0.56,0.74206)--(0.57,0.72737)--(0.58,0.71137)--(0.59,0.69374)--(0.6,0.67399)--(0.61,0.65129)--(0.62,0.62402)--(0.63,0.58804)--(0.64,0.5)--(0.64,0.5)--(0.63,0.41196)--(0.62,0.37598)--(0.61,0.34871)--(0.6,0.32601)--(0.59,0.30626)--(0.58,0.28863)--(0.57,0.27263)--(0.56,0.25794)--(0.55,0.24433)--(0.54,0.23163)--(0.53,0.21973)--(0.52,0.20852)--(0.51,0.19793)--(0.5,0.18789) node[below left,scale=.8] {$\Mc$};

\draw[fill,cyan] (0,1) ellipse (.2pt and .3pt);
\draw[ultra thick] (0,1) ellipse (.2pt and .3pt);

\draw[ultra thick,dotted,teal,->] 
(0.17,1)--(0.18,0.99978)--(0.19,0.99742)--(0.2,0.99492)--(0.21,0.99229)--(0.22,0.9895)--(0.23,0.98658)--(0.24,0.9835)--(0.25,0.98028)--(0.26,0.9769)--(0.27,0.97336)--(0.28,0.96967)--(0.29,0.96581)--(0.3,0.96178)--(0.31,0.95759)--(0.32,0.95321)--(0.33,0.94866)--(0.34,0.94392)--(0.35,0.93899)--(0.36,0.93386)--(0.37,0.92853)--(0.38,0.92298)--(0.39,0.91722)--(0.4,0.91123)--(0.41,0.905)--(0.42,0.89852)--(0.43,0.89178)--(0.44,0.88478)--(0.45,0.87748)--(0.46,0.86989)--(0.47,0.86197)--(0.48,0.85372)--(0.49,0.84511)--(0.5,0.8361);

\draw[ultra thick,dotted,teal,->] (0.5,0.8361)--(0.51,0.82669)--(0.52,0.81682)--(0.53,0.80647)--(0.54,0.79558)--(0.55,0.78409)--(0.56,0.77195)--(0.57,0.75907)--(0.58,0.74534)--(0.59,0.73063)--(0.6,0.71476)--(0.61,0.69749)--(0.62,0.67848)--(0.63,0.65719)--(0.64,0.63274)--(0.65,0.60338)--(0.66,0.56467)--(0.67,0.47)--(0.67,0.47)--(0.66,0.37533)--(0.65,0.33662)--(0.64,0.30726)--(0.63,0.28281)--(0.62,0.26152)--(0.61,0.24251)--(0.6,0.22524)--(0.59,0.20937)--(0.58,0.19466)--(0.57,0.18093)--(0.56,0.16805)--(0.55,0.15591) node[below,scale=.8] {$\Mc^{\e_1}$};

\draw[fill,cyan] (.17,1) ellipse (.2pt and .3pt);
\draw[ultra thick] (.17,1) ellipse (.2pt and .3pt);

\draw[ultra thick,dashed,orange,->] (0.25,1.000)--(0.26,0.99934)--(0.27,0.99603)--(0.28,0.99258)--(0.29,0.98898)--(0.3,0.98523)--(0.31,0.98132)--(0.32,0.97726)--(0.33,0.97303)--(0.34,0.96865)--(0.35,0.96409)--(0.36,0.95937)--(0.37,0.95446)--(0.38,0.94938)--(0.39,0.9441)--(0.4,0.93864)--(0.41,0.93298)--(0.42,0.92711)--(0.43,0.92102)--(0.44,0.91472)--(0.45,0.90819)--(0.46,0.90142)--(0.47,0.8944)--(0.48,0.88712)--(0.49,0.87957)--(0.5,0.87174)--(0.51,0.8636)--(0.52,0.85514)--(0.53,0.84635)--(0.54,0.8372)--(0.55,0.82767)--(0.56,0.81773)--(0.57,0.80735)--(0.58,0.79649)--(0.59,0.78511)--(0.6,0.77316);

\draw[ultra thick,dashed,orange,->] (0.6,0.77316)--(0.61,0.76057)--(0.62,0.74729)--(0.63,0.73321)--(0.64,0.71823)--(0.65,0.7022)--(0.66,0.68493)--(0.67,0.66615)--(0.68,0.64551)--(0.69,0.62242)--(0.7,0.59591)--(0.71,0.56413)--(0.72,0.52227)--(0.73,0.42)--(0.72,0.31773)--(0.71,0.27587)--(0.7,0.24409)--(0.69,0.21758)--(0.68,0.19449)--(0.67,0.17385)--(0.66,0.15507)--(0.65,0.1378)--(0.64,0.12177)--(0.63,0.10679) node[below,scale=.8] {$\Mc^{\e_2}$};

\draw[fill,cyan] (.25,1) ellipse (.2pt and .3pt);
\draw[ultra thick] (.25,1) ellipse (.2pt and .3pt);

\end{tikzpicture}
\caption{The $\Mc$-mapping, introduced in \Cref{theo:01} with $\Mc(1)=0$, and two $\Mc^{\e}$-mappings, introduced in \Cref{lemma:LowerBound} with $\Mc^{\e}(1)=\e$ for $\e=\e_1>0$ and $\e=\e_2>\e_1$. The function $\underline{\Mc}$ defined in \Cref{lemma:LowerBound}, providing a lower bound, is also shown.}\label{fig:Meps}
\end{figure}

\begin{proof}
For a fixed $v>0$ and $s\in (s_{-\infty},1]$,  
\begin{align}
&\Mc_{\e_1}(s)\leq \Mc_{\e_2}(s)<1 &\text{ if } 0<\e_1<\e_2<1,\nonumber\\
&\text{and the equality holds only if }& \Mc_{\e_{1}}(s)= \Mc_{\e_{2}}(s)=0.\label{eq:MepsOrder}
\end{align}
This is evident since orbits, corresponding to initial values $(\e_1,1)$ and $(\e_2,1)$, do not intersect in the interior of $\calR$ due to uniqueness of solutions, see \Cref{lemma:orbit_exists}. The lemma further yields that $\Mc^{\e}<1$ if $\e<1$. Since, $\Mc^\e(s)$ is bounded below by  $\underline{\Mc}(s)$ for $s\in (\underline{s},1)$, see \Cref{lemma:LowerBound}, there exists $\Mc(s)$ such that
\begin{align}
 \Mc(s):=\lim\limits_{\e\searrow 0} \Mc^\e(s) \;\; \text{ for } s\in (s_{-\infty},1),\;\; \text{ and }\;\;  \Mc(s)\ge \underline{\Mc}(s)>0 \text{ for } s\in (\underline{s},1),
\end{align}
see \Cref{fig:Meps}. Let us take $s\in (\underline{s},1)$. Observe that by \eqref{eq:McEqB}, $\Mc^\e$ satisfies in this interval
\begin{align}
\Phi(\Mc^\e(s))= \Phi(\e)+ \frac{v^2}{\g}\int_s^1  (\Mc^\e(\vr)-\ell(\vr))\,\frac{\dd \vr}{f(\vr)}.
\end{align}
Due to \eqref{eq:MepsOrder}, if $\e_0\in (0,1)$, then $\Mc^\e$ is uniformly bounded away from 1 in $[s_{-\infty},1]$ for all $\e\in (0,\e_0]$. Thus, $\Phi$ can be assumed to be locally Lipschitz in the above equation. Hence, passing the limit $\e\to 0$, we get for all $s\in (\underline{s},1)$,
\begin{align}
\Phi(\Mc(s))= \frac{v^2}{\g}\int_s^1 (\Mc(\vr)-\ell(\vr))\,\frac{\dd \vr}{f(\vr)},
 \end{align} 
 which upon differentiation proves that $\Mc$ satisfies \eqref{eq:McEq} in $(\underline{s},1)$ with $(M_0,S_0)=(0,1)$.  For $s<\underline{s}$, the mapping $\Mc$ is simply extended by solving the equation  \eqref{eq:McEq} with $(M_0,S_0)=(\Mc(\underline{s}),\underline{s})\in \calR$. The existence of $\Mc$ in this case  also follows from the existence of orbits, i.e., \Cref{lemma:orbit_exists}. 
 
The differentiability and monotonicity of the function $\z$ is obvious from \eqref{eq:Defzeta}. To prove that $\lim_{s\nearrow 1} \z(s)=0$,  we estimate for any $s\in (\underline{s},1)$,
\begin{align}
0< \int_s^1 \frac{\dd \vr}{f(\vr)\,\Mc(\vr)}\overset{\eqref{eq:MepsOrder}}\leq  \int_s^1 \frac{\dd \vr}{f(\vr)\,\underline{\Mc}(\vr)}\overset{\eqref{eq:McLB}}= \frac{\gamma}{v^2}\int^{\underline{\Mc}(s)}_0 \frac{D(m)}{m^a}\dd m<\infty.
 \end{align} 
Hence, $\z(s)>-\infty$ for all $s\in (\underline{s},1]$, and passing to the limit $s\to 1$ one obtains that $\z(1^-)=0$. 

Differentiating \eqref{eq:Defzeta} and using \eqref{eq:McEqA}, it immediately follows that the mapping $\xi\mapsto (M,S)$ solves \eqref{eq:DS}.
\end{proof}

\begin{corollary}[Behaviour of the orbit connecting to $(0,1)$]\label{cor:orbit}
Let $(M,S)$ be the orbit defined in \Cref{theo:01}. Then there exists $\underline{\xi}\in \R^-\cup \{-\infty\}$ such that $(M,S)\in \calR=[0,1)\times [s_{-\infty},1]$  for all $\xi\geq \underline{\xi}$. In $\calR\cap \{s\geq s_*\}$, both $M$ and $S$ increase with $\xi$ until $M=\ell(S)$ is satisfied for some $S>s^*$ $(s_*,\;s^*\in (s_{-\infty},1)$ defined in \eqref{eq:SupastSdwnast}$)$, after which $M$ decreases and $S$ remains increasing. Finally, $(M,S)=(0,1)$ for all $\xi\geq 0$. 
\end{corollary}
 The statement is evident from the direction of orbits in the phase--plane, \Cref{fig:OrbitArrow}, \Cref{theo:01} and \Cref{lemma:orbit_exists}. Below, we prove that $(M,S)$ is the unique orbit which connects with $(0,1)$ at a finite $\xi$-coordinate.
 
 \subsection{Uniqueness of the orbit connecting with $(0,1)$ at  $\xi=0$}
For a given $v>0$, there are in fact infinitely many orbits $(M,S)$ that connect to $(0,1)$ as $\xi\to +\infty$. However, only one orbit, the one constructed in \Cref{sec:orbit01}, connects to $(0,1)$ at $\xi=0$. This statement will be proved below. This is a common phenomenon for TW solutions of  degenerate diffusion equations \cite{de1998travelling,de1991travelling} and the unique orbit corresponds to the TW with minimum speed in these cases.
 
\begin{proposition}[Uniqueness of the orbit connecting to $(0,1)$ for some $\xi\in \R$] For a fixed $v>0$, let $(\widetilde{M},\widetilde{S})\in (C^1(\R))^2$ be an orbit satisfying \eqref{eq:DS} and connecting with $(0,1)$ from $\{s<1\}$. Let $\widetilde{\Mc}$ and $\widetilde{\z}$ denote the corresponding $\Mc$-mapping (\Cref{def:Mc}) and $\z$-mapping (\Cref{theo:01}) of $(\widetilde{M},\widetilde{S})$ with $\widetilde{\xi}=\widetilde{\z}(1)$. Then, $\widetilde{\xi}<\infty$ if and only if $(\widetilde{M},\widetilde{S})$ is the orbit defined in \Cref{theo:01}.\label{pros:Unique}
\end{proposition} 
 
\begin{proof}
 We first show that 
\begin{align}\label{eq:dMdSinft}
 \widetilde{\xi}=\widetilde{\z}(1)<\infty\quad \text{ implies }\quad \frac{\dd}{\dd s}\widetilde{\Mc}(1)=-\infty.
 \end{align}
  Assume the contrary, i.e, $\frac{\dd}{\dd s}\widetilde{\Mc}(1)>-\infty$. Then, there exists $s\in (s_{-\infty},1)$ and $L>0$ such that 
$$
\widetilde{\Mc}(\vr)\leq L\,(1-\vr) \text{ for all } \vr\in [s,1].
$$
This implies $\widetilde{\xi}=\widetilde{\z}(1)=\infty$ since, writing formally,
$$
\widetilde{\z}(1)-\widetilde{\z}(s)=\frac{v}{\g} \int^1_s \frac{\dd \vr}{f(\vr)\,\widetilde{\Mc}(\vr)}\geq  \frac{v}{\g} \int^1_s \frac{\dd \vr}{L\,f(1)\,(1-\vr)}= \infty.
$$
Now, let us assume $\frac{\dd}{\dd s} \widetilde{\Mc}(1)=-\infty$.
Since $\widetilde{\Mc}$ is continuous and differentiable in a left neighbourhood of $s=1$, and $\ell'(1)<0$ is bounded, for any given $\nu\in (0,1)$, by the L'H\^{o}pital's rule, there exists $\widetilde{s}_\nu\in (0,1)$ such that 
\begin{align}\label{eq:elltildeM}
 \frac{\ell(s)}{\widetilde{\Mc}(s)} <1-\nu   \text{ for all } s\in [\widetilde{s}_\nu,1].
 \end{align}
 Let $(M,S)$ be the orbit defined  in \Cref{theo:01} with $\Mc$ and $\z$ being the corresponding mappings. To shorten notations we introduce
\begin{align}\label{eq:shorthandPhi}
 \phi=\Phi(\Mc),\quad \widetilde{\phi}=\Phi(\widetilde{\Mc}),\; \text{ and }\; \eth \phi=\phi-\widetilde{\phi}.  
 \end{align}
  Since $\Mc(s)=\lim_{\e\searrow 0}\Mc^\e(s)$, see \Cref{theo:01}, the ordering of orbits for a fixed $v>0$ implies that 
\begin{align}\label{eq:MtildeMorder}
  \Mc(s)\geq \widetilde{\Mc}(s),\quad \phi\geq \widetilde{\phi},\quad \eth\phi\geq 0,\; \text{ for all } s\in (s_{-\infty},1).  
  \end{align}
From \eqref{eq:McEqB} applied to $\Mc$ and $\widetilde{\Mc}$, we then obtain
\begin{align}\label{eq:MtildeMsubtract}
\frac{\dd (\eth\phi)}{\dd s}= -\frac{v^2}{\g f(s)}[\Phi^{-1}(\phi)-\Phi^{-1}(\widetilde{\phi})].
\end{align}
Since $\Phi$ is convex and strictly increasing, $\Phi^{-1}$ exists and is concave, and $\{\Phi^{-1}\}'(\Phi(m))=1\slash  \Phi'(m)=1\slash (m\,D(m))$. Hence, \eqref{eq:MtildeMorder} implies $\Phi^{-1}(\phi)-\Phi^{-1}(\widetilde{\phi})\leq \{\Phi^{-1}\}'(\widetilde{\phi})\, \eth\phi={\eth\phi}\slash{(\widetilde{\Mc}\,D(\widetilde{\Mc}))}$. Thus, \eqref{eq:MtildeMsubtract} yields
\begin{align}
\frac{\dd (\eth\phi)}{\dd s} + \frac{v^2}{\g f(s)} \frac{\eth\phi}{\widetilde{\Mc}\,D(\widetilde{\Mc})}\geq 0.
\end{align}
Using the integrating factor $\exp\left (\frac{v^2}{\g} \int_{\widetilde{s}_\nu}^s \frac{\dd \vr}{f\,\widetilde{\Mc}\,D(\widetilde{\Mc})}\right )$ and integrating in $ (\widetilde{s}_\nu,s)$ one has
\begin{align}\label{eq:DeltaPhiRel}
\eth \phi(s) \geq \eth\phi(\widetilde{s}_\nu)\, \exp\left (-\frac{v^2}{\g} \int_{\widetilde{s}_\nu}^s  \frac{\dd \vr}{f(\vr)\,\widetilde{\Mc}(\vr)\,D(\widetilde{\Mc}(\vr))}\right ).
\end{align}
Observe from \eqref{eq:elltildeM} and \eqref{eq:McEqA} that $\frac{\dd \widetilde{\Mc}}{\dd s}<0$ for $s\in (\widetilde{s}_\nu,1)$, which gives
\begin{align}
-\frac{v^2}{\g} \frac{1}{f\,\widetilde{\Mc}\,D(\widetilde{\Mc})}\overset{\eqref{eq:McEqA}}= \frac{1}{\widetilde{\Mc}(s)-\ell(s)}\frac{\dd \widetilde{\Mc}}{\dd s}=  \tfrac{1}{\left (1-\frac{\ell(s)}{\widetilde{\Mc}(s)}\right )}\frac{1}{\widetilde{\Mc}}\frac{\dd \widetilde{\Mc}}{\dd s}
 \overset{\eqref{eq:elltildeM}}\geq  \frac{1}{\nu} \frac{1}{\widetilde{\Mc}}\frac{\dd \widetilde{\Mc}}{\dd s}.
\end{align}
Putting this in \eqref{eq:DeltaPhiRel}, we have
$$
\eth \phi(s) \geq \eth\phi(\widetilde{s}_\nu)\, \exp\left ( \frac{1}{\nu}\int_{\widetilde{s}_\nu}^s \frac{\dd \widetilde{\Mc}}{\widetilde{\Mc}}\right )= \eth\phi(\widetilde{s}_\nu) \left (\frac{\widetilde{\Mc}(s)}{\widetilde{\Mc}(\widetilde{s}_\nu)}\right )^{\frac{1}{\nu}}.
$$
Rearranging the above relation using \eqref{eq:shorthandPhi} one obtains for a constant $\widetilde{C}^1_\nu>0$ only dependent on $\widetilde{s}_\nu<1$ that
\begin{align}\label{eq:TildePhiIneq}
\Phi(\widetilde{\Mc}) + (\widetilde{C}^1_\nu \widetilde{\Mc})^{\frac{1}{\nu}}\leq \Phi(\Mc), \quad \text{ or }\quad \widetilde{C}^1_\nu \widetilde{\Mc}\leq (\Phi(\Mc))^{\nu}.
\end{align}
 Recalling \ref{prop:D} and the definition of $\Phi$ in \eqref{eq:DefPhi}, we choose $\nu\in (0,1)$ such that there exists a constant $\widetilde{C}^2_\nu>0$ for which
 \begin{align}\label{eq:PhiGrowthRate}
 \Phi(m)^{\nu}\leq \widetilde{C}^2_\nu \Phi'(m)\overset{\eqref{eq:DefPhi}}= \widetilde{C}^2_\nu\, m\,D(m) \text{ for all } m\in (0,1).
 \end{align}
In a right neighbourhood $[0,\e]$ of $m=0$ ($\epsilon>0$), where $D(m)\sim C m^a$ (see \ref{prop:D}) for some constant $C>0$, we have $\Phi'(m)\sim C m^{a+1}$ and $\Phi(m)\sim \frac{C}{a+2} m^{a+2}$. Hence, for any constant $\nu\in [(1+a)\slash (2+a), 1)$ there exists $\widetilde{C}^2_\nu>0$ satisfying \eqref{eq:PhiGrowthRate} in $[0,\e]$. For $m\in (0,1)$, one can  take $\widetilde{C}^2_\nu$ to be the minimum of the value of $\widetilde{C}^2_\nu$ for $m\in [0,\e]$,  and $\min_{m\in (\e,1)} \Phi(m)^{\nu} (\Phi'(m))^{-1}>0$. Then, for $s\in (\widetilde{s}_\nu,1)$ one has
\begin{align}\label{eq:FinalTildeIneq}
\widetilde{\z}(s)-\widetilde{\z}(\widetilde{s}_\nu)&=  \frac{v}{\g}\int_{\widetilde{s}_\nu}^s \frac{\dd \vr}{f(\vr)\,\widetilde{\Mc}(\vr)}\overset{\eqref{eq:TildePhiIneq}}\geq \frac{v}{\g \widetilde{C}^1_\nu}\int_{\widetilde{s}_\nu}^s\frac{\dd \vr}{f(\vr)\,\Phi(\Mc(\vr))^{\nu}}\nonumber\\
&\overset{\eqref{eq:PhiGrowthRate}}\geq \frac{v}{\g \widetilde{C}^1_\nu \widetilde{C}^2_\nu} \int_{\widetilde{s}_\nu}^s \frac{\dd \vr}{f(\vr)\,\Mc(\vr) D(\Mc(\vr))}\overset{\eqref{eq:GEs}}= \frac{1}{\widetilde{C}^1_\nu \widetilde{C}^2_\nu}\int_{\widetilde{s}_\nu}^s \left (\frac{\dd \t}{\dd S}\right )\dd S.
\end{align}
The orbit $(M,S)\to (0,1)$ only as $\t\to \infty$ since $(0,1)$ is an equilibrium point of \eqref{eq:GE} and the right hand side of \eqref{eq:GE} is locally Lipschitz with respect to $M$ and $S$.
Hence, the right hand side of \eqref{eq:FinalTildeIneq} tends to $+\infty$ as $s\nearrow 1$.
This contradicts $\widetilde{\xi}=\widetilde{\z}(1)$ being bounded, which concludes the proof of the proposition.

\end{proof}

 By this point we have found a unique orbit which satisfies half of the boundary conditions in \eqref{eq:BC}, i.e. $(M,S)(0)=(0,1)$. For the boundary condition at the other end, i.e. $\xi\to -\infty$, we investigate how the orbit varies with $v$.

\subsection{Ordering of the orbits with respect to $v$}
\begin{figure}[h!]
\centering
\begin{tikzpicture}
[xscale=12,yscale=4.3,domain=0:1,samples=100]
\draw[ultra thick,->] (0,0)--(1.1,0) node[scale=1.2,above] {$m$};
\draw[ultra thick,->] (0,0)--(0,1) node[scale=1.2,above right] {$s$};
\node[scale=.8, below left] at (0, 0) {$0$};

\draw[dashed] (1,1)--(0,1) node[scale=.8,left] {$1$};
\draw[dashed] (1,1)--(1,0) node[scale=.8,below] {$1$};

\draw[dotted] (1,0.13)--(0,0.13) node[scale=.9,above left] {$s^*$};
\draw[dotted] (1,0.05)--(0,0.05) node[scale=.9,left] {$s_*$};

\draw[thick,dashdotted,red] (0.85525,0.01)--(0.96467,0.03)--(1,0.05)--(1.0134,0.07)--(1.0161,0.09)--(1.0124,0.11)--(1.0044,0.13)--(0.99346,0.15)--(0.98023,0.17)--(0.96525,0.19)--(0.94886,0.21)--(0.93132,0.23)--(0.91283,0.25)--(0.89352,0.27)--(0.87353,0.29)--(0.85293,0.31)--(0.83181,0.33)--(0.81022,0.35)--(0.78822,0.37)--(0.76586,0.39)--(0.74316,0.41)--(0.72016,0.43)--(0.69689,0.45)--(0.67337,0.47)--(0.64962,0.49)--(0.62566,0.51)--(0.60151,0.53)--(0.57717,0.55)--(0.55267,0.57)--(0.52802,0.59)--(0.50322,0.61)--(0.47828,0.63)--(0.45322,0.65)--(0.42804,0.67)--(0.40275,0.69)--(0.37735,0.71)--(0.35185,0.73)--(0.32625,0.75)--(0.30057,0.77)--(0.2748,0.79)--(0.24895,0.81)--(0.22302,0.83)--(0.19702,0.85)--(0.17095,0.87)--(0.14482,0.89)--(0.11862,0.91)--(0.092356,0.93)--(0.066037,0.95)--(0.039662,0.97)--(0.013234,0.99);

\draw[thick,dashdotted,red] (0.28935,0.0001)--(0.42749,0.0003)--(0.49156,0.0005)--(0.53367,0.0007)--(0.56504,0.0009)--(0.59004,0.0011)--(0.6108,0.0013)--(0.62854,0.0015)--(0.64402,0.0017)--(0.65774,0.0019)--(0.67006,0.0021)--(0.68123,0.0023)--(0.69145,0.0025)--(0.70086,0.0027)--(0.70957,0.0029)--(0.71768,0.0031)--(0.72527,0.0033)--(0.7324,0.0035)--(0.73911,0.0037)--(0.74545,0.0039)--(0.75146,0.0041)--(0.75717,0.0043)--(0.76261,0.0045)--(0.7678,0.0047)--(0.77276,0.0049)--(0.77751,0.0051)--(0.78207,0.0053)--(0.78645,0.0055)--(0.79066,0.0057)--(0.79471,0.0059)--(0.79862,0.0061)--(0.8024,0.0063)--(0.80604,0.0065)--(0.80957,0.0067)--(0.81299,0.0069)--(0.8163,0.0071)--(0.81951,0.0073)--(0.82263,0.0075)--(0.82565,0.0077)--(0.82859,0.0079)--(0.83145,0.0081)--(0.83424,0.0083)--(0.83695,0.0085)--(0.83959,0.0087)--(0.84216,0.0089)--(0.84467,0.0091)--(0.84712,0.0093)--(0.84951,0.0095)--(0.85185,0.0097)--(0.85413,0.0099);

\node[scale=.9, above,red] at (0.3, 0.55) {$\; m
=\ell(s)$};
\node[scale=.9, right] at (0.0378,0.3981) {$v_1<v_2<v_3<v_4$};

\draw[ultra thick, blue,->] (0,0.999)--(0.065441,0.999)--(0.15499,0.99899)--(0.15692,0.99899)--(0.15888,0.99899)--(0.16086,0.99899)--(0.16286,0.99899)--(0.16489,0.99899)--(0.16695,0.99899)--(0.16903,0.99899)--(0.17114,0.99899)--(0.17327,0.99899)--(0.17544,0.99899)--(0.17762,0.99899)--(0.17984,0.99899)--(0.18208,0.99899)--(0.18436,0.99899)--(0.18666,0.99899)--(0.18899,0.99898)--(0.19135,0.99898)--(0.19373,0.99898)--(0.19615,0.99898)--(0.1986,0.99898)--(0.20108,0.99898)--(0.20359,0.99898)--(0.20614,0.99897)--(0.20871,0.99897)--(0.21132,0.99897)--(0.21396,0.99897)--(0.21663,0.99897)--(0.21934,0.99896)--(0.22208,0.99896)--(0.22485,0.99896)--(0.22766,0.99895)--(0.2305,0.99895)--(0.23339,0.99895)--(0.2363,0.99894)--(0.23926,0.99894)--(0.24225,0.99893)--(0.24527,0.99893)--(0.24834,0.99892)--(0.25144,0.99892)--(0.25459,0.99891)--(0.25777,0.9989)--(0.26099,0.99889)--(0.26426,0.99889)--(0.26756,0.99888)--(0.2709,0.99887)--(0.27429,0.99886)--(0.27772,0.99885)--(0.28119,0.99883)--(0.28471,0.99882)--(0.28827,0.9988)--(0.29187,0.99879)--(0.29552,0.99877)--(0.29922,0.99875)--(0.30296,0.99873)--(0.30675,0.99871)--(0.31058,0.99869)--(0.31447,0.99866)--(0.3184,0.99863)--(0.32238,0.9986)--(0.32641,0.99856)--(0.33049,0.99853)--(0.33462,0.99849)--(0.3388,0.99844)--(0.34303,0.99839)--(0.34732,0.99834)--(0.35166,0.99828)--(0.35605,0.99822)--(0.36049,0.99815)--(0.36499,0.99808)--(0.36955,0.99799)--(0.37416,0.9979)--(0.37883,0.9978)--(0.38355,0.9977)--(0.38834,0.99758)--(0.39318,0.99744)--(0.39807,0.9973)--(0.40303,0.99714)--(0.40805,0.99697)--(0.41313,0.99678)--(0.41826,0.99657)--(0.42346,0.99633)--(0.42872,0.99608)--(0.43404,0.99579)--(0.43942,0.99548)--(0.44486,0.99514)--(0.45037,0.99475)--(0.45594,0.99433)--(0.46156,0.99386)--(0.46726,0.99334)--(0.47301,0.99276)--(0.47882,0.99212)--(0.4847,0.99141)--(0.49063,0.99062)--(0.49663,0.98974)--(0.50268,0.98875)--(0.5088,0.98765)--(0.51496,0.98643)--(0.52119,0.98506)--(0.52746,0.98352)--(0.53378,0.98181)--(0.54016,0.97988)--(0.54657,0.97772)--(0.55303,0.9753)--(0.55951,0.97258)--(0.56603,0.96952)--(0.57258,0.96608)--(0.57914,0.9622)--(0.58571,0.95784)--(0.59228,0.95292)--(0.59883,0.94737)--(0.60537,0.94112)--(0.61187,0.93407)--(0.61831,0.92613)--(0.62469,0.91719)--(0.63098,0.90712)--(0.63715,0.8958)--(0.64319,0.88311)--(0.64906,0.8689)--(0.65474,0.85303)--(0.66019,0.83538)--(0.66538,0.81582)--(0.67026,0.79425)--(0.6748,0.7706)--(0.67896,0.74484);

\draw[ultra thick, blue,->] (0.67896,0.74484)--(0.6827,0.717)--(0.68597,0.68717)--(0.68873,0.65549)--(0.69095,0.62222)--(0.6926,0.58765)--(0.69364,0.55216)--(0.69407,0.51618)--(0.69387,0.48016)--(0.69303,0.44458)--(0.69157,0.40987)--(0.68949,0.37645)--(0.68682,0.34466)--(0.68358,0.31478)--(0.6798,0.28699)--(0.67552,0.26142)--(0.67077,0.23811)--(0.66558,0.21704)--(0.65999,0.19814)--(0.65403,0.18132)--(0.64774,0.16643)--(0.64114,0.15332)--(0.63426,0.14185)--(0.62712,0.13184)--(0.61974,0.12314)--(0.61214,0.11561)--(0.60434,0.10912)--(0.59634,0.10352)--(0.58817,0.098719)--(0.57982,0.094603)--(0.57132,0.091084)--(0.56266,0.088081)--(0.55385,0.085525)--(0.54489,0.083352)--(0.5358,0.081509)--(0.52656,0.079949)--(0.51719,0.078632)--(0.50768,0.077522)--(0.49804,0.076588)--(0.48826,0.075806)--(0.47836,0.075151)--(0.46831,0.074606)--(0.45814,0.074152)--(0.44783,0.073776)--(0.43739,0.073466)--(0.42681,0.07321)--(0.4161,0.073001)--(0.40525,0.072831)--(0.39425,0.072692)--(0.38312,0.07258)--(0.37185,0.07249)--(0.36044,0.072418)--(0.34888,0.072361)--(0.33717,0.072317)--(0.32532,0.072281)--(0.31332,0.072254)--(0.30116,0.072233)--(0.28886,0.072217)--(0.2764,0.072205)--(0.26378,0.072196)--(0.251,0.07219)--(0.23806,0.072185)--(0.22496,0.072182)--(0.21169,0.072179)--(0.19826,0.072178)--(0.18466,0.072177)--(0.17089,0.072176)--(0.15694,0.072176)--(0.14282,0.072175)--(0.12852,0.072175)--(0.11404,0.072175)--(0.099379,0.072175)--(0.084533,0.072175)--(0.0695,0.072175)--(0.054279,0.072175)--(0.038866,0.072175)--(0.023259,0.072175)--(0.0074557,0.072175);

\node[scale=0.8,left] at (0.6827,0.7) {$\Mc_{v_1}$};
 
 \draw[ultra thick, dotted, cyan,->] 
(0.002,0.999)--(0.002,0.999)--(0.20109,0.99899)--(0.20614,0.99899)--(0.21132,0.99899)--(0.21664,0.99899)--(0.22208,0.99899)--(0.22767,0.99899)--(0.2334,0.99899)--(0.23927,0.99898)--(0.24529,0.99898)--(0.25146,0.99898)--(0.25779,0.99898)--(0.26428,0.99897)--(0.27093,0.99897)--(0.27776,0.99896)--(0.28475,0.99895)--(0.29192,0.99895)--(0.29927,0.99894)--(0.30681,0.99893)--(0.31454,0.99891)--(0.32246,0.9989)--(0.33058,0.99888)--(0.33891,0.99886)--(0.34745,0.99884)--(0.3562,0.9988)--(0.36518,0.99877)--(0.37438,0.99873)--(0.38381,0.99867)--(0.39347,0.99861)--(0.40338,0.99853)--(0.41354,0.99844)--(0.42395,0.99833)--(0.43462,0.99819)--(0.44555,0.99803)--(0.45676,0.99782)--(0.46824,0.99757)--(0.48,0.99726)--(0.49205,0.99688)--(0.50438,0.9964)--(0.51701,0.99579)--(0.52993,0.99504)--(0.54315,0.99408)--(0.55667,0.99287)--(0.57048,0.99131)--(0.58457,0.98931)--(0.59894,0.98672)--(0.61356,0.98334)--(0.62842,0.97889)--(0.64346,0.97301)--(0.65865,0.96518)--(0.67389,0.95466)--(0.68907,0.94047)--(0.70404,0.92125)--(0.71858,0.89518)--(0.73239,0.85999)--(0.74509,0.81304)--(0.75617,0.7518)--(0.76506,0.67483)--(0.77114,0.58319);

\draw[ultra thick, dotted, cyan,->] (0.77114,0.58319)--(0.77393,0.48157)--(0.77312,0.37795)--(0.76877,0.28123)--(0.76119,0.19823)--(0.75091,0.13204)--(0.73854,0.082431)--(0.72462,0.047347)--(0.70967,0.024235)--(0.69415,0.01071)--(0.67844,0.0042234)--(0.66275,0.0017086);

\node[scale=0.8,left] at (0.765,0.7) {$\Mc_{v_2}$};

\draw[ultra thick, dashed, teal,->] (0.0002,0.9999)--(0.32261,0.99987)--(0.32585,0.99987)--(0.32913,0.99987)--(0.33243,0.99987)--(0.33577,0.99987)--(0.33914,0.99987)--(0.34255,0.99986)--(0.34599,0.99986)--(0.34947,0.99986)--(0.35298,0.99985)--(0.35652,0.99985)--(0.3601,0.99985)--(0.36372,0.99984)--(0.36737,0.99984)--(0.37106,0.99984)--(0.37479,0.99983)--(0.37855,0.99983)--(0.38236,0.99982)--(0.3862,0.99981)--(0.39008,0.99981)--(0.39399,0.9998)--(0.39795,0.99979)--(0.40195,0.99979)--(0.40598,0.99978)--(0.41006,0.99977)--(0.41418,0.99976)--(0.41834,0.99975)--(0.42254,0.99974)--(0.42678,0.99972)--(0.43107,0.99971)--(0.43539,0.9997)--(0.43976,0.99968)--(0.44418,0.99966)--(0.44864,0.99964)--(0.45314,0.99962)--(0.45769,0.9996)--(0.46228,0.99958)--(0.46692,0.99955)--(0.47161,0.99952)--(0.47634,0.99949)--(0.48112,0.99946)--(0.48595,0.99942)--(0.49083,0.99938)--(0.49575,0.99934)--(0.50072,0.99929)--(0.50574,0.99924)--(0.51082,0.99918)--(0.51594,0.99912)--(0.52111,0.99905)--(0.52633,0.99897)--(0.53161,0.99889)--(0.53693,0.9988)--(0.54231,0.9987)--(0.54774,0.99859)--(0.55323,0.99846)--(0.55876,0.99833)--(0.56435,0.99818)--(0.57,0.99801)--(0.5757,0.99783)--(0.58145,0.99763)--(0.58726,0.9974)--(0.59312,0.99715)--(0.59904,0.99687)--(0.60501,0.99656)--(0.61104,0.99621)--(0.61713,0.99582)--(0.62327,0.99538)--(0.62946,0.99489)--(0.63571,0.99434)--(0.64201,0.99372)--(0.64837,0.99302)--(0.65478,0.99223)--(0.66124,0.99134)--(0.66776,0.99033)--(0.67432,0.98917)--(0.68094,0.98786)--(0.68759,0.98637)--(0.6943,0.98466)--(0.70104,0.9827)--(0.70782,0.98045)--(0.71464,0.97785)--(0.72148,0.97486)--(0.72835,0.9714)--(0.73523,0.96738)--(0.74212,0.9627)--(0.74901,0.95726)--(0.75588,0.95089)--(0.76272,0.94344)--(0.76952,0.93471)--(0.77625,0.92444)--(0.78288,0.91236)--(0.7894,0.89814)--(0.79576,0.8814)--(0.80192,0.86173)--(0.80784,0.83865)--(0.81347,0.81169)--(0.81873,0.78037)--(0.82357,0.74427)--(0.82791,0.70312)--(0.83167,0.65683)--(0.83478,0.60561)--(0.83716,0.55005)--(0.83876,0.49113);

\draw[ultra thick, dashed, teal,->] (0.83876,0.49113)--(0.83952,0.43018)--(0.83944,0.36878)--(0.83851,0.30861)--(0.83676,0.2512)--(0.83424,0.19787)--(0.83102,0.14956)--(0.82718,0.1069)--(0.82281,0.070274)--(0.81804,0.040081)--(0.813,0.017297)--(0.80793,0.0042407);

\node[scale=0.8,left] at (0.835,0.7) {$\Mc_{v_3}$};

\draw[ultra thick, dashdotted, orange,->] (0.002,0.999)--(0.40754,0.99899)--(0.41162,0.99899)--(0.41575,0.99899)--(0.41991,0.99899)--(0.42412,0.99899)--(0.42836,0.99899)--(0.43266,0.99899)--(0.43699,0.99899)--(0.44137,0.99899)--(0.44579,0.99899)--(0.45025,0.99899)--(0.45476,0.99899)--(0.45932,0.99899)--(0.46392,0.99899)--(0.46857,0.99899)--(0.47326,0.99898)--(0.47801,0.99898)--(0.4828,0.99898)--(0.48763,0.99898)--(0.49252,0.99898)--(0.49745,0.99898)--(0.50244,0.99897)--(0.50747,0.99897)--(0.51256,0.99897)--(0.5177,0.99897)--(0.52288,0.99896)--(0.52812,0.99896)--(0.53342,0.99896)--(0.53876,0.99895)--(0.54416,0.99895)--(0.54962,0.99895)--(0.55513,0.99894)--(0.56069,0.99894)--(0.56631,0.99893)--(0.57199,0.99892)--(0.57772,0.99891)--(0.58351,0.99891)--(0.58936,0.9989)--(0.59527,0.99889)--(0.60123,0.99887)--(0.60726,0.99886)--(0.61334,0.99885)--(0.61949,0.99883)--(0.6257,0.99881)--(0.63197,0.99879)--(0.63831,0.99877)--(0.6447,0.99874)--(0.65116,0.99871)--(0.65769,0.99868)--(0.66428,0.99864)--(0.67094,0.9986)--(0.67766,0.99855)--(0.68445,0.99849)--(0.69131,0.99843)--(0.69823,0.99836)--(0.70522,0.99827)--(0.71229,0.99817)--(0.71942,0.99806)--(0.72662,0.99792)--(0.73389,0.99777)--(0.74123,0.99759)--(0.74865,0.99737)--(0.75613,0.99712)--(0.76369,0.99681)--(0.77132,0.99645)--(0.77901,0.99602)--(0.78678,0.99549)--(0.79462,0.99485)--(0.80253,0.99407)--(0.8105,0.99311)--(0.81854,0.9919)--(0.82664,0.99039)--(0.8348,0.98848)--(0.84301,0.98603)--(0.85126,0.98286)--(0.85954,0.9787)--(0.86783,0.97316)--(0.87612,0.9657)--(0.88437,0.95547)--(0.89252,0.94124)--(0.90051,0.92111)--(0.90824,0.89226)--(0.91555,0.85049)--(0.92221,0.7901)--(0.9279,0.70439)--(0.93222,0.58844)--(0.93474,0.44463);
\draw[ultra thick, dashdotted, orange,->] 
(0.93474,0.44463)--(0.93516,0.28747)--(0.93347,0.13993)--(0.93009,0.025894);

\node[scale=0.8,left] at (0.9279,0.7) {$\Mc_{v_4}$};
\end{tikzpicture}
\caption{The ordering of orbits as stated in \Cref{theo:order} computed for the parameter set provided in \cite{eberl2017spatially}.}\label{fig:order}
\end{figure}
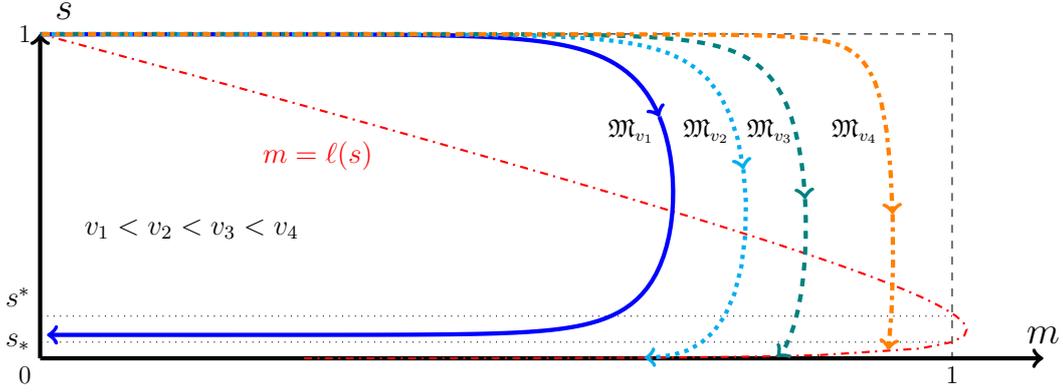

\begin{theorem}[The ordering of the orbits with respect to $v$] For $v>0$, let the function $\Mc_v:[s_{-\infty},1)\to [0,1]$ denote the $\Mc$-mapping introduced in \Cref{theo:01} that satisfies $\Mc_v(1)=0$. Then, for a fixed $s\in [s_{-\infty},1)$, $\Mc_v(s)$ varies continuously with $v$. Moreover,
\begin{enumerate}[label=(\roman*)]
\item $\Mc_v(s)$ is strictly increasing with $v$.
\item $\Mc_v(s)\to 0$ as $v\to 0$.
\item Let $\Gf(s)>0$ for all  $s\in (s_{-\infty},1)$. Then, $\Mc_v(s)\to 1$ as $v\to \infty$.
\end{enumerate}\label{theo:order}
\end{theorem}

Schematically the ordering of the orbits is shown in \Cref{fig:order}.
\begin{proof}
\emph{(i)} For a fixed $v>0$, let $\Mc^\e_v$ denote the $\Mc$-mapping discussed in \Cref{lemma:LowerBound} with $\Mc^\e_v(1)=\e>0$.  Let $0<v_1<v_2$.  Observe that 
$$
\frac{\dd }{\dd s}\Phi(\Mc^\e_{v_1}(1))=-\frac{v^2_1}{\g\, f(1)}\e>  -\frac{v^2_2}{\g\, f(1)}\,\e=\frac{\dd }{\dd s}\Phi(\Mc^\e_{v_2}(1)) ,
$$
 implying  $\Phi(\Mc^\e_{v_1}(s))< \Phi(\Mc^\e_{v_2}(s))$ in a left neighbourhood of $s=1$.
We show below that $\Phi(\Mc_{v_1}^{\e}(s))< \Phi(\Mc_{v_2}^{\e}(s))$ for all $s\in [s_{-\infty},1)$ provided $\Mc_{v_2}^{\e}(s)>0$. To show this, let us assume the contrary, i.e. $\Mc_{v_1}^\e(s_1)=\Mc_{v_2}^\e(s_1)>0$ for some $s_1\in [s_{-\infty},1)$ with $\Mc^\e_{v_1}(s)< \Mc^\e_{v_2}(s)$ for all $s\in (s_1,1)$.
Subtracting two versions of \eqref{eq:McEqB}, we obtain
\begin{align}
\frac{\dd }{\dd s} (\Phi(\Mc_{v_1}^\e(s))-\Phi(\Mc_{v_2}^\e(s)))&=\frac{1}{f(s)}\left [-\frac{v^2_1}{\g}(\Mc_{v_1}^\e-\Mc_{v_2}^\e) -\dfrac{v_2^2-v_1^2}{\g}[\ell(s)-\Mc_{v_2}^\e] \right ] \nonumber\\
=&-\frac{v^2_1}{\g\, f(s)}(\Mc_{v_1}^\e-\Mc_{v_2}^\e) -\left (1-\frac{v_1^2}{v_2^2}\right )\frac{\dd }{\dd s} \Phi(\Mc_{v_2}^\e(s)).
\end{align}
Integrating in $(s_1,1)$ we get
\begin{align*}
\Phi(\Mc_{v_2}^\e(s_1))-\Phi(\Mc_{v_1}^\e(s_1))= \frac{v^2_1}{\g}\int_{s_1}^1(\Mc_{v_2}^\e-\Mc_{v_1}^\e)\frac{\dd s}{f(s)}+ \left (1-\frac{v_1^2}{v_2^2}\right ) \Phi(\Mc_{v_2}^\e(s_1))>0, 
\end{align*}
thus contradicting our assumption. Hence, $\Phi(\Mc_{v_1}^{\e}(s))< \Phi(\Mc_{v_2}^{\e}(s))$ for all $s\in [s_{-\infty},1)$ and passing to the limit $\e\to 0$ we obtain
\begin{align}\label{eq:McMonotone}
\Phi(\Mc_{v_1}(s))\leq \Phi(\Mc_{v_2}(s)) \text{ for all } s\in [s_{-\infty},1).
\end{align}
To prove that $\Mc_{v}$ strictly increases with $v$, we observe integrating \eqref{eq:McEqB} in $(s,1)$  that
$$
\frac{\Phi(\Mc_{v_2}(s))}{v_2^2}=\frac{1}{\g}\int_s^1 [\Mc_{v_2}-\ell](\vr)\frac{\dd \vr}{f(\vr)}\overset{\eqref{eq:McMonotone}}\geq  \frac{1}{\g}\int_s^1 [\Mc_{v_1}-\ell](\vr)\frac{\dd \vr}{f(\vr)}=\frac{\Phi(\Mc_{v_1}(s))}{v_1^2},
$$
which yields the important inequality
\begin{align}
\label{eq:v1v2Ineq}
\Phi(\Mc_{v_2}(s))\geq \frac{v_2^2}{v_1^2} \Phi(\Mc_{v_1}(s))>\Phi(\Mc_{v_1}(s)).
\end{align}
This proves that $\Mc_v(s)$ is strictly increasing with respect to $v$.

\emph{(ii)} Integrating \eqref{eq:McEqB} in $(s,1)$ we have
$$
\Phi(\Mc_v(s))=\frac{v^2}{\g}\int_s^1  (\Mc_v(\vr)-\ell(\vr))\frac{\dd \vr}{f(\vr)}< \frac{v^2}{\g}\int_s^1 (1-\ell(\vr))\frac{\dd \vr}{f(\vr)}\overset{\eqref{eq:relationGell}}=\left (\frac{v}{\g}\right )^2 \Gf(s).
$$ 
This proves point (ii).

\emph{(iii)} We prove the statement first in $(s^*,1)$. Recall that $s^*\in (0,1)$ satisfies $\ell(s^*)=1$. Let $S^{\dagger}_v\in (s^*,1)$ be such that $\Mc_v(S^{\dagger}_v)=\ell(S^{\dagger}_v)$. The existence of $S^{\dagger}_v$ is guaranteed by \Cref{cor:orbit}. Also, point (i) implies that $S^{\dagger}_v$ decreases as $v$ increases. Let $S^\dagger= \lim\limits_{v\to \infty} S^{\dagger}_v$. Then \Cref{cor:orbit} yields again 
\begin{align}\label{eq:Sdagger}
\max_{s\in (s_*,1)}\Mc_v(s)=\ell(S^{\dagger}_v), \text{ and hence, } \sup_{v>0}\max_{s\in (s_*,1)}\Mc_v(s) = \ell(S^{\dagger}).
\end{align}
 Let $S^\dagger >s^*$, yielding $\ell(S^{\dagger})<1$.  Then using \eqref{eq:v1v2Ineq} one has for some $v_1=v>0$ and $v_2=v^2$ that 
$$
\Phi(\Mc_{v_2}(S^{\dagger}_{v_1}))\geq \frac{v_2^2}{v_1^2} \Phi(\Mc_{v_1}(S^{\dagger}_{v_1}))= v^2 \Phi(\ell(S^{\dagger}_{v}))\to \infty, \text{ as } v\to \infty,
$$
which contradicts the upper bound in \eqref{eq:Sdagger} since $\ell(S^\dagger)<1$. Hence, $\lim\limits_{v\to \infty} S^{\dagger}_v=s^*$. 

Now let $s\in (s^*,1)$. Then, there exists $v_1>0$ such that $S^{\dagger}_{v_1}\geq s$. Using \eqref{eq:v1v2Ineq} for $v>v_1$,
\begin{align}
\Phi(\Mc_v(s))\geq \frac{v^2}{v_1^2} \Phi(\Mc_{v_1}(s))\geq \frac{v^2}{v_1^2} \Phi(\ell(s))\to \infty \text{ as } v\to \infty.
\end{align}
Hence, $\Mc_v(s)\to 1$ for all $s\in (s^*,1)$.

Now, we extend the result to $(s_{-\infty},s^*)$. Set $\check{s}_0=s^*$ and let $\check{s}_k\in [s_{-\infty},\check{s}_{k-1}]$, $k\in \N$, be recursively defined by the formula
$$
\int_{\check{s}_k}^{\check{s}_{k-1}}  \frac{\ell(\vr)}{f(\vr)}\, \dd \vr=  \int^{1}_{\check{s}_{k-1}}  \frac{1-\ell(\vr)}{f(\vr)}\, \dd \vr \overset{\eqref{eq:DefG}}= \g^{-1} \Gf(\check{s}_{k-1})>0.
$$
If such $\check{s}_k$ does not exist in $[s_{-\infty},\check{s}_{k-1}]$ then we set $\check{s}_k=s_{-\infty}$. Assume that $\lim_{v\to \infty}\Mc_v(s)\to 1$ holds for all $s\in (\check{s}_{k-1},1)$. This is certainly true for $k=1$. Then we show that it holds for all $s\in (\check{s}_{k},1)$. For all $s\in (\check{s}_k,\check{s}_{k-1})$ one has
\begin{align}\label{eq:Stilde}
\Phi(\Mc_v(s))&= \frac{v^2}{\g} \int_s^1 \frac{1}{f}(\Mc_v-\ell)= \frac{v^2}{\g} \left [\int_{\check{s}_{k-1}}^1 \frac{1}{f}(\Mc_v-\ell) + \int^{\check{s}_{k-1}}_{s} \frac{1}{f} (\Mc_v-\ell)\right ]\nonumber\\
&\geq \frac{v^2}{\g} \left [\int_{\check{s}_{k-1}}^1 \frac{1}{f} (\Mc_v-\ell) - \int^{\check{s}_{k-1}}_{s} \frac{\ell}{f} \right ].
\end{align}
The integral $\int_{\check{s}_{k-1}}^1 \frac{1}{f}(\Mc_v-\ell) \to \int_{\check{s}_{k-1}}^1 \frac{1}{f} (1-\ell) \overset{\eqref{eq:relationGell}}=\g^{-1}\Gf(\check{s}_{k-1})$ as $v\to \infty$. Thus $\int_{\check{s}_{k-1}}^1 \frac{1}{f}(\Mc_v-\ell) - \int^{\check{s}_{k-1}}_{s} \frac{\ell}{f} >0$ for large $v$ and $s>\check{s}_k$. Hence, passing $v\to \infty$ in \eqref{eq:Stilde} one gets, $\Mc_v(s)\to 1$ for all $s> \check{s}_k$, thus proving the statement.

\emph{Continuity:} Finally, we prove that $\lim\limits_{v\to v_0} \Mc_v(s)=\Mc_{v_0}(s) $ for any $v_0>0$. Let us handle the  case $v\searrow v_0$ first. From point (i), $\lim\limits_{v\searrow v_0} \Mc_v(s)\geq \Mc_{v_0}(s)$. Let us assume 
$$\lim\limits_{v\searrow v_0} \Mc_v(s)=:\bar{\Mc}_{v_0}(s)> \Mc_{v_0}(s) \text{ for some } s\in (s_{-\infty},1).$$ 
Just as in the proof of \Cref{theo:01}, $\bar{\Mc}_{v_0}$ satisfies \eqref{eq:McEq} with $v=v_0$. Since $\Mc^\e_{v_0}(s)\searrow \Mc_{v_0}(s)$ as $\e\searrow 0$ ($\Mc^{\e}_{v_0}$ as defined in \Cref{lemma:LowerBound} for $v=v_0$), one can choose $\e>0$ small enough such that $\Mc_{v_0}(s)<\Mc_{v_0}^\e(s)< \bar{\Mc}_{v_0}(s)$. Then $\Mc^{\e}_{v_0}$  and $\bar{\Mc}_{v_0}$ both satisfy \eqref{eq:McEq} and
$$
(\Mc^\e_{v_0}- \bar{\Mc}_{v_0})(1)=\e>0, \text{ whereas } (\Mc^\e_{v_0}- \bar{\Mc}_{v_0})(s)<0,
$$
implying that they intersect at some intermediate point. Since both $\Mc^{\e}_{v_0}$  and $\bar{\Mc}_{v_0}$ correspond to orbits in the phase--plane with the same $v=v_0$, and the orbits cannot intersect in $\{m>0\}$ we have our contradiction. 

The proof of the case $v\nearrow v_0$ follows from \Cref{pros:Unique} since $\bar{\Mc}_{v_0}=\lim_{v\nearrow v_0} \Mc_v< \Mc_{v_0}$ would imply that there are two mappings, $\Mc_{v_0}$ and $\bar{\Mc}_{v_0}$, which satisfy $\Mc(1)=0$ with a corresponding $\z(1)=0$, thus contradicting \Cref{pros:Unique}.
\end{proof}

From the proof of \Cref{theo:order} (ii), we obtain the following property:
\begin{corollary}[Bounds on $\Mc_v$]  For $v>0$, let the function $\Mc_v:[s_{-\infty},1]\to [0,1]$ denote the $\Mc$-mapping introduced in \Cref{theo:01} that satisfies $\Mc_v(1)=0$. Then, for all $v_0<v$ one has
$$
\frac{v^2}{v_0^2}\Phi(\Mc_{v_0}(s))\leq \Phi(\Mc_{v}(s))\leq \left (\frac{v}{\g}\right )^2 \Gf(s).
$$\label{cor:Bounds}
\end{corollary}

\subsection{Existence/non-existence of travelling waves}\label{sec:ProofTheo}

\begin{proof}[\textbf{Proof of \Cref{theo:main}}]
\Cref{theo:01} and \Cref{pros:Unique} prove the existence of a unique orbit $(M_v,S_v)$ satisfying \eqref{eq:DS} and connecting with $(0,1)$ at $\xi=0$. \Cref{cor:orbit} shows that the orbit crosses the line $m=\ell(s)$ for some $s\in (s^*,1)$ and \Cref{lemma:orbit_exists} shows that it enters $\calR=[0,1)\times (s_{-\infty},1]$ through either $\{m=0, s\in [s_{-\infty},1]\}$ or through $\{s=s_{-\infty}\}$. \Cref{theo:order} proves that the orbit varies continuously with $v$, and for large $v$, the orbit enters through $s=s_{-\infty}$. Hence, it remains to be shown that for small $v$, the orbit enters through  $\{m=0, s\in [s_{-\infty},1]\}$. This will prove, by continuity, the existence of $v=\bar{v}>0$ such that the corresponding orbit connects with $(0,s_{-\infty})$ which is the intersection point of the two segments.

We show that for any $\hat{s}\in [s_{-\infty},1)$ there exists a corresponding $\hat{v}>0$  such that 
$$
\Mc_{\hat{v}}(\hat{s})=0 \text{ for all } v\leq \hat{v}.
$$
Assume that no such $\hat{v}>0$  exists. Then $\Mc_v(\hat{s})>0$ for all $v>0$. Integrating \eqref{eq:McEqB} in $(\hat{s},1)$ one then has
$$
0<\Phi(\Mc_v(\hat{s}))= \frac{v^2}{\g}\int_{\hat{s}}^1 (\Mc_v(\vr)-\ell(\vr))\,\frac{\dd \vr}{f(\vr)}.
$$
However, \Cref{cor:Bounds} implies that $\int_{\hat{s}}^1 (\Mc_v(\vr)-\ell(\vr))\,\frac{\dd \vr}{f(\vr)}<0$ for small enough $v>0$, since $\Mc_v(\vr)\searrow 0$ uniformly as $v\searrow 0$. This is a contradiction to $\Phi(\Mc_v(\hat{s}))>0$. Hence, the hypothesized $\hat{v}$ exists. Setting $\hat{s}=s_{-\infty}$ proves \Cref{theo:main}. 
\end{proof}
The profiles of $(M,S)$ as functions of $\xi\leq 0$ are shown in \Cref{fig:TWappearance}. An interesting feature of this TW is that the profile of $M$ has a sharp front at $\xi=0$, whereas, it has a diffused tail at the rear. The TW is in fact a travelling pulse since it connects an equilibrium state with  $M=0$ to another equilibrium state with $M=0$, in contrast to TWs for nonlinear diffusion problems with Fischer type source terms, see \cite{efendiev2009classification,de1991travelling,de1998travelling}.

\begin{figure}
\begin{subfigure}{.5\textwidth}
\includegraphics[scale=.45]{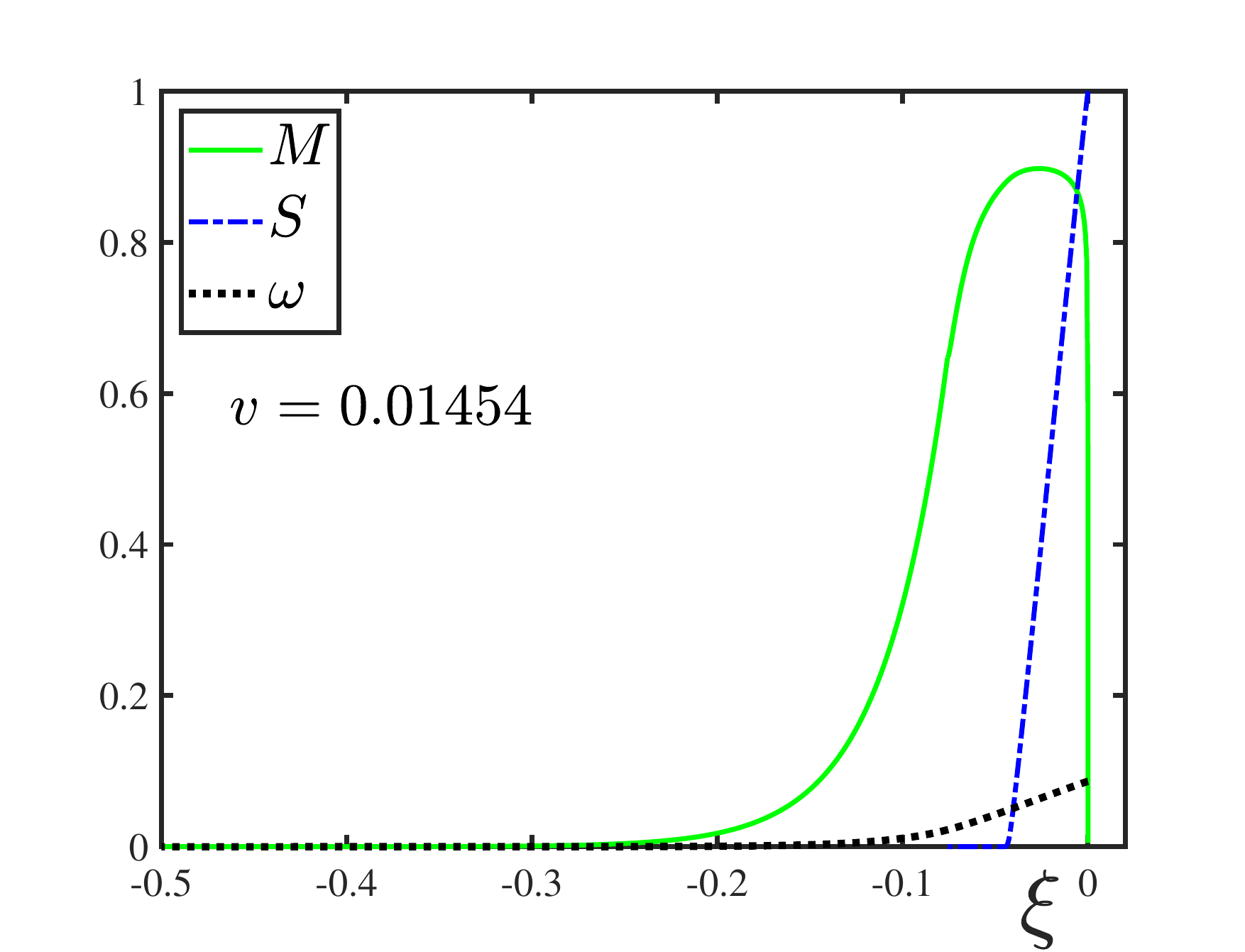}
\end{subfigure}
\begin{subfigure}{.32\textwidth}
\includegraphics[scale=.45]{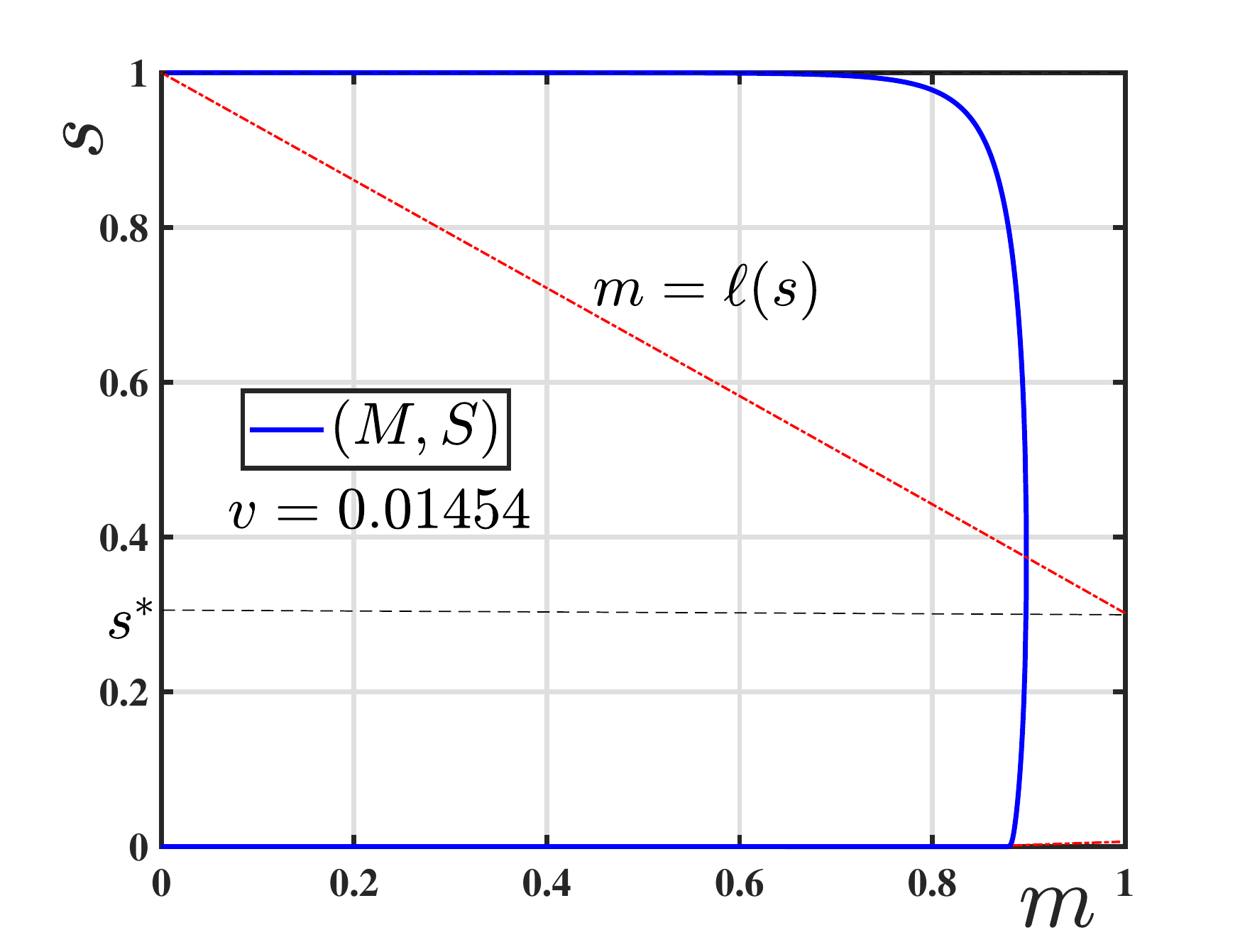}
\end{subfigure}
\caption{The TW profile for the default parameter set given in \Cref{tab:default_parameters} obtained using Algorithm \ref{algo:TW}. In the (left) plot the graphs of $M(\xi)$, $S(\xi)$ and $\om(\xi)=\int_{-\infty}^\xi M$ are shown. The (right) plot shows the orbit $(M,S)$ in the phase--plane (note that $s_{-\infty}\approx 10^{-60}$ in this case). The corresponding wave-speed is $v=0.01454$.} \label{fig:TWappearance}
\end{figure}

\begin{remark}[Consistency of the TW solution] Let $(M,S)$ be the TW solution described in \Cref{theo:main} and let $\Mc$ denote its $\Mc$-mapping. Then $\ell(s)\slash \Mc\to 0$ as $s\nearrow 1$ from \eqref{eq:dMdSinft}, and consequently recalling the TW Assumption \ref{ass:TW} one has
\begin{align}\label{eq:particle_speed}
\begin{cases}-D(M)\p_x M=-D(M)\dz M=v(M-\ell(S))\to 0,\\
-\tfrac{D(M)}{M}\p_x M=-\tfrac{D(M)}{M}\dz M=v\left (1-\tfrac{\ell(S)}{M}\right) \to v ,
\end{cases}
\quad  \text{ as } x-vt=\xi\nearrow 0. 
\end{align}
Here, the term $-D(M)\p_x M$ represents the flux and the term $-\tfrac{D(M)}{M}\p_x M$ is commonly referred to as the \textbf{particle speed} in the literature of the porous medium equation, see \cite[Chapter 19]{vazquez2007porous}. The fact that the flux vanishes at $\xi=0$ and the particle velocity is $v$ implies that the TW is physically consistent with the model \eqref{eq:PDE}. 
\end{remark}

\begin{remark}[Regularity of the sharp front]
From the limits in \eqref{eq:particle_speed}, one can also deduce the regularity of $M(\xi)$ as $\xi\nearrow 0$. Taking the expression \eqref{eq:DefDa} for $D$, one has in the limit $\xi\nearrow 0$, observing that $M(\xi)\searrow 0$, that $ \d\,M^{a-1}\,\dz M\sim -v$, or 
\begin{align}
M(\xi)\sim \tfrac{a\,v}{\d}|\xi|^{\frac{1}{a}}, \text{ for } \xi \text{ in a left neighbourhood of } 0.
\end{align}
Consequently, a sharp front exists at $\xi=0$, however, the TW profile still remains H\"older continuous.
\end{remark}

Finally, the techniques developed in this section also allow us to give necessary conditions for existence of the TWs.

\begin{proposition}[Non-existence of travelling wave solutions] Let \eqref{eq:parameter} be satisfied and $\Gf(\bar{s})<0$ for some $\bar{s}\in (s_{-\infty},1)$. Then there exists no TW orbit $(M,S)$ satisfying both  \eqref{eq:DS} and \eqref{eq:BC}. Similarly, if $\l\geq 1$  then no TW exists.\label{pros:NonExistence}
\end{proposition}
\begin{proof}
For $\Gf(\bar{s})<0$, assume that such a solution exists for some $v>0$. Then, integrating \eqref{eq:McEq} in $(\bar{s},1)$ one has
$$
0< \frac{\g}{v^2}\Phi(\Mc_v(\bar{s}))=\int_{\bar{s}}^1 (\Mc_v(\vr)-\ell(\vr))\,\frac{\dd\vr}{f(\vr)}< \int_{\bar{s}}^1 (1-\ell)\,\frac{\dd\vr}{f(\vr)}= \g^{-1}\Gf(\bar{s})<0,
$$
which is a contradiction. For the case $\l\geq 1$, we have from \eqref{eq:DefEll} that $\ell'(s)\overset{\eqref{eq:DefF}}=\frac{1}{\g f(s)}[f(s)-\l]\overset{\ref{prop:f}}\leq 0$ and hence no $s_{-\infty}\in (0,1)$ satisfying \eqref{eq:Sinf} exists, thus giving the result.
\end{proof}

\section{Linear stability of the travelling waves in two space dimensions}\label{sec:stability}
In this section, we analyse the linear stability of the TW solutions in two space dimensions. It was observed numerically in \cite{eberl2017spatially,hughes2022RODE} that the TWs are stable with respect to random perturbations, and that solutions resulting from arbitrary initial data converge to a TW solution in a long rectangular domain. The stability of the TWs is further investigated numerically in \Cref{sec:num}. In this section, we use asymtotic expansions to prove linear $L^1$-stability of the TWs under certain (suitable) assumptions.

\subsection{Asymtotic expansion and linearization}
Let $(M^{\co{0}},S^{\co{0}})$ denote the TWs introduced in \Cref{theo:main}.  For a given $L>0$, we consider the infinite-cylinder $\Om:=\R\times (0,L)$  as the domain. Let $(x,y)\in \Om$ denote the spatial coordinates. Then, we consider the  two dimensional version of \eqref{eq:PDE} in $\Om$:
\begin{subequations}\label{eq:PDE2D}
\begin{align}
&\p_t M=\p_{x} [D(M)\,\p_x M] + \p_{y}[D(M)\,\p_y M] + \left (f(S)-\lambda\right )M,\\
& \p_t S= -\g\,f(S)\,M.
\end{align}
\end{subequations}
We consider the following boundary conditions for \eqref{eq:PDE2D}:
\begin{subequations}\label{eq:PDE2DBC}
\begin{align}
&M(\pm \infty,y,t)=0, \;\; (D(M)\,\p_x M)(\pm \infty,y,t)=0 &\text{ for all } y\in (0,L) \text{ and } t>0,\label{eq:PDE2DBCa}\\
&(D(M)\,\p_y M)(x,0,t)=(D(M)\,\p_y M)(x,L,t)=0 &\text{ for all } x\in \R \text{ and } t>0.\label{eq:PDE2DBCb}
\end{align}
\end{subequations}
The boundary condition \eqref{eq:PDE2DBCa} generalizes \eqref{eq:PDEC} to two dimensions, whereas, \eqref{eq:PDE2DBCb} introduces homogeneous Neumann conditions for the lateral boundaries.
As initial condition we choose for an arbitrarily small $\e>0$,
\begin{subequations}\label{eq:StabilityIniCond}
\begin{align}
\begin{cases}
M(x,y,0)&=M^{\co{0}}(x) + \e \sum_{n=0}^\infty M^{\co{1}}_{n,0}(x) \cos\left (\Lambda_n\, y\right ),\\[.5em]
S(x,y,0)&=S^{\co{0}}(x) + \e  \sum_{n=0}^\infty S^{\co{1}}_{n,0}(x) \cos\left (\Lambda_n\, y\right ),
\end{cases} \text{ where } \Lambda_n:=\tfrac{2\pi n}{L},
\end{align}
and the functions $M^{\co{1}}_{n,0},\; S^{\co{1}}_{n,0}\in L^1(\R^-)$ are bounded, smooth, and satisfy 
\begin{align}
M^{\co{1}}_{n,0}(0)= S^{\co{1}}_{n,0}(0)=0, \text{ and } \lim\limits_{x\searrow -\infty} M^{\co{1}}_{n,0}(x)=\lim\limits_{x\searrow -\infty} S^{\co{1}}_{n,0}(x)=0 \text{ for all }  n\in \N_0.
\end{align}
\end{subequations} 

\begin{remark}[Generality of the initial condition]
The initial condition \eqref{eq:StabilityIniCond} can be generalized to include also $\sin(\Lambda_n y)$ components which then covers all smooth initial conditions having periodic boundaries at $y\in\{0,\,L\}$ (by Fourier series expansion). The main result of this section (\Cref{theo:stability}) remains unchanged provided that $\sin(\Lambda_n y)$ contributions are added in \eqref{eq:M1S1fourier}.  For simplicity, we have used only the $\cos(\Lambda_n y)$ components here. Observe that the initial condition \eqref{eq:StabilityIniCond} is consistent with the zero Neumann conditions in \eqref{eq:PDE2DBCb}.
\end{remark}

We assume that the solution of \eqref{eq:PDE} can be described in this case by the asymptotic expansion
\begin{align}\label{eq:Asymptotic}
\begin{cases}
M(x,y,t)=M^{\co{0}}(x-vt) + \e M^{\co{1}}(x,y,t)+\dots,\\
S(x,y,t)=S^{\co{0}}(x-vt) + \e S^{\co{1}}(x,y,t)+\dots,
\end{cases} 
\end{align}
where $v>0$ represents the wave speed, and $M^{\co{1}}_{n}$, $S^{\co{1}}_n$ are differentiable functions such that the boundary conditions \eqref{eq:PDE2DBC} and initial conditions \eqref{eq:StabilityIniCond} are satisfied.
Then by inserting the expansion \eqref{eq:Asymptotic} in \eqref{eq:PDE2D} and equating the $\e$-order terms one obtains the system
\begin{subequations}\label{eq:stabPDEoriginal}
\begin{align}
\p_t M^{\co{1}}&= \p_{xx}(D(M^{\co{0}})M^{\co{1}})+ \p_{yy}(D(M^{\co{0}})M^{\co{1}}) \nonumber\\
&\quad + (f(S^{\co{0}})-\l) M^{\co{1}} + f'(S^{\co{0}})  M^{\co{0}} S^{\co{1}},\\
\p_t S^{\co{1}}&=-\g[f(S^{\co{0}}) M^{\co{1}} + f'(S^{\co{0}})  M^{\co{0}} S^{\co{1}}].
\end{align}
\end{subequations}
Since the TW coordinate, $\xi=x-vt$, is a more natural space-coordinate compared to $x$ to analyse this problem, we use the coordinate transform $(x,y,t)\mapsto (\xi,y',t')$, where
\begin{align}
\xi=x-vt,\; y'=y \text{ and } t'=t, \text{ giving } \p_t=\p_{t'} -v\p_{\xi},\; \p_x=\p_{\xi} \text{ and } \p_{y'}=\p_y.
\end{align}
 System \eqref{eq:stabPDEoriginal} is modified accordingly. For simplicity, referring to $(y',t')$  as $(y,t)$, we then have the following problem
\begin{subequations}\label{eq:stabPDEmain}
\begin{align}
\p_t M^{\co{1}}-v\,\p_{\xi} M^{\co{1}}&= \p_{\xi\xi}(D(M^{\co{0}})M^{\co{1}}) + \p_{yy}(D(M^{\co{0}})M^{\co{1}})\nonumber\\
& \quad + (f(S^{\co{0}})-\l) M^{\co{1}} + f'(S^{\co{0}})  M^{\co{0}} S^{\co{1}},\\
\p_t S^{\co{1}}-v\,\p_{\xi} S^{\co{1}}&=-\g[f(S^{\co{0}}) M^{\co{1}} + f'(S^{\co{0}})  M^{\co{0}} S^{\co{1}}].
\end{align}
\end{subequations}
Due to the form of the initial condition prescribed, we look for  $M^{\co{1}}$ and $S^{\co{1}}$ of the following form
\begin{align}\label{eq:M1S1fourier}
\begin{cases}
M^{\co{1}}(x,y,t)&=\sum\limits_{n=0}^\infty M^{\co{1}}_{n}(\xi,t) \cos\left (\Lambda_n\, y\right ),\\[.5em]
S^{\co{1}}(x,y,t)&=\sum\limits_{n=0}^\infty S^{\co{1}}_{n}(\xi,t) \cos\left (\Lambda_n\,y\right ), 
\end{cases} \text{ with } \Lambda_n=\tfrac{2\pi n}{L}.
\end{align}
Substituting \eqref{eq:M1S1fourier} in \eqref{eq:stabPDEmain}, observing that $(M^{\co{0}},S^{\co{0}})$ only depends on $\xi$, and equating the $\cos\left (\Lambda_n y\right )$ terms, we have for each $n\in \N_0$ the system
\begin{subequations}\label{eq:stabPDE}
\begin{align}
\p_t M^{\co{1}}_n-v\,\p_{\xi} M^{\co{1}}_n&= \p_{\xi\xi}(D(M^{\co{0}})M^{\co{1}}_n) - \Lambda_n^2 M^{\co{1}}_n\nonumber\\
& \quad + (f(S^{\co{0}})-\l) M^{\co{1}}_n + f'(S^{\co{0}})  M^{\co{0}} S^{\co{1}}_n,\label{eq:stabPDEa}\\
\p_t S^{\co{1}}_n-v\,\p_{\xi} S^{\co{1}}_n&=-\g[f(S^{\co{0}}) M^{\co{1}}_n + f'(S^{\co{0}})  M^{\co{0}} S^{\co{1}}_n]\label{eq:stabPDEb}.
\end{align}
\end{subequations}
In order to satisfy the boundary and initial conditions \eqref{eq:PDE2DBC}--\eqref{eq:StabilityIniCond} for each $n\in \N_0$ one must have
\begin{subequations}\label{eq:stabBCpde}
\begin{align}
M^{\co{1}}_n(0,t)&= S^{\co{1}}_n(0,t) =0 &\text{ for all } t>0, \\[.5em]
\lim\limits_{\xi\searrow -\infty} M^{\co{1}}_n(\xi)&=\lim\limits_{\xi\searrow -\infty}  S^{\co{1}}_n(\xi)= 0,&\\
M^{\co{1}}_n(\xi,0)&=M^{\co{1}}_{n,0}(\xi),\quad  S^{\co{1}}_n(\xi,0)=S^{\co{1}}_{n,0}(\xi)  &\text{ for all } \xi<0.
\end{align}
\end{subequations}

\begin{remark}[Choice of boundary conditions \eqref{eq:stabBCpde}] The boundary conditions at $\xi=0$ prescribed in \eqref{eq:stabBCpde} imply that $\p_\xi ( D(M^{\co{0}}) M^{\co{1}}_n)=D(M^{\co{0}}) \p_\xi M^{\co{1}}_n + \p_\xi D(M^{\co{0}}) M^{\co{1}}_n= 0$ at $\xi=0$ since $\p_{\xi} D(M^{\co{0}})$ is bounded. Thus, the flux is zero at $\xi=0$ which ensures that $(M^{\co{1}}_n(t),S^{\co{1}}_n(t))$, as a solution to \eqref{eq:stabPDE}, can be extended to $\R^+$ by setting
$$
M^{\co{1}}_n=S^{\co{1}}_n=0, \text{ for all } \xi>0 \text{ and } t>0. 
$$
Hence, \eqref{eq:stabPDE} is satisfied for all $\xi\in \R$ and $t>0$.
Finally, the boundary conditions $M^{\co{1}}_n=S^{\co{1}}_n=0$ at $\xi=-\infty$ is consistent with the initial conditions in \eqref{eq:StabilityIniCond} and make it possible to have absolutely integrable solutions  $(M^{\co{1}}_n,S^{\co{1}}_n)$.
\end{remark}

\subsection{Stability in $L^1$-norm}
\begin{theorem}[Stability of the travelling waves]\label{theo:stability}
Let $(M^{\co{0}},S^{\co{0}}):\R^-\to [0,1]^2$ be a travelling wave solution satisfying \eqref{eq:DS}--\eqref{eq:BC} with a given wave-speed $v>0$.
Assume that for all $n\in \N_0$, a continuously differentiable and absolutely integrable solution $(M^{\co{1}}_n,S^{\co{1}}_n):\R^-\times[0,\infty)  \to \R^2$ exists of the problem \eqref{eq:stabPDE} which satisfies the initial and boundary conditions \eqref{eq:stabBCpde}. Then, for any given $t>0$, one has 
\begin{align}
&\int_{\R^-}\left [|M^{\co{1}}_n(t)| + \tfrac{1}{\g}|S^{\co{1}}_n(t)|\right ] + \left (\l + \Lambda_n^2\right )\int_0^t \int_{\R^-} |M^{\co{1}}_n|  \nonumber\\
& \leq\int_{\R^-}\left [|M^{\co{1}}_{n,0}| + \tfrac{1}{\g}|S^{\co{1}}_{n,0}|\right ] + \int_0^t \int_{\{M^{\co{1}}_n\cdot S^{\co{1}}_n\leq 0\}} [2 f(S^{\co{0}})|M^{\co{1}}_n|-f'(S^{\co{0}})M^{\co{0}}| S^{\co{1}}_n|].\label{eq:L1contraction}
\end{align}
Consequently, if $\Lambda_n^2>  2 f(1)-\l>0$, then $\int_0^\infty \int_{\R^-} |M^{\co{1}}_n|<\infty $, and 
\begin{align}
&\int_{\R^-}\left [|M^{\co{1}}_n(t)| + \tfrac{1}{\g}|S^{\co{1}}_n(t)|\right ] \text{ is strictly decreasing  with respect to } t>0.\label{eq:Decay}
\end{align}
\end{theorem}

\begin{remark}[Stability of the travelling waves] $$ \text{If }\; L< \frac{2\pi}{\sqrt{2 f(1)-\l}}, \text{ then \eqref{eq:Decay} holds for all } n\in \N,$$ thus proving stability of the TWs in two dimensions. However, for $n=0$, $\Lambda_n=0$ and the stability of the TW is not guaranteed. In practice, this means that perturbations to the TW that have fast transverse variations decay.  The case  $n=0$ represents no perturbation in the transverse direction but only in the direction of the TW. Unfortunately, only conditional stability can be proven for this case using our analysis. Increased stability due to transverse variations, similar to \Cref{theo:stability}, has been studied earlier, for example in  \cite{van2019stability}. Using numerical simulations we show in \Cref{sec:transient} that the TWs are also stable in terms of large longitudinal perturbations. 
\end{remark}

\begin{proof}
The proof uses the well-known $L^1$-contraction principle. We reproduce a formal version here for the sake of brevity. 

Let $\sgn_\e:\R\to [-1,1]$ denote a regularised version of the signum function for $\e>0$ with $\calU_\e$ as its primitive. More precisely,
\begin{align}
\sgn_\e(u):=\begin{cases}
1 &\text{ if } u>\e,\\
u\slash \e  &\text{ if } |u|\leq \e,\\
-1 &\text{ if } u<-\e,\\
\end{cases} \text{ and }\; \calU_\e(u)=\int_0^u \sgn_\e.
\end{align}
Note the following properties of these functions for future use
\begin{subequations}
\begin{align}\label{eq:propSignA}
&|u\, {\sgn_\e}'(u)|\begin{cases}
<1 \text{ for } |u|\leq\e,\\
=0 \text{ for } |u|>\e,\\
 \end{cases} \text{ and }\; \calU_\e(u)\geq 0 \text{ with equality only for } u=0, \\
& \sgn_\e(u)\to \sgn(u), \; u \,\sgn_\e(u)\to |u|, \text{ and } \calU_\e(u)\to |u| \text{ pointwise as } \e\searrow 0.\label{eq:propSignB}
\end{align}
\end{subequations}
We use $\sgn_\e(M^{\co{1}}_n)$ as a test function in \eqref{eq:stabPDEa}. Multiplying \eqref{eq:stabPDEa} by $\sgn_\e(M^{\co{1}}_n)$ and integrating in $\R^-$ one has
term by term
\begin{subequations}\label{eq:L1contrM}
\begin{align}
&\int_{\R^-} \sgn_\e(M^{\co{1}}_n)\,\p_t M^{\co{1}}_n=\p_t \left ( \int_{\R^-} \calU_\e(M^{\co{1}}_n)\right )\overset{\eqref{eq:propSignB}}\to \p_t \left ( \int_{\R^-} |M^{\co{1}}_n|\right ) \text{ as } \e\searrow 0,\\
-v&\int_{\R^-} \sgn_\e(M^{\co{1}}_n)\,\p_\xi M^{\co{1}}_n = -v\int_{\R^-} \p_\xi\,\calU_\e(M^{\co{1}}_n) = v[\calU_\e(M^{\co{1}}_n(-\infty,t))- \calU_\e(0)]\overset{\eqref{eq:stabBCpde}} = 0.
\end{align}
From the second order term, one has using integration by parts, and the boundary conditions in \eqref{eq:stabBCpde} that
\begin{align}
&\int_{\R^-} \sgn_\e(M^{\co{1}}_n)\, \p_{\xi\xi}(D(M^{\co{0}})M^{\co{1}}_n)= - \int_{\R^-} \p_{\xi}\left (\sgn_\e(M^{\co{1}}_n)\right  )\p_{\xi}(D(M^{\co{0}})M^{\co{1}}_n)\nonumber\\
&=- \int_{\R^-} {\sgn_\e}'(M^{\co{1}}_n) D(M^{\co{0}})|\p_{\xi} M^{\co{1}}_n|^2 - \int_{\R^-} {\sgn_\e}'(M^{\co{1}}_n)\,M^{\co{1}}_n\, \p_\xi \left (D(M^{\co{0}})\right )\p_{\xi} M^{\co{1}}_n\nonumber\\
&\overset{\eqref{eq:propSignA}}\leq  \int_{\{|M^{\co{1}}_n|<\e\}} |\p_\xi D(M^{\co{0}})\,\p_{\xi} M^{\co{1}}_n|\to 0 \qquad \text{ as } \e\searrow 0.
\end{align}
Finally, for the source terms, one has as $\e\to 0$,
\begin{align}
&\int_{\R^-} \sgn_\e(M^{\co{1}}_n)\,[(f(S^{\co{0}})-\l-\Lambda_n^2) M^{\co{1}}_n + f'(S^{\co{0}})  M^{\co{0}} S^{\co{1}}_n]\nonumber\\
&\overset{\eqref{eq:propSignB}}\to - (\l+\Lambda_n^2)\int_{\R^-} |M^{\co{1}}_n| + \int_{\R^-} \sgn(M^{\co{1}}_n)\,[f(S^{\co{0}}) M^{\co{1}}_n + f'(S^{\co{0}})  M^{\co{0}} S^{\co{1}}_n].
\end{align}
\end{subequations}
Similarly, multiplying \eqref{eq:stabPDEb} by $\sgn_\e(S^{\co{1}}_n)$, integrating in $\R^-$ and following the steps of \eqref{eq:L1contrM} one has 
\begin{align}\label{eq:L1contrS}
\p_t \left ( \int_{\R^-} |S^{\co{1}}_n|\right )\leq -\g \int_{\R^-} \sgn(S^{\co{1}}_n)\,[f(S^{\co{0}}) M^{\co{1}}_n + f'(S^{\co{0}})  M^{\co{0}} S^{\co{1}}_n].
\end{align}
Adding \eqref{eq:L1contrM}--\eqref{eq:L1contrS} one thus obtains 
\begin{align}\label{eq:L1contr}
&\p_t \left ( \int_{\R^-}\left [|M^{\co{1}}_n| + \tfrac{1}{\g}|S^{\co{1}}_n|\right ]\right ) + \left (\l + \Lambda_n^2\right ) \int_{\R^-} |M^{\co{1}}_n|\nonumber\\
&\leq \int_{\R^-} (\sgn(M^{\co{1}}_n)-\sgn(S^{\co{1}}_n))\,[f(S^{\co{0}}) M^{\co{1}}_n + f'(S^{\co{0}})  M^{\co{0}} S^{\co{1}}_n].
\end{align}
Note that if both $ M^{\co{1}}_n,\, S^{\co{1}}_n>0$ or $ M^{\co{1}}_n,\, S^{\co{1}}_n<0$ then the right hand side vanishes. On the other hand, if $ M^{\co{1}}_n> 0$ and $ S^{\co{1}}_n<0$ then $\sgn(M^{\co{1}}_n)-\sgn(S^{\co{1}}_n)=2$ and $f(S^{\co{0}}) M^{\co{1}}_n + f'(S^{\co{0}})  M^{\co{0}} S^{\co{1}}_n= f(S^{\co{0}}) |M^{\co{1}}_n|-  f'(S^{\co{0}})  M^{\co{0}} |S^{\co{1}}_n|$. Hence, the integrand on the right hand side becomes 
$$ 2 (f(S^{\co{0}}) |M^{\co{1}}_n|-  f'(S^{\co{0}})  M^{\co{0}} |S^{\co{1}}_n|)<  2f(S^{\co{0}}) |M^{\co{1}}_n|-  f'(S^{\co{0}})  M^{\co{0}} |S^{\co{1}}_n|.$$ 
By symmetry, we have the same inequality when $ M^{\co{1}}_n< 0$ and $ S^{\co{1}}_n>0$. Including the trivial cases of $ M^{\co{1}}_n= 0$ and/or $ S^{\co{1}}_n=0$ which in both cases yield
$$(\sgn(M^{\co{1}}_n)-\sgn(S^{\co{1}}_n))\,[f(S^{\co{0}}) M^{\co{1}}_n + f'(S^{\co{0}})  M^{\co{0}} S^{\co{1}}_n]\leq  2f(S^{\co{0}}) |M^{\co{1}}_n|-  f'(S^{\co{0}})  M^{\co{0}} |S^{\co{1}}_n|,$$
 we have  \eqref{eq:L1contraction} by integrating \eqref{eq:L1contr} in time.

Since $f(S^{\co{0}})< f(1)$ in $\R^-$, we have $\l+\Lambda_n^2>2f(S^{\co{0}})$ in \eqref{eq:Decay}. This implies that there exists a constant $c_1>0$ such that for any $0<t_1<t_2$,
$$
\int_{\R^-}\left [|M^{\co{1}}_n(t_2)| + \tfrac{1}{\g}|S^{\co{1}}_n(t_2)|\right ] + c_1\int_{t_1}^{t_2} \int_{\R^-} |M^{\co{1}}_n|   \leq\int_{\R^-}\left [|M^{\co{1}}_{n}(t_1)| + \tfrac{1}{\g}|S^{\co{1}}_{n}(t_1)|\right ].
$$
Hence, $\int_{\R^-}\left [|M^{\co{1}}_n(t)| + \tfrac{1}{\g}|S^{\co{1}}_n(t)|\right ] $ is a decreasing function having a limit, and $\int_0^t \int_{\R^-} |M^{\co{1}}_n| $ is bounded uniformly for all $t>0$. 
\end{proof}

\section{Numerical results}\label{sec:num}
In this section, we verify the analytical predictions of  \Cref{sec:TW existence,sec:stability} numerically by computing solutions of the PDE systems \eqref{eq:PDE} and \eqref{eq:PDE2D}. It is shown that for an arbitrary initial condition, the PDE solutions indeed converge to a profile moving with constant speed both in one and two space dimensions, which further shows numerically the stability of the TWs. The TWs are also  obtained directly by solving \eqref{eq:DS} and finding the correct wave-speed by a bisection algorithm. The TWs produced by the PDE simulations and the bisection algorithm are shown to coincide. The algorithm is then used to further investigate the parametric dependence of the TW profiles nd wave-speed for a greater range of parameters. Finally, a numerical continuation approach is  considered which enables us to study the limiting cases for which the assumptions in \Cref{theo:main} are satisfied.

\subsection{PDE simulations}\label{sec:PDEscheme}
Here, we solve the PDE systems \eqref{eq:PDE} (one space dimension) and \eqref{eq:PDE2D} (two space dimensions) with expressions \eqref{eq:DefD} for $D,\,f$, on finite rectangular domains using homogeneous Neumann boundary conditions. The backward Euler method with uniform time-step $\D t>0$ is used for the temporal discretization.  A uniform square grid with mesh-size $\D x>0$ is used to discretize the domain. The PDEs are solved using the standard two-point flux approximation finite volume method. The solution is approximated at the centres of each grid cell. The diffusion across the cell interfaces are approximated via arithmetic averaging as described in \cite{eberl2007finite}. For time integration we use the trapezoidal method. The large system of arithmetic equations generated by this approach is solved via a fixed point iteration scheme described in \cite{hughes2022RODE}. The default set of simulation parameters is given in Table \ref{tab:default_parameters}.  These parameters will be used throughout this section unless stated otherwise.

\begin{table}[h!]
\centering
\begin{tabular}{|llll|}
\hline
Model parameter            & Symbol         & Value     & Reference \\ \hline
Motility coefficient in \eqref{eq:DefD} & $\delta$ & $10^{-6}$ & \cite{eberl2017spatially}\\
Diffusion exponent 1 in \eqref{eq:DefD} & $a$ & 4.0 & \cite{eberl2017spatially}\\
Diffusion exponent 2 in \eqref{eq:DefD}& $b$ & 4.0 & \cite{eberl2017spatially} \\
Half saturation concentration in \eqref{eq:DefD} & $\kappa$ & 0.01 & \cite{eberl2017spatially}\\
Maximum consumption rate & $\gamma$ & 0.4 & \cite{eberl2017spatially}\\ 
Cell-loss rate & $\lambda$ & 0.42 & \cite{eberl2017spatially}\\
\hline
\end{tabular}
\caption{Default model parameters (their names refer to the context of cellulolytic biofilm models) used for the simulations of \eqref{eq:PDE}, \eqref{eq:PDE2D} and \eqref{eq:DS}. The functions $D$ and $f$ are as in \eqref{eq:DefD}. All values in the table are dimensionless.} \label{tab:default_parameters}
\end{table}

\subsubsection{One dimensional results: transience and stability}\label{sec:transient}
In this case, our spatial domain is $(0,H)$ for $H>0$. The default initial conditions for the system are for $h\in (0,1)$ and $d\in (0,H)$,
\begin{align} \label{eq:1D_init}
M(x,0)=\begin{cases}
h\left (1-\tfrac{x^4}{d^4}\right ) & \text{if } x\leq d,\\
0 & \text{otherwise},
\end{cases} \quad
S(x,0)\equiv 1.
\end{align}
The numerical parameters used for the simulation are
\begin{align}\label{eq:1D_NumPar}
H=1,\;\; \D x=2^{-N}, \;\;  \D t=10^{-2}\,(2^9\Delta x)^2,\;\; d=\tfrac{5}{127},\;\; h=0.1.
\end{align}
In the above, $N$ is an integer between 9 and 16. A grid independence study is done in Appendix \ref{sec:grid_refinement} and based on the result the default value of $N=14$ is chosen. 

\begin{figure}[h!]
\centering
\includegraphics[width=.32\textwidth]{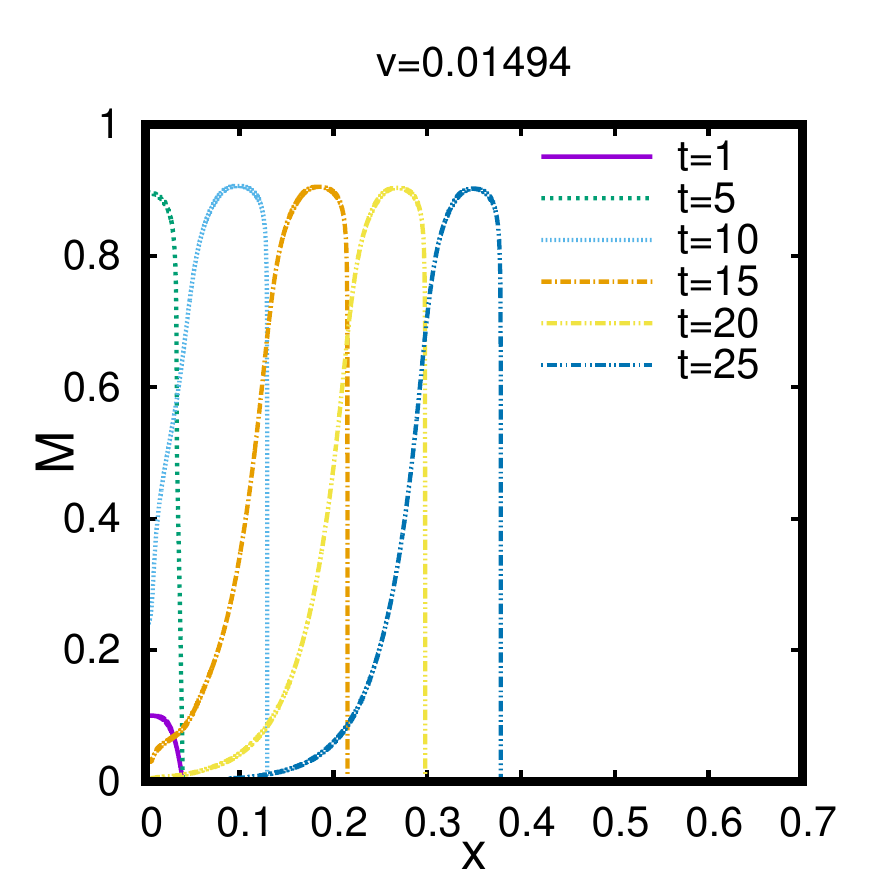}
 \includegraphics[width=.32\textwidth]{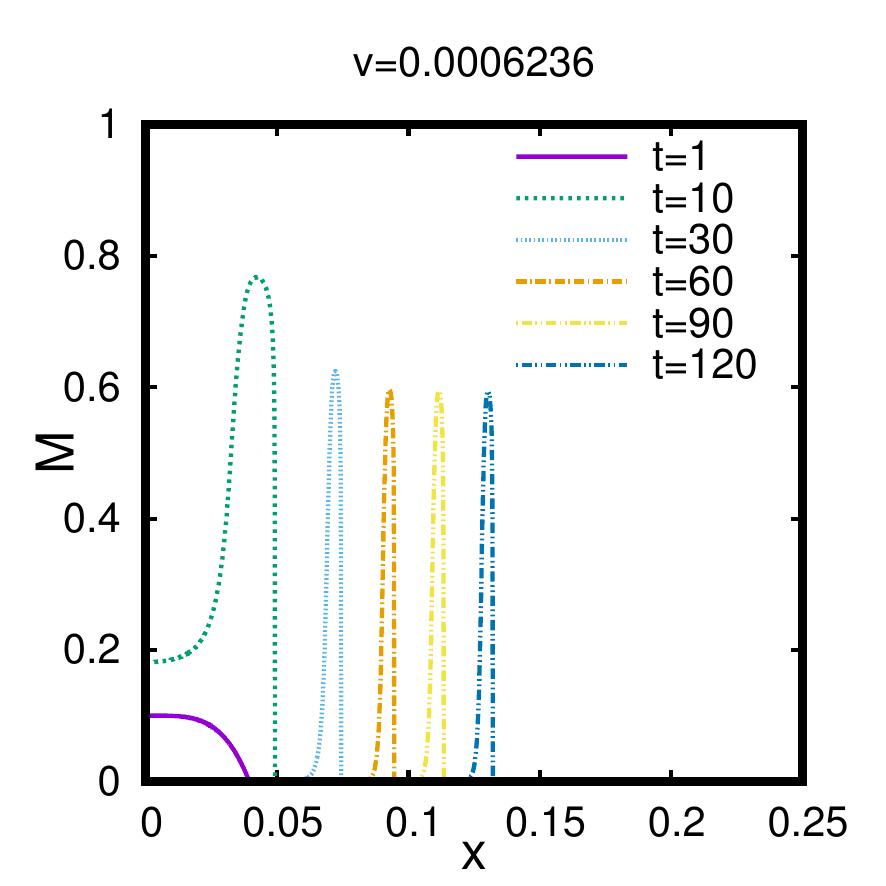}
 \includegraphics[width=.32\textwidth]{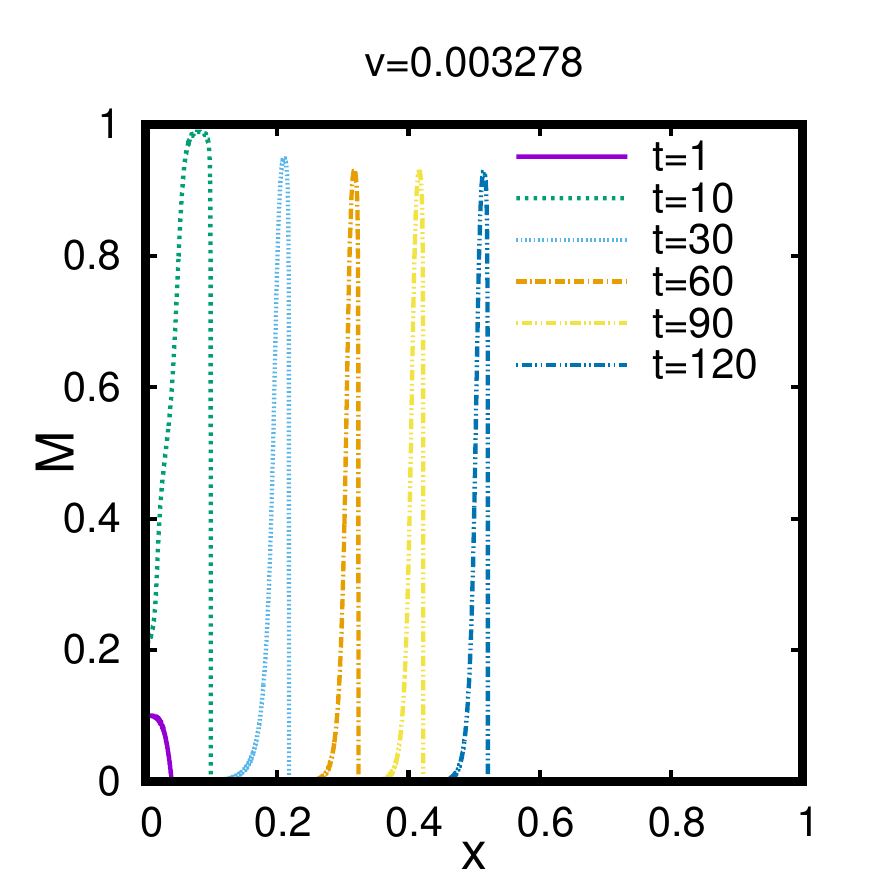} 
\caption{Transient behaviour of the $M$-profiles obtained from the PDE simulations. The parameters are taken from \Cref{tab:default_parameters} except in (left) $\l=0.42$, (center) $\l=0.60$, and (right) $a=b=2$.}\label{fig:transient}
\end{figure}

Observe that the initial condition \eqref{eq:1D_init} is arbitrary and has no relation to the TWs. Nevertheless, the numerical solutions develop into profiles that move with a constant speed. This is shown in \Cref{fig:transient} using three sets of parameters. To estimate the wave-speed of the developed profile, we calculate the interface of the biomass wave by taking the largest $x$-coordinate such that $M(x,t)>10^{-2}$. We use $10^{-2}$ as an approximation for $0$ to avoid any numerical noise generated due to the degeneracy. Once we have the wave interface, we can estimate the wave speed by fitting a linear function through the data points corresponding to the wave interface. This is done using the built-in function $\textsf{fit}$ from GNUPLOT which gives the wave-speed. \Cref{fig:wavespeed_calc} (left) illustrates the process for the parameter set in \Cref{tab:default_parameters}.  To calculate the wave profile, we plot the solution $M(x,t)$ after the profiles have developed. To verify that a TW solution exists, we plot the wave profiles at various time points $t_i$ and horizontally translate the waves by $v(t_0-t_i)$. If the wave profiles coincide, then we conclude that a TW solution exists. Unless otherwise stated, all the PDE-simulations are verified to permit a TW solution. 

\begin{figure}[h!]
\centering
\begin{subfigure}{.48\textwidth}
\includegraphics[scale=.55]{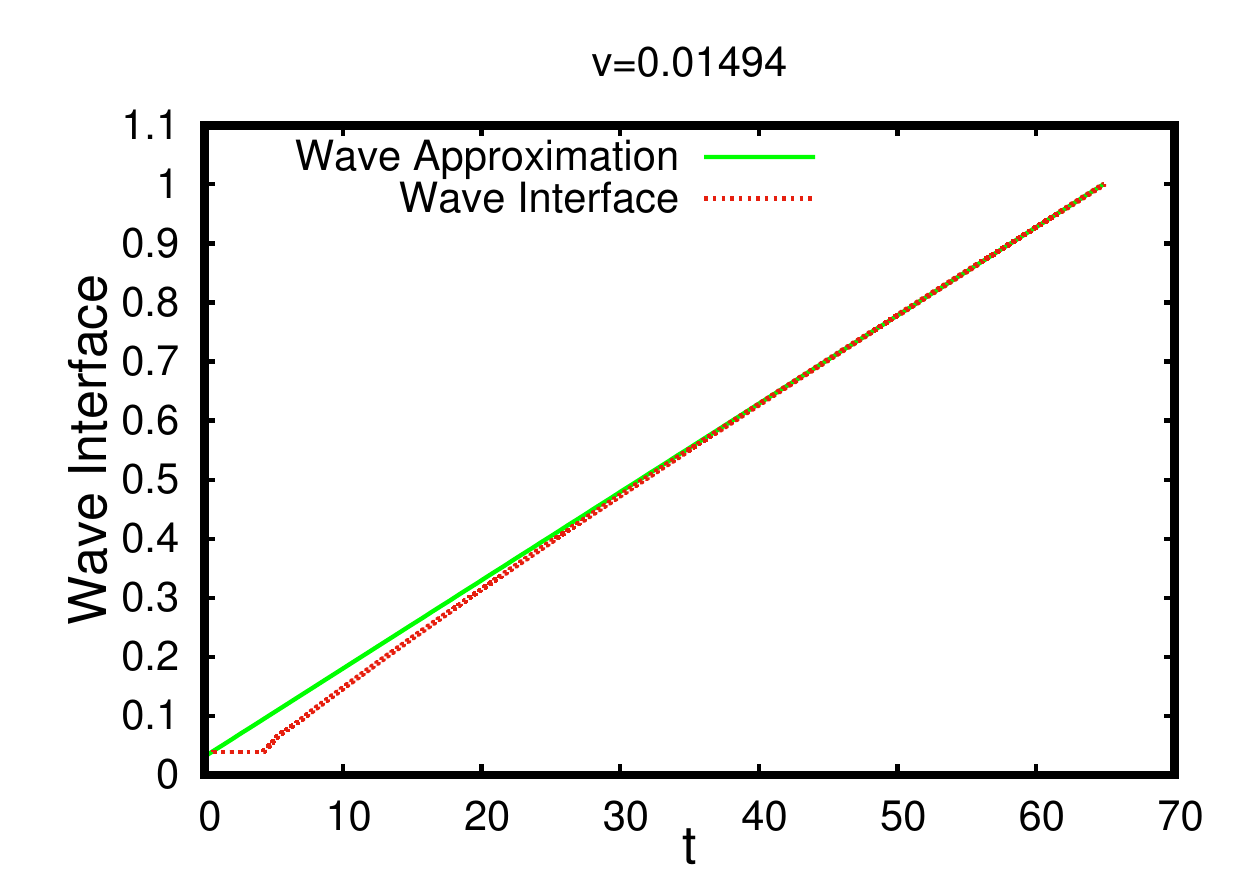} 
\end{subfigure}
\begin{subfigure}{.48\textwidth}
\includegraphics[scale=.55]{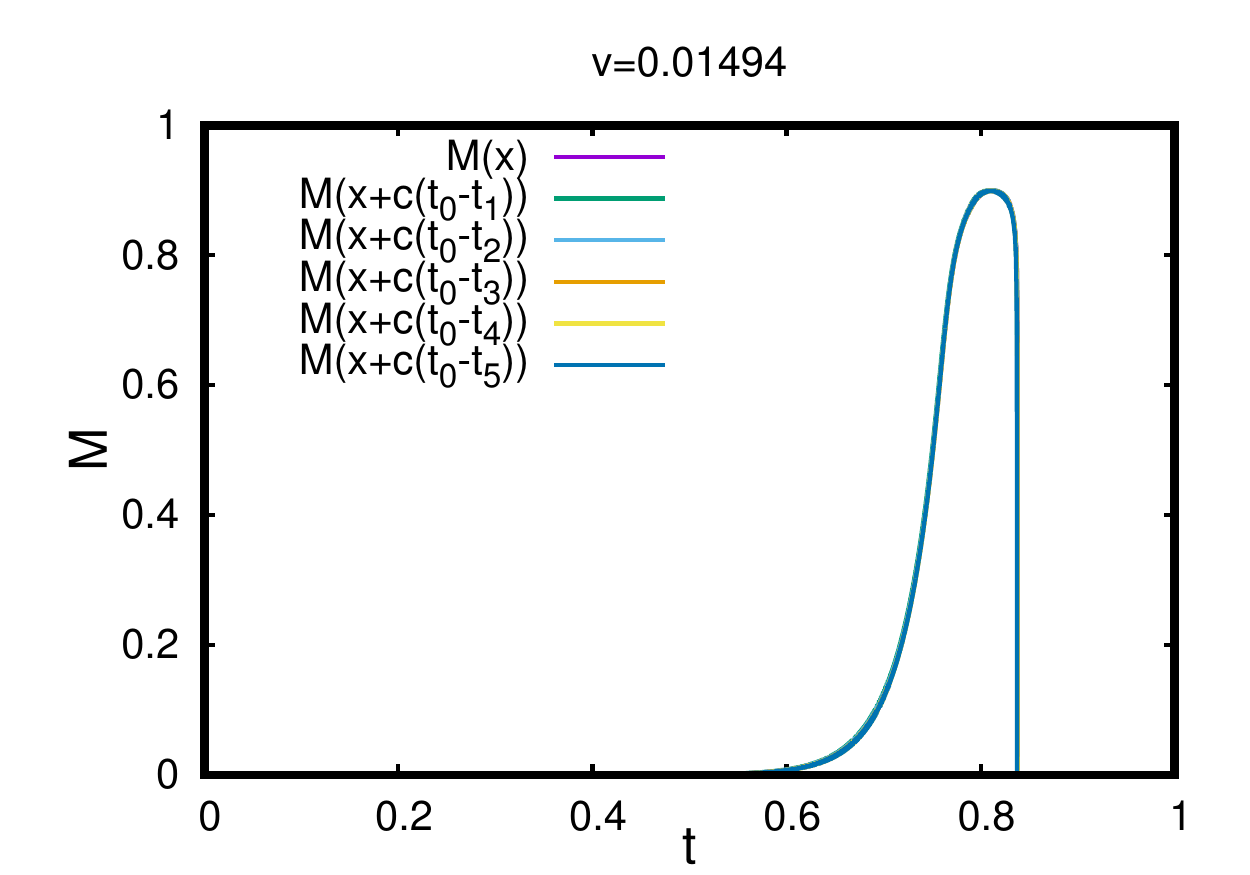} 
\end{subfigure}
\caption{(left) Wave-speed calculation for the default parameters in \Cref{tab:default_parameters} using linear fitting. The resulting wave-speed estimated is $v=0.01494$. (right) The $M$-profiles horizontally shifted by $v(t_0-t_i)$ for various time points $t_i\in \{56,57,58,59,60,64\}$.}\label{fig:wavespeed_calc}
\end{figure}

\subsubsection{Two dimensional results:  transience and stability}
For the two dimensional results our spatial domain is the rectangle $[0,H]\times [0,L]$. The initial condition prescribed is
\begin{align} \label{eq:2D_init}
M(x,y,0)=\begin{cases}
h\left (1-\tfrac{x^4}{d^4}\right )\left (1+ \tfrac{1}{5}\cos\left (\frac{2\pi y}{L}\right )\right ) & \text{if } x\leq d,\\
0 & \text{otherwise},
\end{cases} \quad
S(x,y,0)\equiv 1.
\end{align}
The numerical parameters for this case are
\begin{align}
H=2,\;\; L=1,\;\; \D x=2^{-9}, \;\; \D t=0.01,\;\; d=\tfrac{5}{127},\;\; h=0.1.
\end{align}

As before, the initial condition in this case is arbitrary and deviates largely from any TW profile both in terms of its $x$ and $y$ variations. As such, it does not satisfy the restrictions imposed in \Cref{sec:stability} and \Cref{theo:stability} for proving linear stability.  Nevertheless, \Cref{fig:2D_stability} shows that the numerical solution still develops slowly into a planar front that moves with a constant speed. This speed, estimated in \Cref{fig:wavespeed2D}, is very close to the wave-speed computed at the same level of  discretization, i.e.,  $\D x=2^{-9}$, for the one dimensional case, see \Cref{tab:speeds} of  Appendix \ref{sec:grid_refinement}.

\begin{figure}[h!]
\centering
\begin{subfigure}{.48\textwidth}
\includegraphics[width=\textwidth]{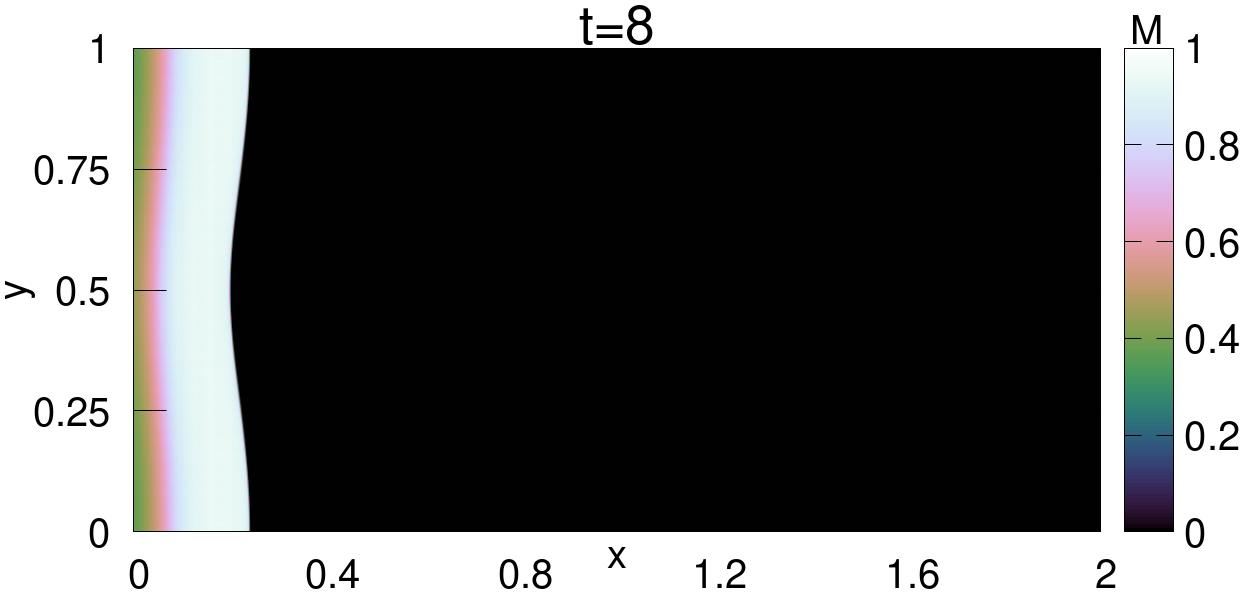} 
\end{subfigure}
\begin{subfigure}{.48\textwidth}
\includegraphics[width=\textwidth]{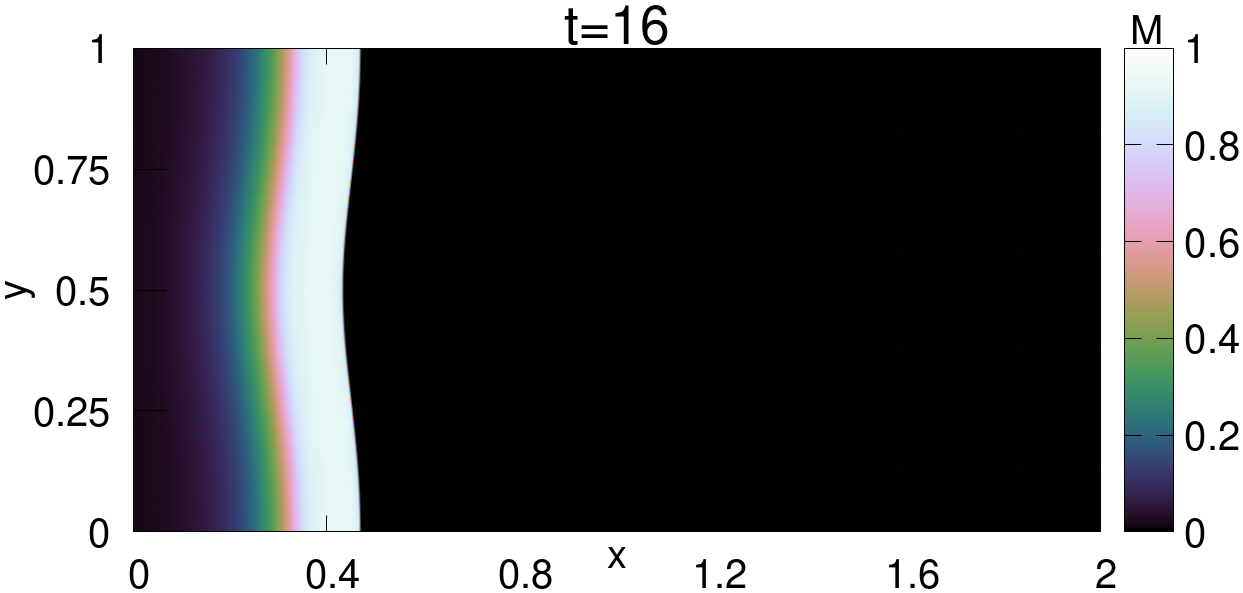} 
\end{subfigure}
\begin{subfigure}{.48\textwidth}
\includegraphics[width=\textwidth]{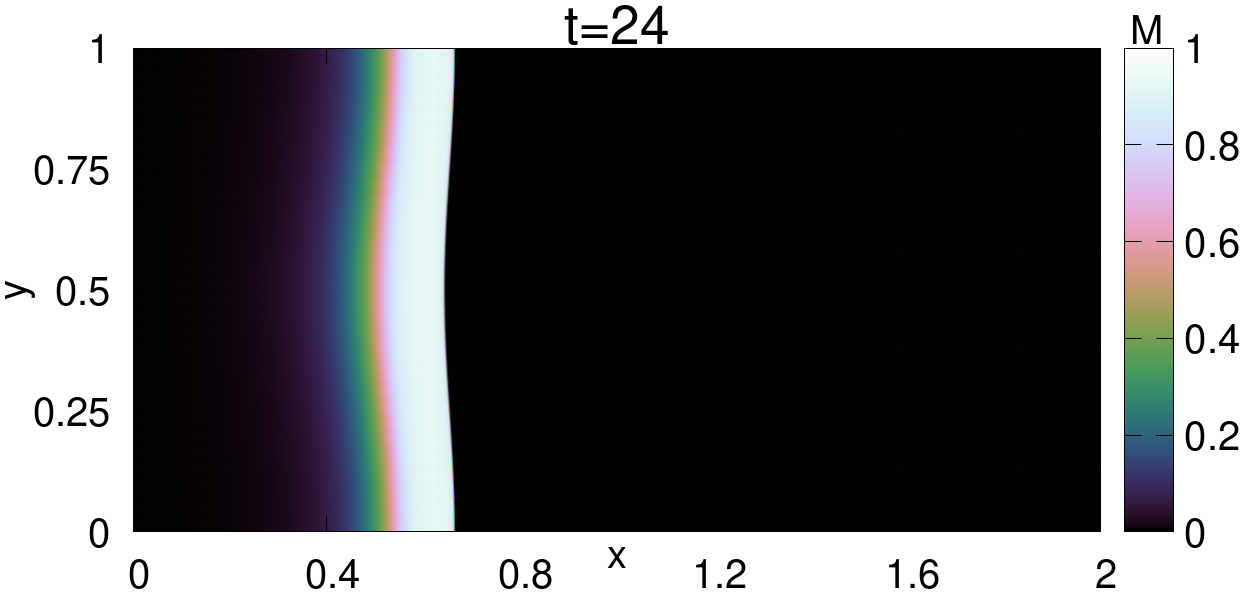} 
\end{subfigure}
\begin{subfigure}{.48\textwidth}
\includegraphics[width=\textwidth]{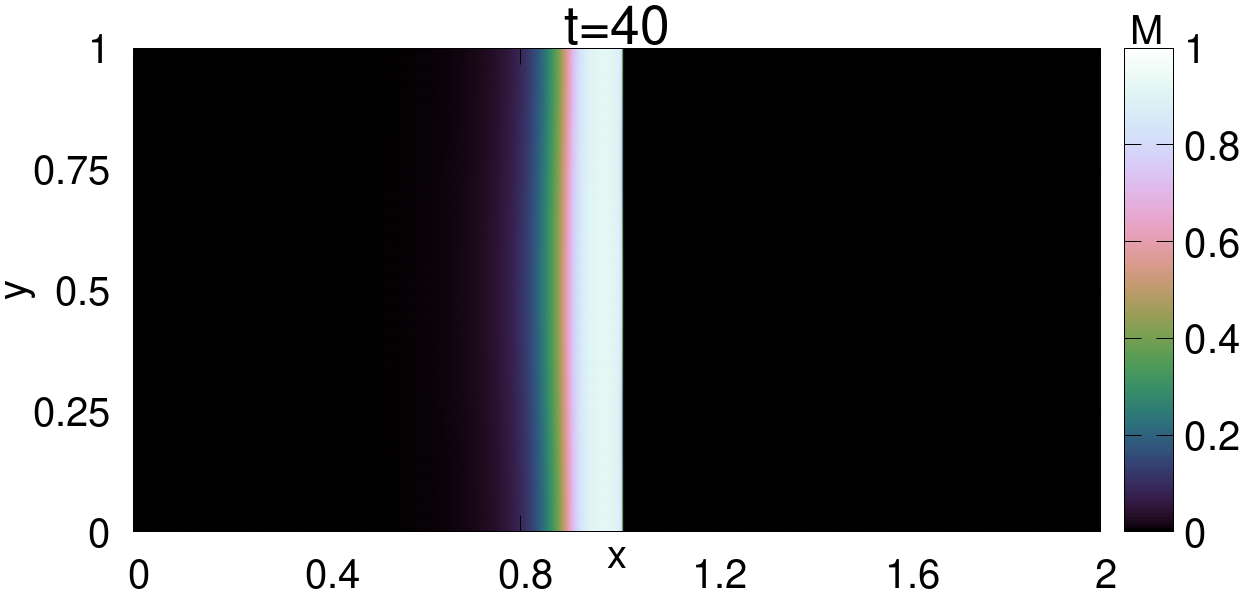} 
\end{subfigure}
\caption{The transient behaviour of the numerical solution of \eqref{eq:PDE2D} subjected to the initial condition \eqref{eq:2D_init}. The parameter values are taken from \Cref{tab:default_parameters}. }\label{fig:2D_stability}
\end{figure}

\begin{figure}[h!]
\centering
\includegraphics[scale=.55]{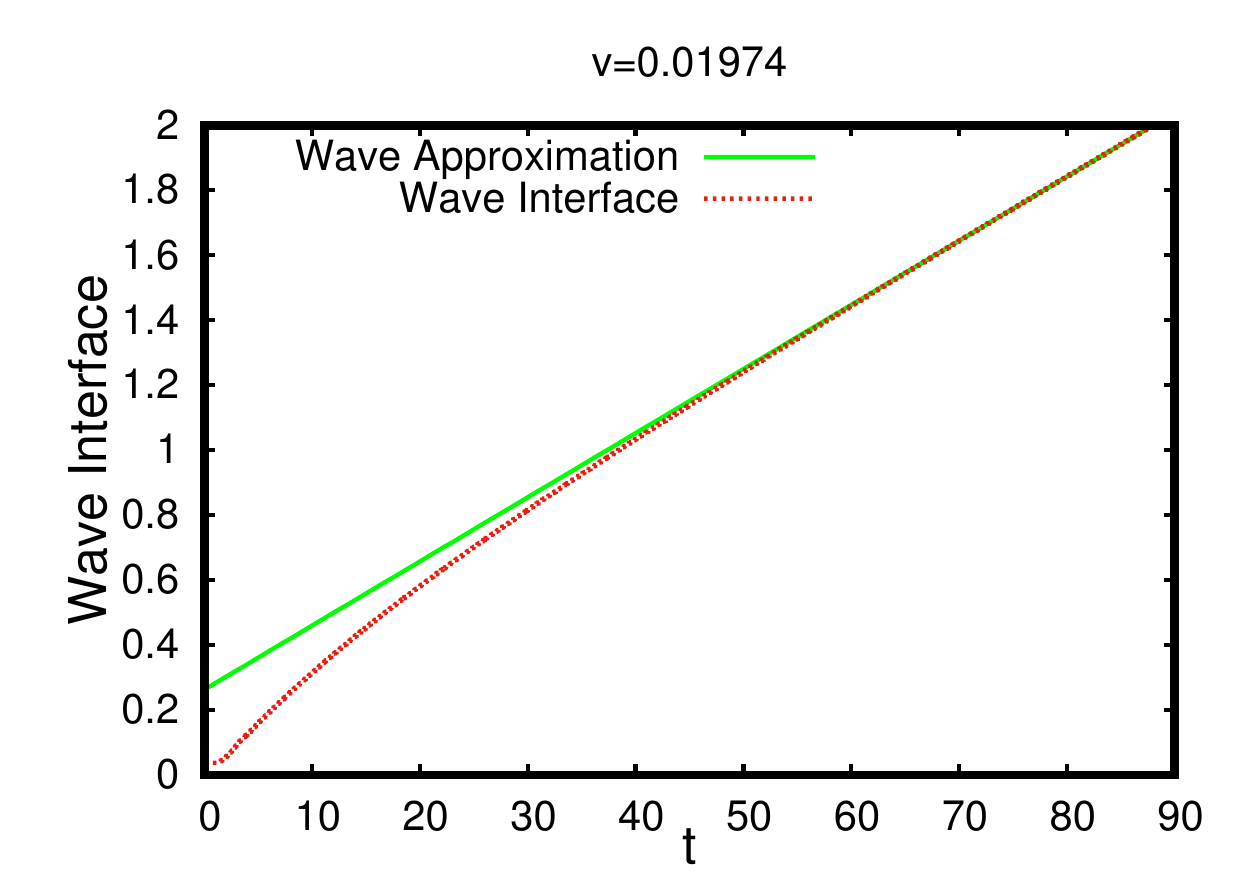}
\caption{Wave-speed estimation for the two dimensional simulation presented in \Cref{fig:2D_stability}.}\label{fig:wavespeed2D}
\end{figure}

\subsection{ODE simulations}\label{sec:ODEscheme}
Upon verifying that the solutions indeed converge to TW solutions, we device a more direct and faster method to compute the TW profiles.

\subsubsection{TW algorithm}
To directly find the TW profile and the corresponding wave-speed $v>0$ for a given set of parameters, we use a semi-analytic approach that utilizes the monotonicity property described in \Cref{theo:order} with respect to $v>0$. We solve the dynamical system \eqref{eq:GE} for a given $v>0$, except we solve the negation of the vector field, i.e.
\begin{subequations}\label{eq:GE-}
\begin{align}
&\dt M^{v}= -v\,[\ell(S^v)- M^v],
\label{eq:GEm-}\\[.5em]
& \dt S^v= -\tfrac{\g}{v} {f(S^v) M^v D(M^v)},\label{eq:GEs-}
\end{align}
\end{subequations}
with $D$ and $f$ as in \eqref{eq:DefD}. The above system is solved numerically for $\t>0$ using the 4$^{\rm th}$-order Runge-Kutta method \cite{suli2003numerical} and the following initial condition is used for the computation:
\begin{align}\label{eq:algoBC}
(M^v,S^v)(0)=(\e,1), &\text{ for } \e=10^{-3}.
\end{align}
This avoids the problem of starting from the degenerate equilibrium point $(0,1)$ and gives a close approximation to the TW profile as indicated by \Cref{theo:01}. 

Then, the \textbf{bisection method} is used to determine the wave-speed for which $(M^v,S^v)$ connects to the other equilibrium point $(0,s_{-\infty})$. Let $\bar{v}>0$ be large enough so that $(M^{\bar{v}},S^{\bar{v}})$ exits the region $\calR :=[0,1)\times [s_{-\infty},1]$ through the line $\{s=s_{-\infty}\}$. The existence of such $\bar{v}>0$ is ensured by \Cref{theo:order}. Similarly, let $\underline{v}\in (0,\bar{v})$ be such that $(M^{\underline{v}},S^{\ul{v}})$ exits $\calR$ through the line $\{m=0\}$. Then we use the following algorithm to determine $v>0$:

\begin{algorithm}[Bisection iteration to determine the travelling wave] $\!$\\[-1em]
\begin{enumerate}
\item Set $v=\frac{1}{2}|\bar{v}+\underline{v}|$. Solve $(M^v,S^v)$ satisfying \eqref{eq:GE-} and \eqref{eq:algoBC} numerically. 
\item If $(M^v,S^v)$ exits $\calR$ through $s=s_{-\infty}$, then set $\bar{v}=v$. Else, set $\underline{v}=v$.
\item If $|\bar{v}-\underline{v}|<10^{-4}\,v$ then stop; otherwise go to Step 1.
\end{enumerate}\label{algo:TW}
\end{algorithm}

For the default parameter set in \Cref{tab:default_parameters}, Algorithm \ref{algo:TW} is  over 1000 times faster than the PDE computation under the same computational set-up.

\subsubsection{Validation of the TW algorithm using the PDE scheme}\label{sec:validation}
\begin{figure}[h!]
\includegraphics[width=0.32\textwidth]{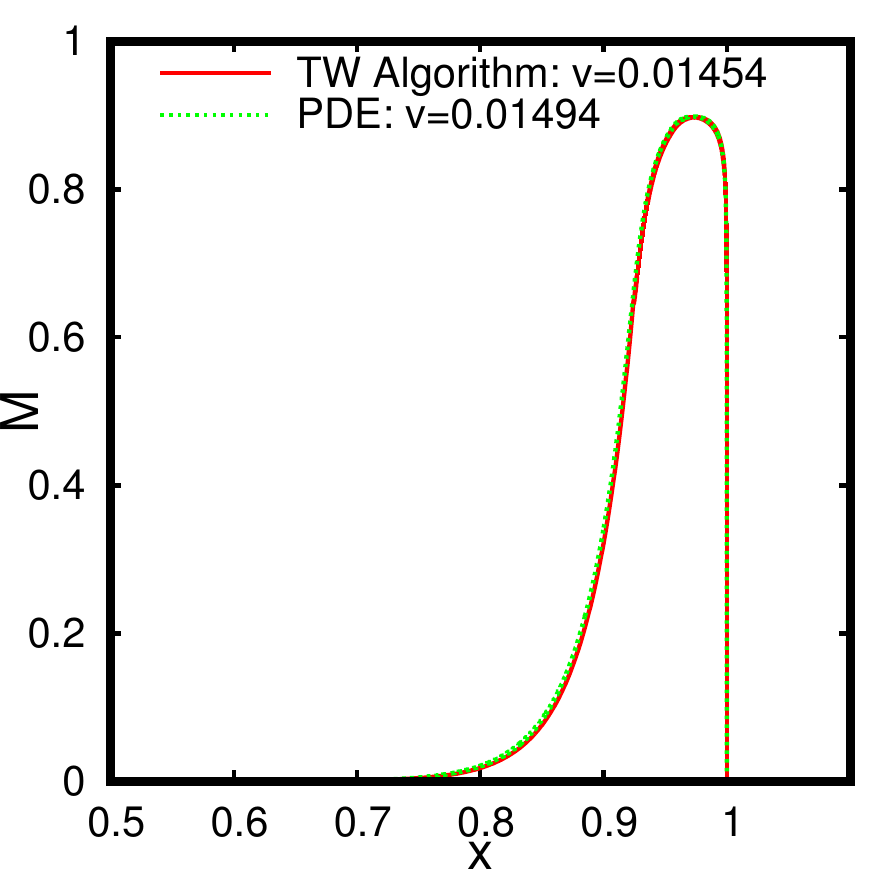}
 \includegraphics[width=0.32\textwidth]{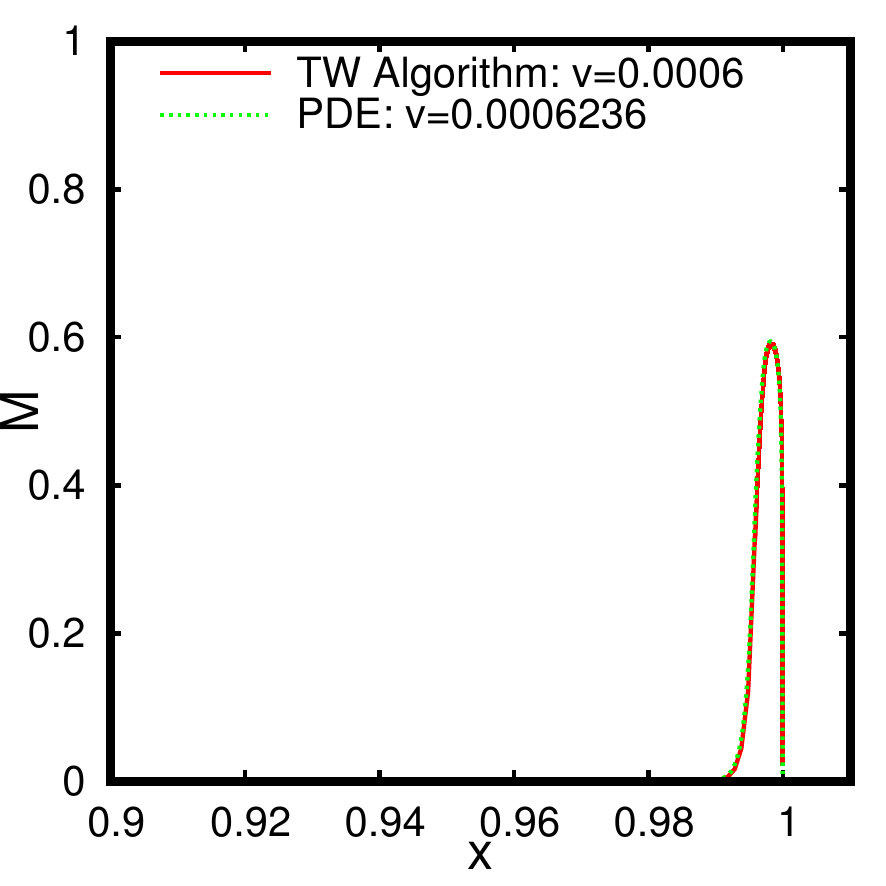}
 \includegraphics[width=0.32\textwidth]{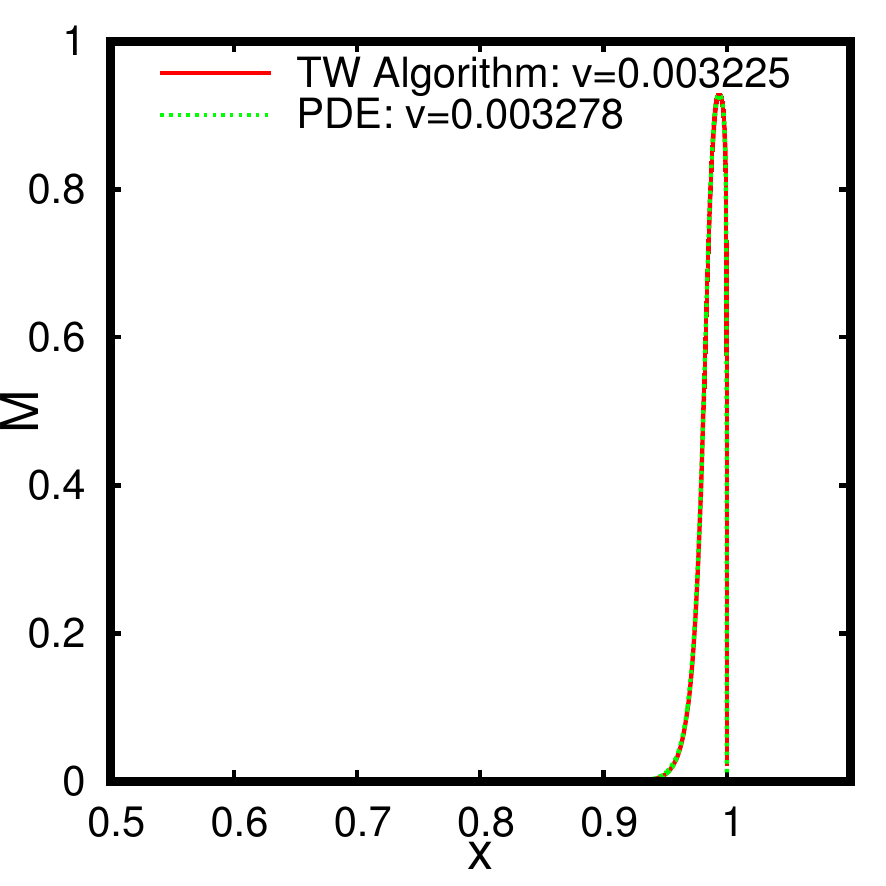} 
\caption{The comparison of the $M$-profiles obtained by using Algorithm \ref{algo:TW} with the  PDE simulations. The parameters are taken from \Cref{tab:default_parameters} except in (left) $\l=0.42$, (center) $\l=0.60$, and (right) $a=b=2$. }\label{fig:validation}
\end{figure}

We compare the profiles obtained from Algorithm \ref{algo:TW} with the profiles of the full PDE simulations. The same three test cases are chosen as in \Cref{fig:transient}. The results are shown in \Cref{fig:validation}. We consider the agreement to be excellent with the wave-speed varying only by $2.72\%$ in the default case of \Cref{tab:default_parameters} for $N=14$ in \eqref{eq:1D_NumPar}, and the profiles are barely distinguishable from each other for the two methods.

\subsubsection{Parametric study in the limiting cases}
\Cref{fig:validation} shows us that when $\lambda$ changes from $0.42$ to $0.6$, the wave-speed and the perceived width of the TW become about 20 times smaller. Hence,  PDE simulations become impractical both for large $\l$ values, due to very small mesh-sizes required, and for small $\l$ values, due to the increase in the  required domain size. Similar problems occur for the coefficients $a,\,b$ (see \Cref{fig:validation}) and the parameters $\g$ and $\k$. Since Algorithm \ref{algo:TW} is much faster compared to the PDE computations, these cases can be better  explored thoroughly using the ODE simulations.

An interesting case-study is the investigation of the conditions provided in \Cref{theo:main} for the existence of TWs. They are,  Condition 1: $\Gf(s)>0$ for all $s\in (s_{-\infty},1)$, which is also a necessary condition due to \Cref{pros:NonExistence}; and Condition 2: \eqref{eq:parameter} is satisfied. For $\g$ and $\k$ given in \Cref{tab:default_parameters}, Condition 1 is satisfied if $\l\geq 0.26$ whereas Condition 2 is satisfied if $\l\leq 0.56$.  Here we are interested in exploring the limits of $\l$ for which the TWs exist. We already saw from \Cref{fig:transient,fig:validation} that a TW solution exists for $\l=0.6$, thus indicating that Condition 2 is not a necessary condition. Using the TW algorithm, we also find TW solutions up to $\l=0.8$. The profile is much narrower in this case and has a minuscule $v=0.000033$, see \Cref{fig:big_lambda} (left).  However, we were unable to find any TW solutions for $\l\geq 0.9$ which is still smaller than the absolute limit of $\l=1$ for which it is guaranteed that TW solutions do not exist, see \Cref{pros:NonExistence}.  \Cref{theo:main} does not guarantee the existence of TW solutions in these cases since the nullcline $m=\ell(s)$ does not intersect the line $\{m=1\}$, see \Cref{fig:big_lambda} (right).
\begin{figure}[h!]
\begin{subfigure}{.5\textwidth}
\includegraphics[trim=0 0 0 1cm,clip, scale=.35]{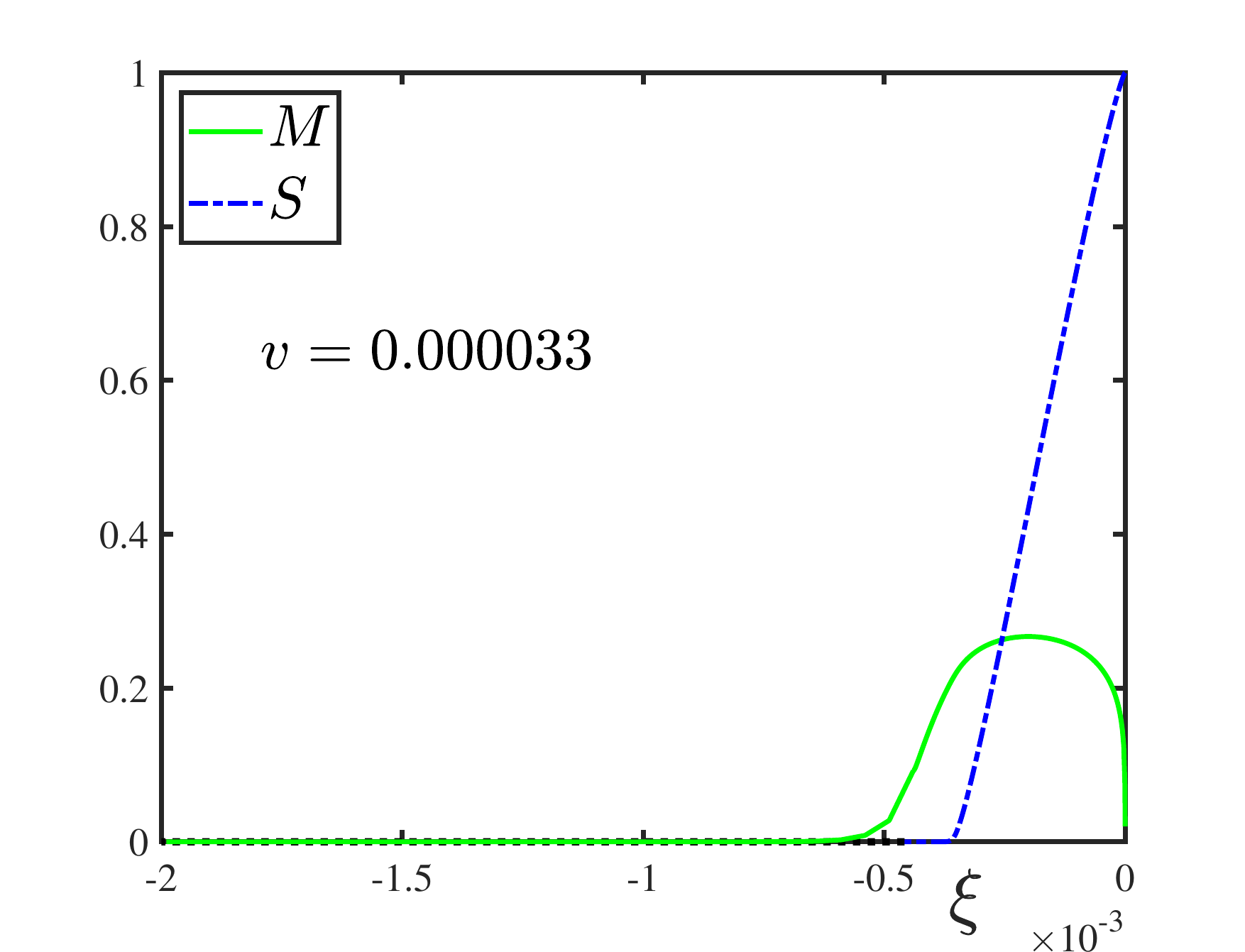}
\end{subfigure}
\begin{subfigure}{.5\textwidth}
 \includegraphics[trim=0 0 0 1cm,clip,scale=.35]{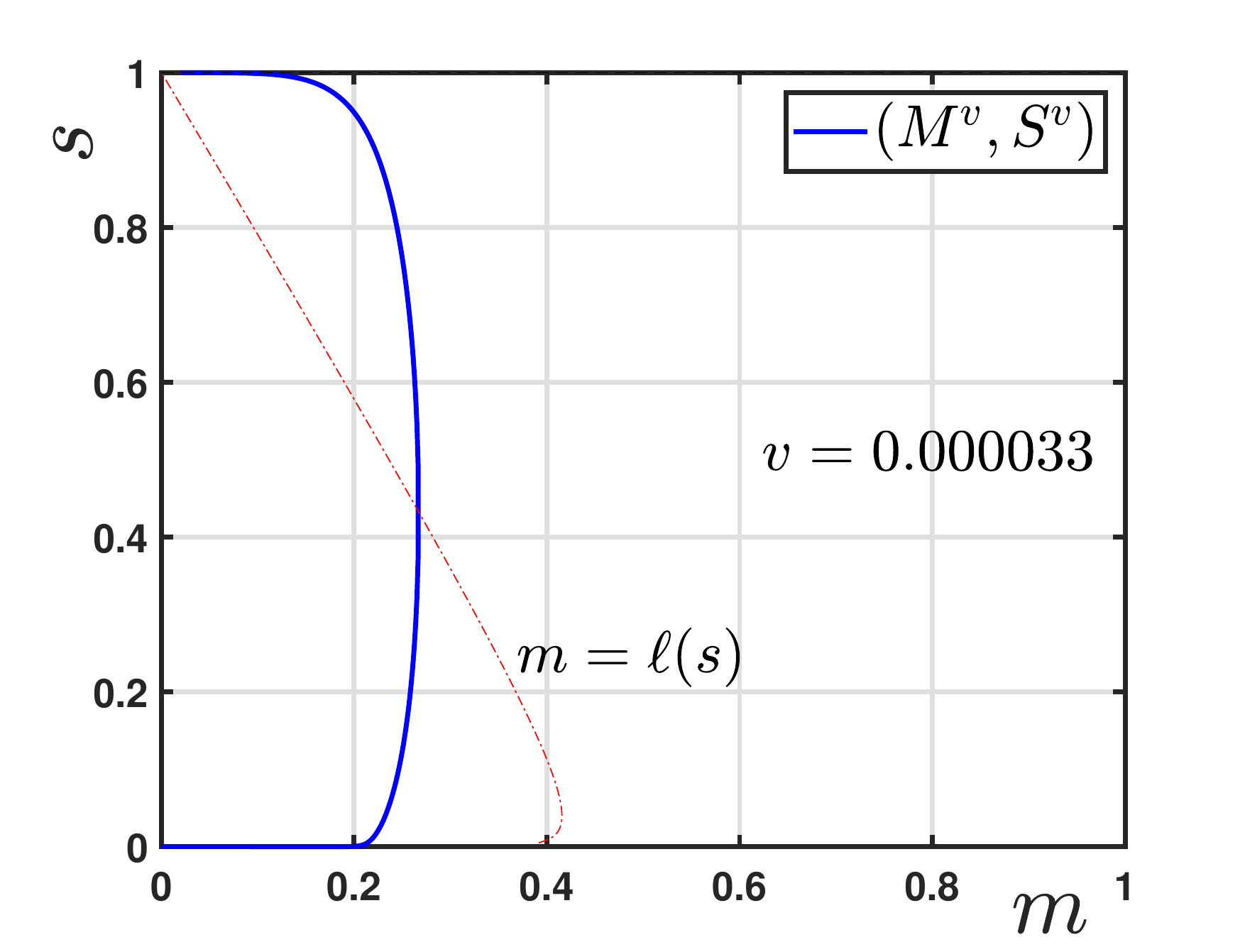}
\end{subfigure}
\caption{(left) The plots of $M^v$ and $S^v$ vs. $\xi=x-vt$ computed using Algorithm \ref{algo:TW} for $\l=0.8$. (right) The orbit $(M^v,S^v)$  in the $m$--$s$ phase--plane.}\label{fig:big_lambda}
\end{figure}

On the other hand, if $\l$ is chosen smaller than $\l=0.42$, then we see a rapid increase in wave-speed and a widening of the profile, see \Cref{fig:small_lambda}. In fact, varying between $\l=0.38$ and $\l=0.36$ the wave-speed and the profile width increase ten fold. As such, the computational time required for the algorithm to converge increases exponentially, and we were unable to obtain TW solutions for values of $\l$  smaller than $0.36$. However, the trend in behaviour when varying $\l$ is  evident from the simulations. 

\begin{figure}[h!]
\begin{subfigure}{.5\textwidth}
\includegraphics[trim=0 0 0 1cm,clip, scale=.35]{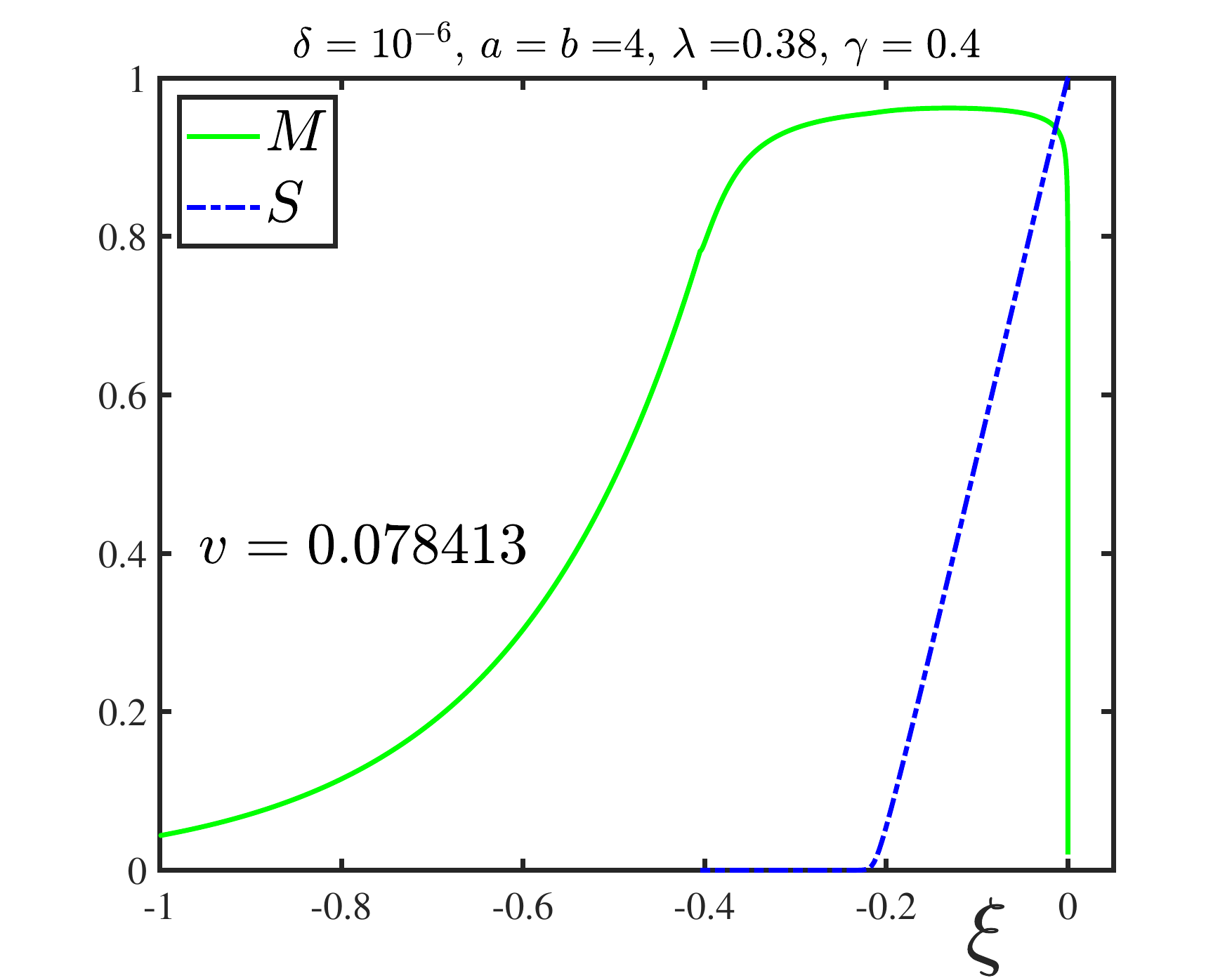}
\end{subfigure}
\begin{subfigure}{.5\textwidth}
 \includegraphics[trim=0 0 0 1cm,clip,scale=.35]{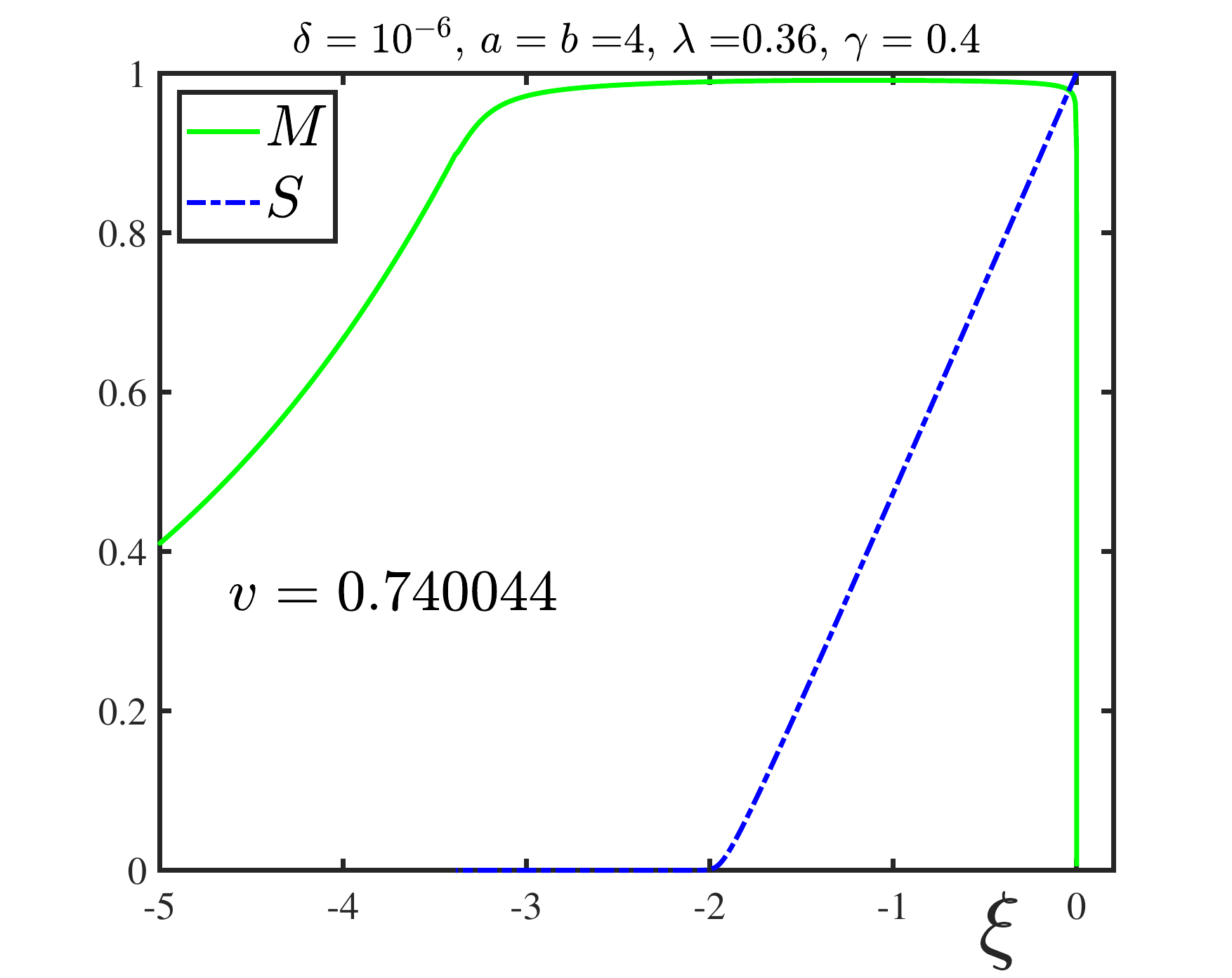}
\end{subfigure}
\caption{The plots of $M^v$ and $S^v$ vs. $\xi=x-vt$ computed using Algorithm \ref{algo:TW}. For the (left) plot $\l=0.38$, and the (right) plot $\l=0.36$.}\label{fig:small_lambda}
\end{figure}

\subsection{Numerical continuation method}\label{sec:NumCont}
Another approach to study parameter regimes for which the system  
(\ref{eq:GE}) has a connecting orbit is to use numerical continuation techniques \cite{doedel1989}. 
The approach is to find a connecting orbit on the center unstable manifold of the rest state
$(0,s_{-\infty})$ to the strong stable manifold of the rest state $(0,1)$.

Following \cite{doedel1989},  we set this up as a two point boundary value problem:
\begin{subequations}\label{eq:BVP}
\begin{align}
&\dt M= v\,T\, [\ell(S)- M], &\mbox{~~~for $\t \in (0,1)$},
\label{eq:BVPm}\\[.5em]
& \dt S=  \tfrac{\g}{v}\,T\, \,f(S)\, M D(M), &\mbox{~~~for $\t \in (0,1)$}
\label{eq:BVPs},\\[.5em]
&(M(0),S(0))= (\e_0, s_{\e_0}),&\label{eq:BVP0} \\[.5em]
& (M(1),S(1)= (\e_1,1)\label{eq:BVP1},& 
\end{align}
\end{subequations}
where $s_{\e_0}$ is the smallest root of the equation
\begin{equation}
  \ell(s)=\e_0, \label{eq:appcum}
\end{equation}
i.e.,  $\ell(s_{\e_0}) = \e_0$.
Note that (\ref{eq:BVP1}) is a point on the linearized strong stable manifold for the rest state
$(0,1)$. The center unstable manifold of the rest state $(0,s_{-\infty})$ is well approximated
by $(\ell(s),s)$ for $s$ near $s_{-\infty}$ since the 
homological equation for the center manifold is given by
\[ M\, D(M)\, \dfrac{\dd M}{\dd S} = \frac{v^2}{\g f(S)}\,[\ell(S)- M], \]
and $M D(M) ={ \cal O}(M^{1+a})$ for $M$ near zero.
This is why we use the boundary condition (\ref{eq:BVP0}). The parameter $T$ is the time of travel and is taken to be very large in order to approximate the heteroclinic
orbit. An integral condition (see, e.g. \cite{doedel1989,doedel91}) is added to (\ref{eq:BVP}) to
facilitate adaptive mesh selection when computing a branch of solutions.
We use AUTO-07P \cite{auto} for these computations.

To further illustrate the conclusions of Theorem~\ref{theo:main}, and to investigate the limiting cases where the condition $\Gf(s)>0$ for all $s\in (s_{-\infty},1)$ is violated, we compute the heteroclinic connections using the parameter
values $a=b=2$, $\kappa=1$,  and $\lambda = 0.3$. With these values of $\kappa$ and $\lambda$, we find numerically that $s_{-\infty}\approx 0.1319$ and    
$\Gf(s)>0$ for all $s\in (s_{-\infty},1)$ provided $\gamma > \approx 0.1093$. Hence, TWs cannot exist for  $
\g$ below this threshold. This is exactly what is observed numerically from 
\Cref{fig:autores} (left) where the wave-speed $v$ is plotted against corresponding $\g$ values. The horizontal asymptote shows that no solution is possible below a certain threshold of $\g$ close to the predicted value. In the (right) figure we see how the orbits vary in the phase--plane when $\g$ is decreased, tending towards the line $\{m=1\}$ uniformly. This is similar to what was observed in the case when $\l$ was lowered, see \Cref{fig:small_lambda}.

\begin{figure}[h!]
\begin{subfigure}{.5\textwidth}
\includegraphics[scale=.4]{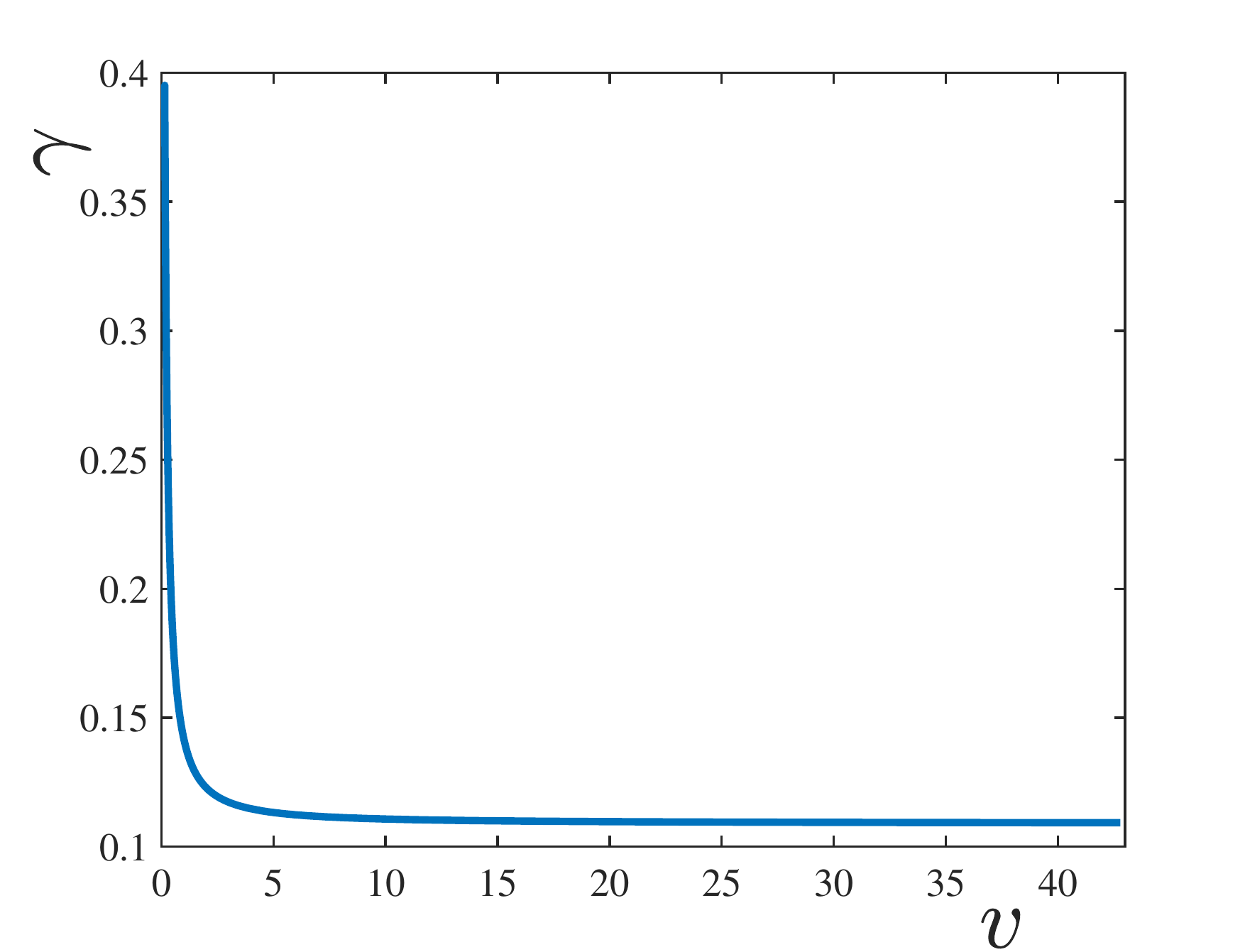}
\end{subfigure}
\begin{subfigure}{.5\textwidth}
 \includegraphics[scale=.4]{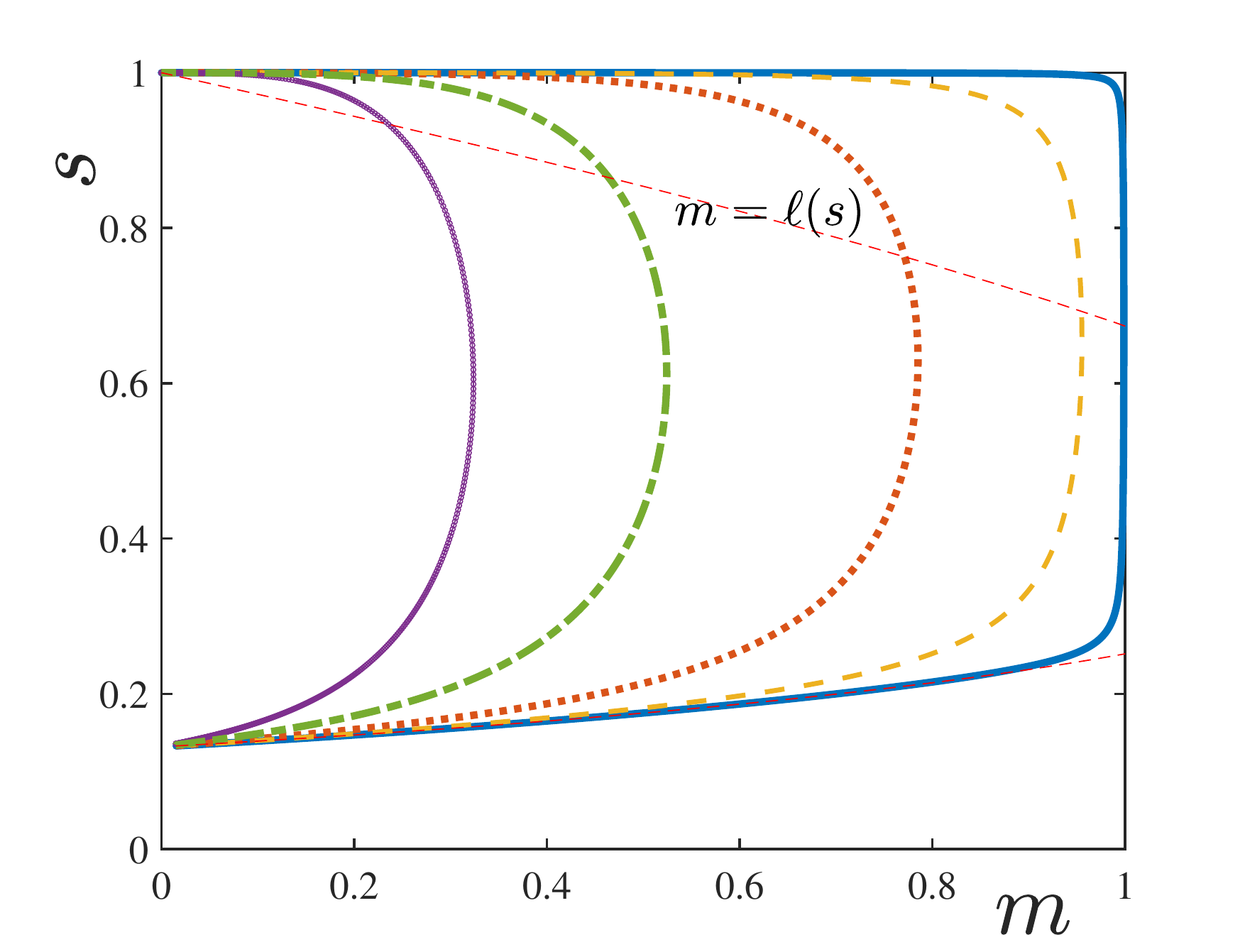}
\end{subfigure}
\caption{The results for the numerical continuation method. The (left) plot is the bifurcation diagram of $\gamma$ versus the wave speed $v$. In the (right) plot are the corresponding solutions in the phase--plane.}\label{fig:autores}
\end{figure}
In \Cref{fig:ppnullcline} (left) we see a representative solution in the phase--plane
along with the nullcline, $m=\ell(s)$. The effect of using the phase-condition as described
in \cite{doedel91} is illustrated in \Cref{fig:ppnullcline} (right) where solutions at two different parameter values are depicted.

\begin{figure}[h!]
\begin{subfigure}{.4\textwidth}
\includegraphics[width=\textwidth]{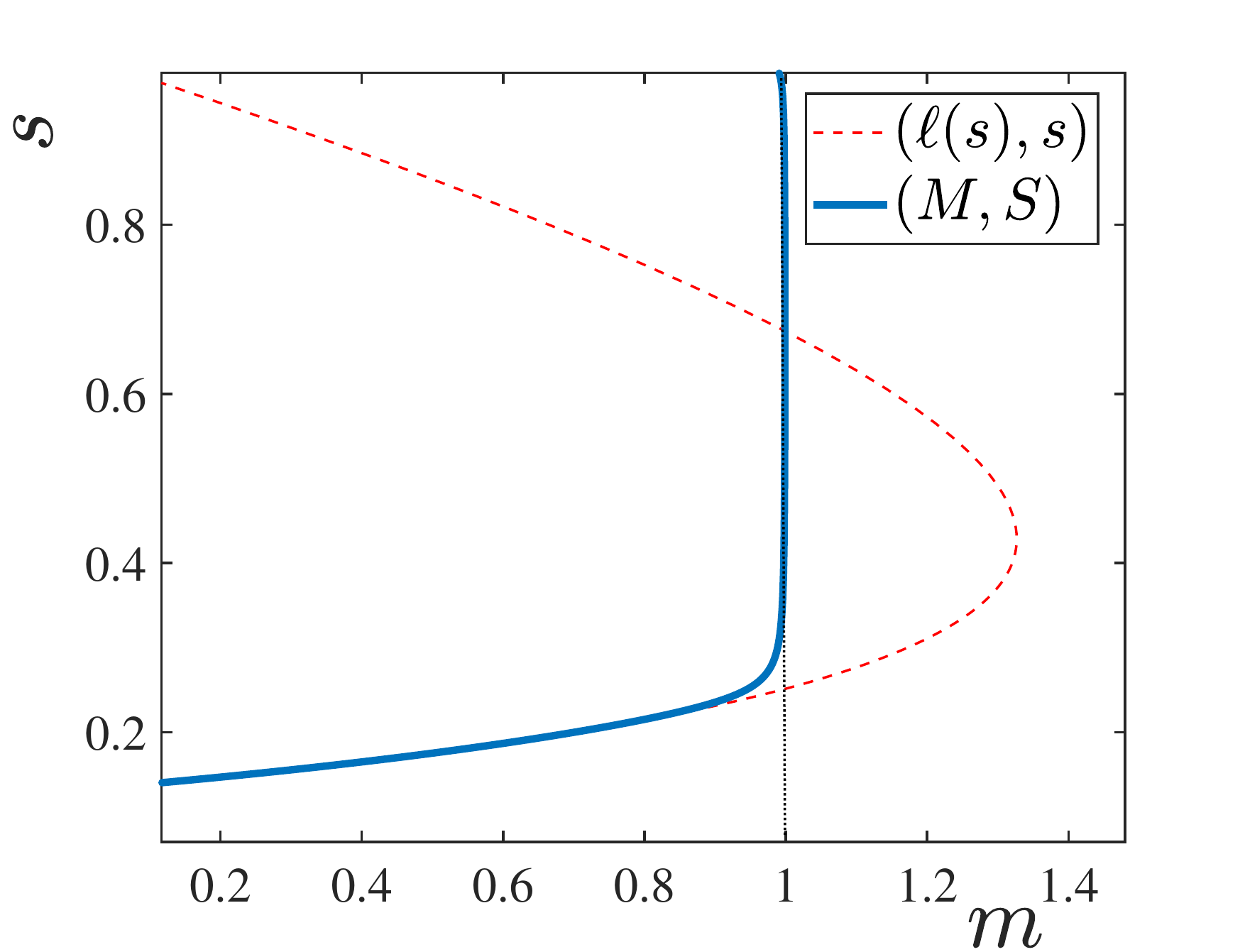}
\end{subfigure}
\begin{subfigure}{.6\textwidth}
\includegraphics[width=\textwidth]{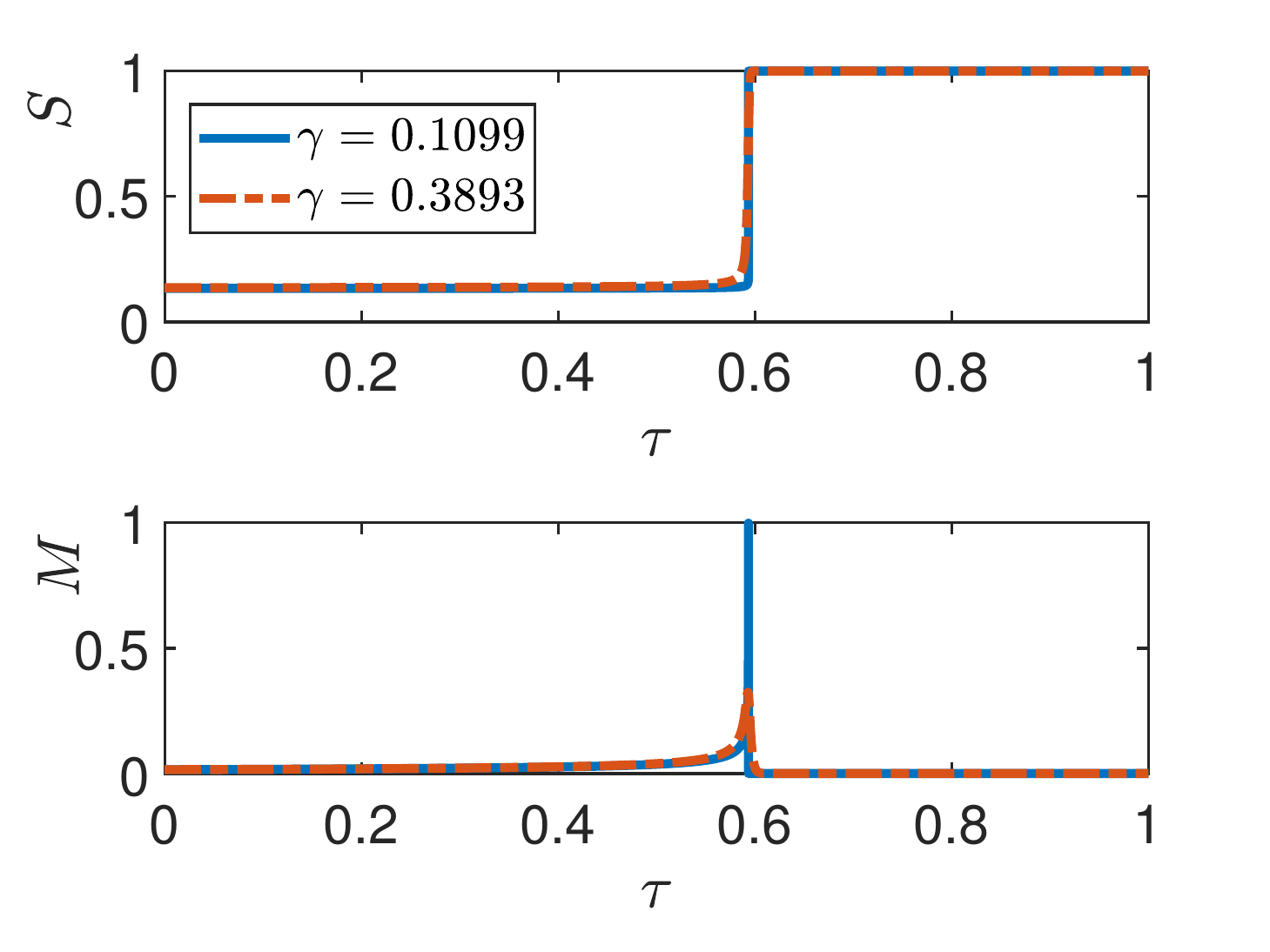}
\end{subfigure}
\caption{(left) Plot of one solution (blue) in the plase--plane along with the nullcline $m=\ell(s)$ (red). Here $\gamma =0.1099$ and $v=16.227$. (right) Plot of two different solutions as functions of $\tau$, one for $\gamma=0.3893$ (red) and one for $\gamma = 0.1099$ (blue).}
\label{fig:ppnullcline}
\end{figure}

\section{Interpretation of the results}\label{sec:interpretation}
The analysis in \Cref{sec:TW existence} shows that TWs for the system \eqref{eq:PDE} exist as expected from the numerical experiments in \cite{eberl2017spatially} and  observations on cellulolytic biofilms made in \cite{wang2011spatial}. Moreover, we predicted theoretically and verified numerically that they exist for a wide range of parameters.  The resulting TWs are found to be stable even for large perturbations (see \Cref{fig:transient,fig:2D_stability}) which agrees with the numerical observations in \cite{eberl2017spatially,hughes2022RODE}. 
In the context of biofilm growth, the TWs describe the formation of crater like structures (inverted colonies), i.e. the invasion and degradation of the undisturbed cellulosic environment by microbes. At the leading edge of this invading front is a microbially active layer that consumes the substrate. In the wake of  this layer, decay terms dominate over growth terms, and thus, this region is dominated by the dynamics of how fast substrates  degrade and how fast bacteria decay if growth cannot be sustained. This complex interplay gives the TW its distinct shape.

Our analysis also suggests that the TWs do not exist if for example, either the cell-loss rate $\l\in (0,1)$ is too small or too large (\Cref{pros:NonExistence}). From the numerical experiments in \Cref{sec:ODEscheme}, we can postulate the reason behind the non-existence of TWs. We see from \Cref{fig:small_lambda} that as  $\l$ decreases, the TW profile becomes wider very quickly, and the biofilm concentration $M$ approaches a value close to 1 in a large interval. This is expected since small $\l$ implies less decay of $M$, and for very small $\l$ we expect that the biofilm concentration would grow monotonically with time and reach 1 in every point of the domain, implying that a TW solution cannot exist. On the other hand, for $\l$  large, the profiles become narrower and their amplitudes decrease as seen from \Cref{fig:big_lambda}. Hence, one expects that for $\l$ large enough, the initial biofilm profile would decay to 0 monotonically with time.

The variation of the wave-speed can also be explained through these observations. If $\l$ is small, then the profile is wider. Hence, the bacteria consume the substrate faster, which results in a higher wave-speed. For $\l$ large the effect is reversed. 

Similarly, the effects of the consumption rate $\g>0$ on the existence, profile-width and wave-speed of the TW can be explained, see \Cref{fig:autores}. Higher values of $\g$ result in a faster consumption of the substrate, which leads to a decrease in the production of biomass. Hence, the effects of increasing $\g$ are analogous to the effects of increasing $\l$.

Lastly, we remark that our approach can also be applied to PDE--ODE systems with a porous media type diffusion coefficient (i.e. $D(M)=M^a$), although boundedness of $M\in [0,1)$ cannot be expected in this case. Furthermore, we expect that the results in this paper  can be extended to study TW solutions of PDE--ODE systems with multiple substrates, see the system in  \cite{efendiev2020mathematical}  for instance. Due to the structure of the TW,  as a pulse with sharp front and diffusive tail travelling at a constant speed, systems of the form \eqref{eq:PDE} can likely be used to model several other biological and physical processes with immobile substrates, such as tumor growth, fungal growth, and the spreading of wildfire.

\subsection*{Acknowledgements}
K. Mitra and S. Sonner would like to thank the Nederlandse Organisatie voor Wetenschappelijk Onderzoek (NWO) for their support through the grant OCENW.KLEIN.358. J. Hughes is supported by the Natural Sciences and Engineering Research Council (NSERC) of Canada. The data and models used in this paper are publicly available and have been properly referenced.


\bibliographystyle{plain}

\begin{appendices}
\section{Grid independence study for the PDE simulations}\label{sec:grid_refinement}
In this section, we study how the mesh-size $\D x$ affects the TW solution obtained numerically from solving the PDE. The focus is to investigate the convergence of the wave-speed and profile as $\Delta x\to0$.  Let $v_N$ denote the wave-speed, numerically measured using the scheme described in \Cref{sec:transient}, for the $2^N$-grid simulation. We quantify the relative differences between the wave-speeds across grid resolutions via the ratios
\begin{align} \label{eq:GRR}
    {\rm ReD}_N=\frac{|v_N-v_{16}|}{|v_{16}|}, \text{ and } {\rm GRR}_N=\frac{{\rm ReD}_{N-1}}{{\rm ReD}_N}.
\end{align} 
Here, the quantity ${\rm GRR}_N$ stands for the grid refinement ratio which gives a quantitative comparison between the successive relative differences. 
The wave-speeds for each simulation along with the relative difference calculations are given in Table \ref{tab:speeds}.

\begin{table}[ht!]
\centering
\begin{tabular}{|l|l|l|l|l|}
\hline
N &Grid Size & Wave-speed $v_N$ & ReD$_N$ & GRR$_N$\\ \hline
9 &$512$ & 0.019450 & 0.308242 &-- \\
10 &$1024$ & 0.017100 & 0.150125 & 2.05\\
11 &$2048$ & 0.015909 & 0.070031 & 2.14\\
12 &$4096$ & 0.015335 & 0.031434 & 2.23\\
13 &$8192$ & 0.015064 & 0.013201 & 2.38\\
14 &$16384$ & 0.014944 & 0.005165 & 2.56\\ \hline
\end{tabular}
\caption{Wave-speeds calculated from the various grid resolutions along with the relative difference calculations ReD$_N$ and GRR$_N$ from \eqref{eq:GRR}. The wave speed for the $2^{16}$ grid simulation is $v_{16}=0.014868$.} \label{tab:speeds}
\end{table}
\begin{figure}[h!]
\centering
\includegraphics[width=.48\textwidth]{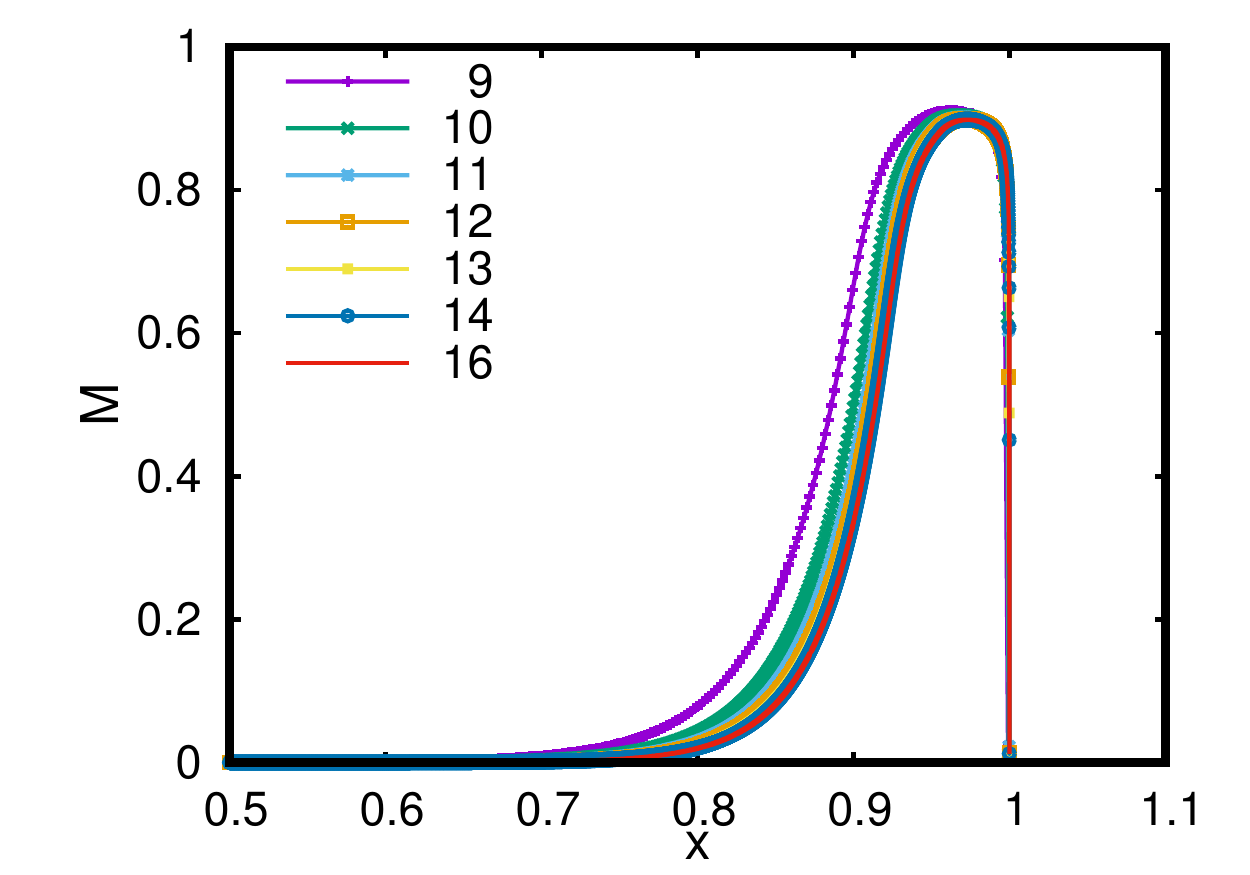} 
\includegraphics[width=.48\textwidth]{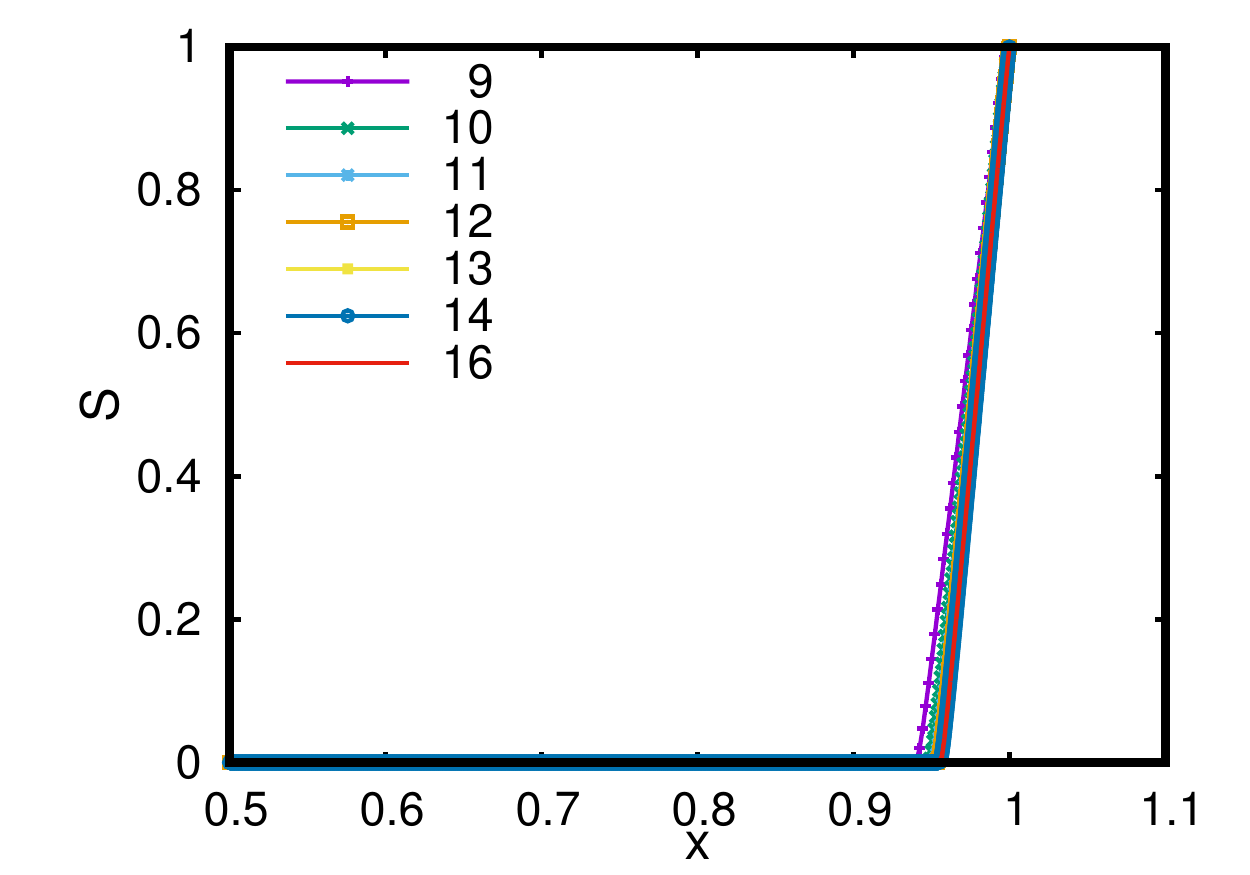}
\caption{Wave profiles of the various grid refinement simulations. The simulation parameters are given in Table \ref{tab:default_parameters}. } \label{fig:grid}
\end{figure}

It is seen from Table \ref{tab:speeds} that the wave-speeds decrease as the grid is refined and relative differences tend to zero. These observations indicate that we have convergence with respect to the wave-speed. Figure \ref{fig:grid} shows the convergence of the $M$ and $S$ profiles as the grid is made finer. A narrower overall wave profile and a steeper wave front is observed with finer grid resolutions. This is due to less numerical diffusion at smaller discretization levels.

The convergence of both the wave-speed and the wave profile guarantees that the TWs can indeed be reproduced using the numerical scheme described in \Cref{sec:PDEscheme}. Given the results in Table \ref{tab:speeds} and considerations on simulation runtime, we use a $2^{14}$-grid in our simulations and a time step of $\Delta t=10^{-2}(2^9(1/2^{14}))^2$ as presented in \eqref{eq:1D_NumPar}.

\end{appendices}
\end{document}